\documentclass[11pt]{amsart}

\usepackage{amssymb}
\usepackage{amsmath}
\usepackage{amsthm}
\usepackage{amscd}
\usepackage{bbm}
\usepackage{changepage}
\usepackage{enumerate}
\usepackage{fancybox}
\usepackage{fancyhdr}
\usepackage{float}
\usepackage{marvosym}
\usepackage{mathrsfs}
\usepackage[pdftex]{graphicx}
\usepackage{setspace}
\usepackage{subfig}
\usepackage{theoremref}
\usepackage{verbatim}
\usepackage{xy}
\usepackage{ amsmath, amsthm, amsfonts, amssymb, color}
\usepackage{amsfonts, amsmath}
\usepackage{amsmath,amstext,amsthm,amssymb,amsxtra}
 \usepackage[colorlinks, citecolor=blue,pagebackref,hypertexnames=false]{hyperref}

\def\XXint#1#2#3{{\setbox0=\hbox{$#1{#2#3}{\int}$}
     \vcenter{\hbox{$#2#3$}}\kern-.5\wd0}}

\hypersetup{
    colorlinks,
    citecolor=green,
    linkcolor=red
  }

\makeatletter
\def\blfootnote{\xdef\@thefnmark{}\@footnotetext}
\makeatother

\newcommand{\xx}{\times}

\newcommand{\R}{\mathbb{R}}
\newcommand{\N}{\mathbb{N}} 
\newcommand{\br}[1]{\left( #1 \right)}
\newcommand{\brs}[1]{\left[ #1 \right]}
\newcommand{\norm}[1]{\left\Vert #1 \right\Vert}
\newcommand{\abs}[1]{\left\vert #1 \right\vert}
\newcommand{\lb}[0]{\left\lbrace}
\newcommand{\rb}[0]{\right\rbrace}

\newtheorem{thm}{Theorem}[section]
\newtheorem{prop}{Proposition}[section]
\newtheorem{lem}{Lemma}[section]
\newtheorem{cor}{Corollary}[section]

\theoremstyle{definition}

\newtheorem{deff}{Definition}[section]
\newtheorem{rmk}{Remark}[section]

\newtheorem*{notation}{Notation}

\title[Square Functions on a Class of
Non-Doubling Manifolds]{Vertical and horizontal Square Functions on a Class of
	Non-Doubling Manifolds}
\author{Julian Bailey, Adam Sikora}
\address {Julian Bailey, Department of Mathematics, Macquarie University, NSW 2109, Australia}
\email{u4545137@alumni.anu.edu.au}

\address {Adam Sikora, Department of Mathematics, Macquarie University, NSW 2109, Australia}
\email{Adam.Sikora@mq.edu.au}


\date{}

\begin{document}

\begin{abstract}
We consider a class of non-doubling manifolds $\mathcal{M}$ that are the connected
sum of a finite number of $N$-dimensional manifolds of the form
$\R^{n_{i}} \xx \mathcal{M}_{i}$. Following on from the work of
Hassell and the second author \cite{hs2019}, a particular decomposition of the
resolvent operators $(\Delta + k^{2})^{-M}$, for $M \in \N^{*}$, will be
used to demonstrate that the vertical square function operator
$$
Sf(x) :=  \br{ \int^{\infty}_{0} \abs{t \nabla (I + t^{2}
  \Delta)^{-M}f(x)}^{2} \frac{dt}{t}}^{\frac{1}{2}}
$$
is bounded on $L^{p}(\mathcal{M})$ for $1 < p < n_{min} = \min_{i}
n_{i}$ and weak-type $(1,1)$. In addition, it will be proved that the reverse inequality
$\norm{f}_{p} \lesssim \norm{S f}_{p}$ holds for $p \in (n_{min}',n_{min})$ and
that $S$ is unbounded for $p \geq n_{min}$ provided $2 M <
n_{min}$.

 Similarly, for $M > 1$, 
this method of proof will also be used to ascertain that the horizontal square function operator
$$
sf(x) := \br{\int^{\infty}_{0} \abs{t^{2}\Delta (I + t^{2}
    \Delta)^{-M}f(x)}^{2} \, \frac{dt}{t}}^{\frac{1}{2}}
$$
is bounded on $L^{p}(\mathcal{M})$ for all $1 < p < \infty$ and weak-type $(1,1)$. Hence, for $p
\geq n_{min}$, the
vertical and horizontal square function operators are not equivalent
and their corresponding Hardy spaces $H^p$ do not coincide.
  \end{abstract}

  \maketitle

  \section{Introduction}


In any historical account of the development of harmonic analysis, the doubling
condition will certainly appear as a central actor. In each step of
its genesis, from the minds of many great mathematicians in the 1960's and 1970's, the doubling condition was interwoven, seemingly
inextricably, throughout the entirety of the body of work that embodied
harmonic analysis.
 In a metric measure space $(X,d,\mu)$,  the doubling
condition reads that there must exist a constant $C > 0$ such that
$$
\mu(B(x,2r)) \leq C \mu(B(x,r))
$$
for all $x \in X$ and $r > 0$, where the notation $B(x,r)$ is used to
denote the ball of radius $r$ and centered at the point $x$. If this
condition is satisfied then $(X,d,\mu)$ is said to be a
space of homogeneous type, in the sense of Coifman
and Weiss \cite{CoW}, whilst any space that does
not satisfy this condition is called non-homogeneous.

Although
this condition aligns with our real world intuition,  with the continued progress of mathematics as
a whole there are now many different situations and applications that depart from this idealised
world and demand the consideration
of non-homogeneous spaces. Moreover, it has
become increasingly apparent that the doubling condition is not quite
as critically indispensable for many harmonic analytic results as previously believed.
As such, and acting as a reversal to the assimilation of the doubling
condition in the adolescence of the field, there is now an intensive
research effort underway that aims to unthread the doubling 
strand from this body of work, where possible, and push the boundaries of harmonic
analysis beyond this condition.
 This
process must often be approached with the utmost care since one is
certain to encounter behaviours that depart very far from the
prototypical doubling example of Euclidean space
$\R^{d}$. Some examples of
notable results in this area include: the extension of
Calder\'{o}n-Zygmund theory to non-homogeneous spaces through the work of Nazarov, Treil and Volberg \cite{nazarov1998weak,nazarov1998cauchy,nazarov2003Tb} and Tolsa
\cite{tolsa2001proof,tolsa2001littlewood,tolsa2011calderon}; the generalisation of $BMO$ and $H^{1}$
theory by Bui and Duong \cite{bui2013hardy} and Tolsa \cite{tolsa2001BMO}; the consideration of weighted norm
inequalities through the work of Orobitg and P\'{e}rez \cite{orobitg2002Ap}; and
non-homogeneous Tb type theorems by Hyt\"{o}nen and Martikainen
\cite{hytonen2012nonhomogeneous}.

The non-doubling spaces that are of interest to
us in this article consist of a particular class of non-doubling manifolds formed as
connected sums.

\begin{deff} 
 \label{def:ConnectedSum} 
A manifold $\mathcal{V}$ is said to be formed as the connected sum of a
 finite number of complete and connected manifolds $\mathcal{V}_{1},
 \cdots, \mathcal{V}_{l}$ of the same dimension, denoted
 $$
\mathcal{V} = \mathcal{V}_{1} \# \cdots \# \mathcal{V}_{l},
 $$
 if there exists some compact subset with non-empty interior $K \subset \mathcal{V}$ for which
 $\mathcal{V} \setminus K$ can be expressed as the disjoint union
 of open subsets $E_{i} \subset \mathcal{V}$ for $i = 1, \cdots, l$,
 where each $E_{i}$ is isometric to $\mathcal{V}_{i} \setminus K_{i}$ for some
compact $K_{i} \subset \mathcal{V}_{i}$. 
\end{deff}

Fix dimension $N \in \N^{*} := \N \setminus \lb 0 \rb$, $l \geq 1$ and let $\R^{n_{i}} \xx \mathcal{M}_{i}$  for $i = 1, \cdots, l$
be a collection of manifolds with $\mathcal{M}_{i}$ compact and $n_{i}
+ \mathrm{dim} \, \mathcal{M}_{i} = N$. We will be interested in
smooth Riemannian manifolds $\mathcal{M}$ that are of the form
$$
\mathcal{M} := (\R^{n_{1}} \xx \mathcal{M}_{1}) \# \cdots \#
(\R^{n_{l}} \xx \mathcal{M}_{l}).
$$
As in the previous definition, it is possible to choose open subsets $E_{i}
\subset \mathcal{M}$ and compact $K \subset \mathcal{M}$ with
non-empty interior for which $\mathcal{M} \setminus K$ can be expressed as the
disjoint union of the $E_{i}$. The subsets $E_{i}$ are referred to as
the ends of $\mathcal{M}$, $K$ the center of $\mathcal{M}$ and the entire manifold $\mathcal{M}$
itself is fittingly said to be a manifold with ends.

The constituent manifolds $\R^{n_{i}} \xx \mathcal{M}_{i}$ for $1
\leq i \leq l$ each have topological dimension $N$, but asymptotic
dimension $n_{i}$ at infinity. That is, for any ball $B(x,r) \subset
\R^{n_{i}} \xx \mathcal{M}_{i}$, it will be true that the volume of
the ball will satisfy
$$
V(B(x,r)) \simeq \left\lbrace \begin{array}{c c} r^{N} & for \ r \leq
                                                         1, \\
                                r^{n_{i}} & for \ r > 1. \end{array} \right.
$$
If the values of $n_{i}$ differ, then the manifold with ends $\mathcal{M}$ will have varying
asymptotic dimension on these ends. This will lead to a
violation of the doubling condition.

In the landmark article by Grigoryan and Saloff-Coste
\cite{grigoryan}, the authors effectively computed, using
probabilistic methods, two-sided estimates
for the heat kernel generated by the Laplacian $\Delta$ on this prototypical collection of non-doubling
spaces. Although not the first to study this
remarkable class of
manifolds in detail (see \cite{Nix} for a detailed historical account), this article acted
as an inflection point for interest in this class and had a
pronounced effect on the non-homogeneous community. Indeed, it essentially designated this class of manifolds
as a battlefront for the advancement of non-homogeneous harmonic
analysis. In point of fact, the sustained interest in these model spaces has
led to investigations into the boundedness of the heat
maximal operators \cite{duong2013boundedness}, the functional calculus of
$\Delta$ \cite{bui2020functional} and
Littlewood-Paley decompositions \cite{bouclet2010littlewood}. Recently, and of particular significance
to our article, Hassell in
collaboration with the second named author considered the $L^{p}(\mathcal{M})$-boundedness
of the Riesz transforms operator $\nabla \Delta^{-\frac{1}{2}}$ on
such manifolds \cite{hs2019}. This paper in tern is a generalisation to the non-doubling setting  of the 
result obtained by Carron, Coulhon and Hassell in \cite{CarronCH}. For other relevant results we refer the reader to \cite{Ca2, Dev} and references therein.

 In this article,  it is our aim to extend the classical
 theory of square functions to this class of non-doubling
 manifolds. 
 Consider a general complete Riemannian manifold $\mathcal{M}$ and let
 $\Delta$  denote the Laplace-Beltrami operator for this manifold.
For $M \in \N^{*}$, the vertical square function operator is defined: 
\begin{equation}
  \label{eqtn:Vertical}
S f(x) := \br{ \int^{\infty}_{0} \abs{t \nabla (I + t^{2}
  \Delta)^{-M}f(x)}^{2} \frac{dt}{t}}^{\frac{1}{2}}.
\end{equation}
The notion of square functions forms an essential part
 of harmonic analysis and has numerous applications, from the
 definition of Hardy spaces \cite{can23} to providing an equivalent characterisation of the bounded holomorphic functional calculus of a
 sectorial operator \cite{cowling1996banach}. As such, the above
 operators have been extensively studied, either in this form or
defined with the semigroup replacing the higher-order resolvent (c.f. Remark \ref{rmk:Vertical}), since the formation of
harmonic analysis, and a great deal is
known about their behaviour when the manifold under consideration is
doubling.

\begin{enumerate}
\item[$\bullet$] For the classical case when $\mathcal{M}$ is simply
  Euclidean space $\R^{d}$, $S$ is bounded on $L^{p}(\R^{d})$ for any
  $p \in (1,\infty)$ and weak-type $(1,1)$.
  \item[$\bullet$] For general complete Riemannian manifolds, the work
    of Coulhon, Duong and Li \cite{coulhon2003littlewood} states that
    the semigroup variation of the vertical square function is bounded on $L^{p}(\mathcal{M})$
for all $p \in (1,2]$, whether doubling or not, and weak-type $(1,1)$
if the manifold is doubling and if Gaussian upper bounds are satisfied
by the heat kernel.

\item[$\bullet$] For doubling manifolds $\mathcal{M}$ whose heat
  kernel satisfies Gaussian upper bounds bounds, one can define
  $$
q_{+} := \sup \lb p \in (1,\infty) : \norm{\abs{\nabla
    \Delta^{-\frac{1}{2}} f}}_{p} \lesssim \norm{f}_{p} \rb.
$$
It is then known that $q_{+} \geq 2$ \cite{coulhon1999riesz} and that
the semigroup variation of $S$ is bounded on
$L^{p}(\mathcal{M})$ for any $p \in (1,q_{+})$. A sparse proof of this
was shown in {\cite[Prop.~3.8]{bailey2020quadratic}}.
  \end{enumerate}

Nothing is currently known about the $L^{p}(\mathcal{M})$-boundedness
of the vertical square function for $p > 2$ on non-doubling manifolds. 
Our aim in this article is to prove the following theorem.

\begin{thm} 
 \label{thm:Main} 
Let $\mathcal{M} = (\R^{n_{1}} \xx \mathcal{M}_{1}) \# \cdots \#
(\R^{n_{l}} \xx \mathcal{M}_{l})$ be a manifold with ends. For $M \in
\N^{*}$, the vertical
square function operator $S$, as defined in \eqref{eqtn:Vertical},
will satisfy the following properties:

\begin{enumerate}
\item[(i)] $S$ is bounded on $L^{p}(\mathcal{M})$ for all $p \in
  (1,n_{min})$, where $n_{min} := \min_{i} \, n_{i}$, and weak-type $(1,1)$;
  
  \item[(ii)]  If $2M < n_{min}$ then $S$ is unbounded on $L^{p}(\mathcal{M})$ for $p \geq
    n_{min}$; and
    
    \item[(iii)] For $p \in (n_{min}',n_{min})$, there
exists $c > 0$ for which
$$
\norm{f}_{p} \leq c \norm{S f}_{p}
$$
for all $f \in L^{p}(\mathcal{M})$.
  \end{enumerate}
 \end{thm}

Let us briefly discuss the proof of the $L^{p}(\mathcal{M})$-boundedness portion of this result. 
Notice that through a change of variables $S$ has the representation
\begin{align*}\begin{split}  
 Sf(x) &= \br{ \int^{\infty}_{0} \abs{\nabla (t^{-2} + 
     \Delta)^{-M}f(x)}^{2} t^{1 -  4 M} \, dt}^{\frac{1}{2}} \\
 &= \br{\int^{\infty}_{0} \abs{\nabla (k^{2} + \Delta)^{-M} f(x)}^{2}
   k^{4 M - 3} \, dk}^{\frac{1}{2}}.
\end{split}\end{align*}
This can be controlled from above by the corresponding high and
low energy parts of this square function,
$$
Sf(x) \leq S_{>}f(x) + S_{<} f(x),
$$
where
$$
S_{>}f(x) := \br{\int^{\infty}_{1} \abs{\nabla (k^{2} + \Delta)^{-M} f(x)}^{2}
   k^{4 M - 3} \, dk}^{\frac{1}{2}}
 $$
 and
 $$
S_{<}f(x) := \br{\int^{1}_{0} \abs{\nabla (k^{2} + \Delta)^{-M} f(x)}^{2}
   k^{4 M - 3} \, dk}^{\frac{1}{2}}.
 $$
 The $L^{p}(\mathcal{M})$-boundedness of these two parts will be proved
 separately. The high energy component is local in nature and
 does not see the large scale non-doubling character of the
 manifold. It can thus be treated in a manner analogous to the classical
 Calder\'{o}n-Zygmund case. This will be accomplished in Section \ref{sec:High}.

 The low energy component, on the other
hand, will prove to be
more challenging since it involves extensive interaction between the
different ends of the manifold. Indeed, ultimately it is the low energy
component that will prove to be solely responsible for the unboundedness of the
operator on the range $p \geq n_{min}$. Another major difficulty that
appears is that, in contrast to many classical cases, it will not be
possible to treat this term using estimates
for the spatial derivative of the heat kernel $\nabla e^{-t
  \Delta}$. This is due to the pronounced abscence of these estimates in the literature.
Instead, we rely on a decomposition of the resolvent operator
 $(\Delta +
k^{2})^{-1}$ that was constructed in \cite{hs2019} using a parametrix
style argument. In Section \ref{sec:Resolvent} it will be proved that this decomposition can be
generalised to the higher-order resolvent operators $(\Delta +
k^{2})^{-M}$. Following this, in Section \ref{sec:Low}, this higher-order resolvent decomposition will be
applied to the low energy component and boundedness will ensue.

For the unboundedness portion of Theorem \ref{thm:Main}, observe that
a constraint on the order of the resolvent is required, $2M <
n_{min}$. This constraint is a consequence of our method of
proof. Indeed, our proof that $S$ is
unbounded on $L^{p}(\mathcal{M})$ for $p \geq n_{min}$ relies
heavily on the use of the Riesz potential operators
$\Delta^{-M}$. In an analogous manner to classical theory on Euclidean
space, these operators are not well-defined when the order $M$ is too
large in comparison to the dimension, namely $2 M \geq n_{min}$. As
a result, our method of proof will not be applicable for this range.
We do not study the range $2 M \geq n_{min}$ here. Therefore, we do not study the case $n_{min}=2$ which was investigated in \cite{hns2019, Nix}.

\begin{rmk}
  \label{rmk:Vertical}
  Through a careful study of the literature, one will notice
  that the
  term vertical square function is often used to refer to the operator
  $$
\br{\int^{\infty}_{0} \abs{t \nabla e^{-t \Delta}f(x)}^{2} \frac{dt}{t}}^{\frac{1}{2}},
$$
with the semigroup $e^{-t \Delta}$ taking on the role held by the
higher-order resolvent operator for $S$. Although the semigroup form
is  more
frequently encountered, our consideration of square functions defined using higher-order
resolvent operators is by no means rare. For instance, in the article
\cite{Frey2018} by
Frey, McIntosh and Portal, resolvent based conical square function estimates were proved for perturbations of Dirac-type
operators on $L^{p}(\R^{d})$.

Indeed, the two different forms of square
function are seen to be morally equivalent. For if one can prove that
$S$ is bounded on $L^{p}(\mathcal{M})$, together with a similar estimate for
$\nabla \mathrm{div}$, then one can obtain the boundedness of
the holomorphic functional calculus of the corresponding Dirac-type
operator on $L^{p}(\mathcal{M})$ (c.f. {\cite[Cor.~6.8]{cowling1996banach}}). The boundedness of the semigroup square function
would then
follow immediately as a corollary. This argument can also be reversed
in order to obtain the $L^{p}(\mathcal{M})$-boundedness of the resolvent square function from
the $L^{p}(\mathcal{M})$-boundedness of the semigroup square function.
  \end{rmk}

 As an alternative to the vertical square function $S$, one can also
 consider the horizontal square function for $M > 1$:
 \begin{equation}
   \label{eqtn:Horizontal}
s f(x) := \br{ \int^{\infty}_{0} \abs{t^{2} \Delta (I + t^{2}
  \Delta)^{-M}f(x)}^{2} \frac{dt}{t}}^{\frac{1}{2}}.
\end{equation}
For this operator, we will prove the following theorem.

\begin{thm} 
  \label{thm:Main2}
Let $\mathcal{M} = (\R^{n_{1}} \xx \mathcal{M}_{1}) \# \cdots \#
(\R^{n_{l}} \xx \mathcal{M}_{l})$ be a manifold with ends and fix $M >
1$.
 For any $p \in (1,\infty)$ the square function operator $s$, as
 defined in \eqref{eqtn:Horizontal},
 is bounded on $L^{p}(\mathcal{M})$ and there exists $c, \, C > 0$
 such that for all $f \in L^{p}(\mathcal{M})$
 $$
c \norm{f}_{p} \leq \norm{s f}_{p} \leq C \norm{f}_{p}.
 $$
 In addition, the operator $s$ is of weak-type $(1,1)$. 
\end{thm}

This result will be proved in Section \ref{sec:Laplacian} and, similar
to the operator $S$, it will be achieved by decomposing $s$ into low
and high energy components and then by proving $L^{p}(\mathcal{M})$-boundedness for both
of these components separately. For the low energy component, we will
once again make extensive use of the decomposition for the
higher-order resolvent operators $(\Delta + k^{2})^{-M}$.

The boundedness of $s$ on $L^{p}(\mathcal{M})$ for $1 < p < \infty$, as stated in Theorem
\ref{thm:Main2}, is a result that already exists in the
literature. Indeed, the $L^{p}(\mathcal{M})$-boundedness of the semigroup based square
function,
$$
\br{\int^{\infty}_{0} \abs{t \Delta e^{-t \Delta} f(x)}^{2} \frac{dt}{t}}^{\frac{1}{2}},
$$
was proved to hold in the general symmetric Markov semigroup
setting (see {\cite[pg.~111]{stein1970topics}}). This implies that $\Delta$ possesses a bounded
$H^{\infty}(S^{o}_{\mu})$-functional calculus on $L^{p}(\mathcal{M})$
for any $\mu \in [0,\pi)$, which
immediately leads to the
boundedness of $s$ on $L^{p}(\mathcal{M})$
(c.f. {\cite[Cor.~6.8]{cowling1996banach}}).
 A new proof of this result has been included here to
illustrate the applicability of our methods to resolvent based
operators and to obtain the weak-type $(1,1)$ bounds for $s$
which, to the best of our knowledge, is a result that is new. It is
also particularly illuminating to compare the proofs for the vertical
and horizontal square functions in order to glean some intuition as to
why one is bounded on the
full reflexive range, while the other fails for $p \geq n_{min}$.

\begin{rmk}
Our results can also be viewed through the lens of Hardy spaces. In an analogous manner to the classical case $\mathcal{M} = \R^{d}$,
one can define, for $p > 0$, Hardy spaces $H_{\nabla}^{p}$ and
$H_{\Delta}^{p}$ associated with $s$
and $S$ through the norms
$$
\norm{f}_{H^{p}_{\nabla}} := \norm{S f}_{p} \quad and
\quad \norm{f}_{H^{p}_{\Delta}} := \norm{s f}_{p}.
$$
Theorem \ref{thm:Main} tells us that for $p \in (n_{min}',n_{min})$ it will be true that
$H^{p}_{\nabla} = L^{p}$, but
$H^{p}_{\nabla} \neq L^{p}$ for $p \in [n_{min},\infty)$ when $2 M < n_{min}$. In contrast,
$H^{p}_{\Delta} = L^{p}$ for all $p \in (1,\infty)$, and thus the two
square functions define distinct Hardy spaces for $p \geq n_{min}$. It
remains an open problem to check that $H^{p}_{\nabla} =
H^{p}_{\Delta}$ for $p \in (1,n_{min}']$.
\end{rmk}

\begin{rmk}
Comparing Theorem \ref{thm:Main} with the main result of \cite{hs2019}
we note a posteriori that in the considered setting  boundedness of the Riesz Transform and 
vertical square function are equivalent for any $L^p$ space, including weak type $(1,1)$ estimates. It is an interesting question whether such equivalence 
could be verified in some more general setting in the form of the abstract statement. It seems that the implication from Riesz Transform bounds to the square functions estimates can be approach 
using \cite[Theorem 9.5.1 Section 9]{hytonenNeervenWeis}. We do not know how to approach the opposite implication. We expect any result which includes  weak type $(1,1)$ estimates in any directions to be especially challenging. Here we are more interested with understanding the difference 
between horizontal and vertical square functions and we do not attempt to answer this question.  
\end{rmk}

\section{Preliminaries}
\label{sec:Preliminaries}

 Throughout this article, we fix a manifold with ends
 $\mathcal{M} = (\R^{n_{1}} \xx \mathcal{M}_{1}) \# \cdots \#
 (\R^{n_{l}} \xx \mathcal{M}_{l})$. The notation $K$
will be used to denote the center of $\mathcal{M}$ and the sets
$K_{i}$ and $E_{i}$ will be as given in Definition
\ref{def:ConnectedSum}. In particular, these sets will have the
property that $\mathcal{M}
\setminus K$ can be expressed as
the disjoint union of the ends $E_{i}$ and $E_{i} \simeq \R^{n_{i}}
\xx \mathcal{M}_{i} \setminus K_{i}$. In this manner, the
ends $E_{i}$ will be identified with the sets $\R^{n_{i}} \xx
\mathcal{M}_{i} \setminus K_{i}$ so that a point $x \in E_{i}$ can
also be viewed as belonging to the space $\R^{n_{i}} \xx
\mathcal{M}_{i} \setminus K_{i}$.

\begin{notation}
For estimates concerning two quantities $a, \, b \in \R$, the notation
$a \lesssim b$ will be employed to signify the existence
of a constant $C > 0$ such
that $a \leq C \cdot b$.  Similarly, $a \simeq b$ will denote that $a
\lesssim b$ and $b \lesssim a$ both hold. The dependence of the
constant $C$ on certain parameters should be clear from the context of
the argument under consideration.

For $x \in \R^{d}$, define $\langle x \rangle := (1 +
\abs{x}^{2})^{\frac{1}{2}}$. We employ the notation
$d(x,y)$ to denote the intrinsic distance between two
points $x$ and $y$ in some ambient Riemannian manifold.  When the space under consideration is the
entire space $\mathcal{M}$, we use the shorthand notation $L^{p}$
to denote the Lebesgue space $L^{p}(\mathcal{M})$.  Finally, for a
function $g(x,y)$ of two variables, the notation $\nabla_{x} g(x,y)$ will
be understood to denote the gradient with respect to the first variable.
  \end{notation}

Before delving into the substance of our proof, it will first be
beneficial to record some useful estimates satisfied by the
higher-order resolvent operators on the constituent manifolds $\R^{n_{i}} \xx
\mathcal{M}_{i}$. Following this, we recall a vital decomposition
of the first-order resolvent on the entire space $\mathcal{M}$ that
was introduced in \cite{hs2019}. This will form the foundation upon which
our proofs of both Theorem \ref{thm:Main} and \ref{thm:Main2} will be built.

\subsection{Higher-Order Resolvents on $\R^{n_{i}} \xx \mathcal{M}_{i}$}

Fix $1 \leq i \leq l$ and consider the
higher-order resolvents $(\Delta_{\R^{n_{i}} \xx \mathcal{M}_{i}} +
k^{2})^{-j}$ for $j \in \N^{*}$ and $k > 0$, where $\Delta_{\R^{n_{i}}
\xx \mathcal{M}_{i}}$ denotes the Laplacian on $\R^{n_{i}} \xx \mathcal{M}_{i}$.  Recall the definition
of the Bessel kernel $G_{a}^{d} : \R \rightarrow (0,\infty)$ for $a, \, d > 0$,
$$
G_{a}^{d}(s) := \frac{1}{(4 \pi)^{\frac{a}{2}} \Gamma(a/2)}
\int^{\infty}_{0} \frac{e^{- \frac{\pi s^{2}}{t} - \frac{t}{4
      \pi}}}{t^{1 + \frac{d - a}{2}}} \, dt.
$$
 The Bessel kernels are well-known to satisfy the estimates,
 \begin{align}\begin{split}
     \label{eqtn:BesselEstimates}
 G^{d}_{a}(s) &\simeq \left\lbrace \begin{array}{c c}
                                       \frac{e^{-c s}}{s^{d - a}} & if
                                                                  \ a
                                                                  < d,
                                       \\ & \\
                                       \max\br{1,\ln(s^{-1})} e^{-c s} &
                                                                       if
                                                                       \
                                                                       d
                                                                       =
                                                                         a,
                                                                         \\
                                       & \\
                                     e^{-c s} & if \ a > d, \end{array} \right.
                               \end{split}\end{align}
                             where the value of $c$ is allowed to
                             differ in each case and for the upper and lower estimates.
                             So, when $a < d$ for instance,
                             $$
\frac{e^{-c_{1} s}}{s^{d - a}} \lesssim G^{d}_{a}(s) \lesssim \frac{e^{-c_{2} s}}{s^{d - a}},
                            $$
for some $c_{1} > c_{2} > 0$. Refer to \cite{aronszajn1961bessel} for a detailed proof.

\begin{prop} 
 \label{prop:Bessel} 
For $j \in \N^{*}$, there exists $c_{1}, \, c_{2}, \, c_{3} > 0$ for which
 $$
(\Delta_{\R^{n_{i}} \xx \mathcal{M}_{i}} + k^{2})^{-j}(x,y) \lesssim
k^{n_{i} - 2j} \cdot G^{n_{i}}_{2j}(c_{1} d(x,y) k) + k^{N - 2 j} \cdot
 G^{N}_{2j} (c_{1} d(x,y) k),
 $$
 $$
(\Delta_{\R^{n_{i}} \xx \mathcal{M}_{i}} + k^{2})^{-j}(x,y) \gtrsim
k^{n_{i} - 2j} \cdot G^{n_{i}}_{2j}(c_{2} d(x,y) k) + k^{N - 2 j} \cdot
 G^{N}_{2j} (c_{2} d(x,y) k)
 $$
 and
 $$
\abs{\nabla_{x} (\Delta_{\R^{n_{i}} \xx \mathcal{M}_{i}} +
  k^{2})^{-j}(x,y)} \lesssim k^{n_{i} + 1 - 2 j} \cdot G^{n_{i}}_{2j -
1}(c_{3} d(x,y) k) + k^{N + 1 - 2 j} G^{N}_{2j - 1}(c_{3} d(x,y) k)
$$
 for all $x, \, y \in \R^{n_{i}} \xx \mathcal{M}_{i}$ and $k >
 0$.
 \end{prop}

\begin{proof}  
 Recall that the higher-order resolvent can be expressed in terms of
 the heat operator through the integral relation
 \begin{equation}
   \label{eqtn:ResolventFormula}
   (\Delta_{\R^{n_{i}} \xx \mathcal{M}_{i}} + k^{2})^{-j} =
   \frac{1}{(j - 1)!}
   \int^{\infty}_{0} t^{j - 1} e^{-t k^{2}} e^{-t \Delta_{\R^{n_{i}}
       \xx \mathcal{M}_{i}}} \, dt.
 \end{equation}
 On applying the heat kernel estimate (11) of \cite{hs2019},
 \begin{align*}\begin{split}  
 (\Delta_{\R^{n_{i}} \xx \mathcal{M}_{i}} + k^{2})^{-j}(x,y) &=
 \frac{1}{(j - 1)!}
 \int_{0}^{\infty} t^{j - 1} e^{-t k^{2}} e^{-t \Delta_{\R^{n_{i}} \xx
     \mathcal{M}_{i}}}(x,y) \, dt \\
 &\lesssim \int_{0}^{\infty} t^{j - 1} (t^{-n_{i}/2} + t^{-N/2}) e^{-t
   k^{2}} e^{-c \frac{d(x,y)^{2}}{t}} \, dt \\
 &= \int_{0}^{\infty} t^{-1 - \frac{n_{i} - 2j}{2}} e^{-t k^{2}} e^{-c
 \frac{d(x,y)^{2}}{t}} \, dt + \int_{0}^{\infty} t^{-1 - \frac{N - 2j}{2}} e^{-t k^{2}} e^{-c
 \frac{d(x,y)^{2}}{t}}.
\end{split}\end{align*}
On applying a change of variables,
\begin{align}\begin{split}
    \label{eqtn:Bessel}
  (\Delta_{\R^{n_{i}} \xx \mathcal{M}_{i}} + k^{2})^{-j}(x,y) &\lesssim
  k^{n_{i} - 2j} \int^{\infty}_{0} t^{-1 - \frac{n_{i} - 2j}{2}}
  e^{-\frac{t}{4 \pi}} e^{- \pi \frac{4 c d(x,y)^{2}k^{2}}{t}} \, dt \\
  & \qquad \qquad + k^{N - 2 j} \int^{\infty}_{0} t^{-1 - \frac{N - 2
      j}{2}} e^{-\frac{t}{4 \pi}} e^{- \pi \frac{4 c d(x,y)^{2}
      k^{2}}{t}} \, dt \\
  &\simeq k^{n_{i} - 2j} \cdot G^{n_{i}}_{2j}(2 \sqrt{c} d(x,y) k) + k^{N - 2 j}
 \cdot G^{N}_{2j} (2 \sqrt{c} d(x,y) k).
\end{split}\end{align}
The lower estimate for $(\Delta_{\R^{n_{i}} \xx \mathcal{M}_{i}} +
k^{2})^{-j}(x,y)$ and the upper estimate for $\nabla_{x}
(\Delta_{\R^{n_{i}} \xx \mathcal{M}_{i}} + k^{2})^{-j}(x,y)$ follow in an identical manner to the upper estimate
through an application of (13) and (12) respectively of \cite{hs2019}.
 \end{proof}

 Proposition \ref{prop:Bessel}, when combined with the Bessel kernel
 estimates \eqref{eqtn:BesselEstimates}, immediately
 yield the following corollary.

\begin{cor} 
 \label{cor:EndResolvent} 
 For any $j \in \N^{*}$ with $j \neq \frac{n_{i}}{2}$ and $j \neq
 \frac{N}{2}$, there exists $c > 0$ such that
 \begin{align}\begin{split}  
     \label{eqtn:cor:EndResolvent1}
     \br{\Delta_{\R^{n_{i}} \xx \mathcal{M}_{i}} + k ^{2}}^{-j}(x,y)
 &\lesssim \left( d(x,y)^{\min(2j - N, 0)} k^{-\max(2j - N, 0)}  \right. \\
 & \qquad \qquad \qquad  \left. +
  d(x,y)^{\min(2j - n_{i}, 0)} k^{- \max(2j - n_{i}, 0)} \right) e^{-c k d(x,y)}.
\end{split}\end{align}
If $j = \frac{n_{i}}{2} < \frac{N}{2}$,
\begin{equation} 
 \label{eqtn:cor:EndResolvent3} 
 (\Delta_{\R^{n_{i}} \xx \mathcal{M}_{i}} + k^{2})^{-j}(x,y) \lesssim
 \br{d(x,y)^{2j - N} + \max \brs{1, \ln\br{\frac{1}{k d(x,y)}}}}
 e^{-c k d(x,y)}.
 \end{equation}
 If $j = \frac{n_{i}}{2} = \frac{N}{2}$,
 \begin{equation} 
 \label{eqtn:cor:EndResolvent2} 
 (\Delta_{\R^{n_{i}} \xx \mathcal{M}_{i}} + k^{2})^{-j}(x,y) \lesssim \max \brs{1, \ln\br{\frac{1}{k d(x,y)}}}
 e^{-c k d(x,y)}.
 \end{equation}
 Finally, if $\frac{n_{i}}{2} < j = \frac{N}{2}$,
  \begin{equation} 
 \label{eqtn:cor:EndResolvent4} 
 (\Delta_{\R^{n_{i}} \xx \mathcal{M}_{i}} + k^{2})^{-j}(x,y) \lesssim
 \br{ \max \brs{1, \ln\br{\frac{1}{k d(x,y)}}} + k^{n_{i} - 2 j}}
 e^{-c k d(x,y)}.
 \end{equation}
\end{cor}

Similarly, for the operators $\nabla(\Delta_{\R^{n_{i}} \xx
  \mathcal{M}_{i}} + k^{2})^{-j}$, the following estimates follow from Proposition \ref{prop:Bessel}.

\begin{cor} 
 \label{cor:GradEndResolvent} 
 For any $j \in \N^{*}$ with $j \neq \frac{n_{i} + 1}{2}$ and $j \neq
 \frac{N + 1}{2}$, there exists $c > 0$ such that
 \begin{align}\begin{split}  
     \label{eqtn:cor:GradEndResolvent1}
\abs{\nabla_{x}     \br{\Delta_{\R^{n_{i}} \xx \mathcal{M}_{i}} + k ^{2}}^{-j}(x,y)}
 &\lesssim \left( d(x,y)^{\min(2j - 1 - N, 0)} k^{-\max(2j - 1 - N, 0)}  \right. \\
 & \qquad \qquad \qquad  \left. +
  d(x,y)^{\min(2j - 1 - n_{i}, 0)} k^{- \max(2j - 1 - n_{i}, 0)} \right) e^{-c k d(x,y)}.
\end{split}\end{align}
If $j = \frac{n_{i} + 1}{2} < \frac{N + 1}{2}$,
\begin{equation} 
 \label{eqtn:cor:GradEndResolvent2} 
\abs{\nabla_{x} (\Delta_{\R^{n_{i}} \xx \mathcal{M}_{i}} + k^{2})^{-j}(x,y)} \lesssim
 \br{d(x,y)^{2j - 1 - N} +  \max \brs{1, \ln \br{\frac{1}{k d(x,y)}}}  }
 e^{-c k d(x,y)}.
 \end{equation}
 If $j = \frac{n_{i} + 1}{2} = \frac{N + 1}{2}$,
 \begin{equation} 
 \label{eqtn:cor:GradEndResolvent3} 
 \abs{ \nabla_{x} (\Delta_{\R^{n_{i}} \xx \mathcal{M}_{i}} + k^{2})^{-j}(x,y)} \lesssim
 \max \brs{1, \ln \br{\frac{1}{k d(x,y)}}}  e^{-c k d(x,y)}.
 \end{equation}
 Finally, if $\frac{n_{i} + 1}{2} < j = \frac{N + 1}{2}$,
  \begin{equation} 
 \label{eqtn:cor:GradEndResolvent4} 
 \abs{\nabla_{x} (\Delta_{\R^{n_{i}} \xx \mathcal{M}_{i}} + k^{2})^{-j}(x,y)} \lesssim
 \br{  \max \brs{1, \ln \br{\frac{1}{k d(x,y)}}} + k^{ n_{i} + 1 - 2 j}}
 e^{-c k d(x,y)}.
 \end{equation}
\end{cor}

 The below proposition follows almost immediately from the two
 previous corollaries. We provide an alternative proof, that follows
 directly from the first-order resolvent estimates, for the first
 estimate in the statement.

\begin{prop} 
 \label{prop:EndResolvent} 
Let $j \in \N^{*}$. There exists $c > 0$ such that
 \begin{equation}
   \label{eqtn:EndResolvent1}
   (\Delta_{\R^{n_{i}} \xx \mathcal{M}_{i}} + k^{2})^{-j}(x,y)
   \lesssim k^{-2(j - 1)} \br{d(x,y)^{2 - N} + d(x,y)^{2 - n_{i}}}
   e^{-c k d(x,y)}
 \end{equation}
 and
 \begin{equation}
   \label{eqtn:EndResolvent2}
\abs{\nabla_{x} \br{\Delta_{\R^{n_{i}} \xx \mathcal{M}_{i}} +
    k^{2}}^{-j}(x,y)} \lesssim k^{-2 (j - 1)} \br{d(x,y)^{1 - N} + d(x,y)^{1 -
  n_{i}}} e^{-c k d(x,y)} 
\end{equation}
for all $x, \, y \in \R^{n_{i}} \xx \mathcal{M}_{i}$ and $k > 0$.
\end{prop}

\begin{proof}  
  For any $k > 0$, the higher-order resolvent operator is given by the formula
  $$
(\Delta_{\R^{n_{i}} \xx \mathcal{M}_{i}} + k^{2})^{-j} = \frac{1}{(j -
  1)!}
\int^{\infty}_{0} t^{j - 1} e^{- t k^{2}} e^{-t \Delta_{\R^{n_{i}} \xx
  \mathcal{M}_{i}}} \, dt.
$$
Let $e^{-t \Delta_{\R^{n_{i}} \xx \mathcal{M}_{i}}}(x,y)$ denote the
heat kernel on $\R^{n_{i}} \xx \mathcal{M}_{i}$. Since $x^{j - 1} \lesssim e^{\epsilon x}$ for any $0 < \epsilon < 1$,
\begin{align*}\begin{split}  
 (\Delta_{\R^{n_{i}} \xx \mathcal{M}_{i}} + k^{2})^{-j}(x,y)
 &\simeq k^{-2(j - 1)} \int^{\infty}_{0} (t k^{2})^{(j - 1)} e^{-t
   k^{2}} e^{-t \Delta_{\R^{n_{i}} \xx \mathcal{M}_{i}}}(x,y) \, dt \\
 &\lesssim k^{-2(j - 1)}
 \int^{\infty}_{0} e^{-t k^{2}(1 - \epsilon)} e^{-t \Delta_{\R^{n_{i}}
     \xx \mathcal{M}_{i}}}(x,y) \, dt \\
 &= k^{-2 (j - 1)} (\Delta_{\R^{n_{i}} \xx \mathcal{M}_{i}} + k^{2}(1 - \epsilon))^{-1}(x,y).
\end{split}\end{align*}
On applying the known estimates for the first-order resolvent,
equation (17) from \cite{hs2019}, we obtain
$$
 (\Delta_{\R^{n_{i}} \xx \mathcal{M}_{i}} + k^{2})^{-j}(x,y) \lesssim
 k^{-2(j - 1)} \br{d(x,y)^{2 - N} + d(x,y)^{2 - n_{i}}} e^{- c' \sqrt{1
   - \epsilon} k d(x,y)},
$$
for some $c' > 0$. This proves \eqref{eqtn:EndResolvent1} with $c =
\sqrt{1 - \epsilon} c'$.
\end{proof}

Finally, when we come to consider the unboundedness of the square
function $S$ for $p \geq n_{min}$, the following lower bounds will
prove useful. This lower bound follows directly from Proposition
\ref{prop:Bessel} and the lower bound for the Bessel kernel given in \eqref{eqtn:BesselEstimates}.

\begin{cor}
\label{cor:LowerEndResolvent}
For $j < \frac{n_{i}}{2}$, there must exist $c > 0$ such that
$$
\br{\Delta_{\R^{n_{i}} \xx \mathcal{M}_{i}} + k^{2}}^{-j}(x,y) \gtrsim
\br{d(x,y)^{2 j - N} + d(x,y)^{2 j - n_{i}}} e^{-c k d(x,y)},
$$
for all $x, \, y \in \R^{n_{i}} \xx \mathcal{M}_{i}$ and $k > 0$.
\end{cor}

\subsection{A Decomposition for the Resolvent}
\label{subsec:Decomposition}

   Recall from \cite{hs2019} that in order to prove the boundedness of
 the low energy part of the Riesz transforms operator, the resolvent
 was separated into four separate components. That is, for $0 < k \leq
 1$ the resolvent is given by
 \begin{equation}
   \label{eqtn:ResolventSplitting}
(\Delta + k^{2})^{-1} = \sum_{j = 1}^{4} G_{j}(k).
\end{equation}
Let us recall the definitions of each of these components. For each $i
= 1, \cdots, l$, choose a point $x_{i}^{\circ}$ in the interior of  $K_{i}$. Let
$\phi_{i} \in C^{\infty}(\mathcal{M})$ be a function  with support
entirely contained in $\R^{n_{i}} \xx \mathcal{M}_{i} \setminus K_{i}$
that is identically equal to $1$ everywhere outside of a compact
set. Define $v_{i} := - \Delta \phi_{i}$ and let $u_{i}$ be the
function whose existence is asserted by {\cite[Lem.~2.7]{hs2019}} for
$v_{i}$. The term $G_{1}(k)$ is entirely supported on the diagonal ends and is
defined through
$$
G_{1}(k)(x,y) := \sum_{i = 1}^{l} \br{\Delta_{\R^{n_{i}} \xx
    \mathcal{M}_{i}} + k^{2}}^{-1}(x,y) \phi_{i}(x) \phi_{i}(y).
$$
Let $G_{int}(k)$ be an interior parametrix for the resolvent that is
supported close to the compact subset
$$
K_{\Delta} := \lb (x,x) : x \in
K \rb \subset \mathcal{M}^{2},
$$
and agreeing with the resolvent of $\Delta_{\R^{n_{i}} \xx
  \mathcal{M}_{i}}$ in a smaller neighbourhood of $K_{\Delta}$
intersected with the support of $\nabla \phi_{i}(x) \phi_{i}(y)$.
$G_{2}(k)$ is an operator with kernel that is compactly supported
in $\mathcal{M}^{2}$,
$$
G_{2}(k)(x,y) := G_{int}(k)(x,y) \br{1 - \sum_{i = 1}^{l} \phi_{i}(x) \phi_{i}(y)}.
$$
$G_{3}(k)$ has the nice property that its kernel is multiplicatively
separable into functions of $x$ and $y$,
$$
G_{3}(k)(x,y) = \sum_{i = 1}^{l} \br{\Delta_{\R^{n_{i}} \xx
    \mathcal{M}_{i}} + k^{2}}^{-1}(x_{i}^{\circ},y) u_{i}(x,k) \phi_{i}(y).
$$
For the final term $G_{4}(k)$, first the error term is defined by
  $$
(\Delta + k^{2})(G_{1}(k) + G_{2}(k) + G_{3}(k)) = I + E(k).
$$
Then the operator $G_{4}(k)$ is given by
$$
G_{4}(k)(x,y) := - \br{\Delta + k^{2}}^{-1} v_{y}(x),
$$
where $v_{y}(x) := E(k)(x,y)$. As computed in \cite{hs2019}, it is
useful to note that the error term $E(k)$ has the representation
\begin{equation}
  \label{eqtn:ErrorSplitting}
E(k) = \sum_{i = 1}^{l} (E_{1}^{i}(k) + E_{2}^{i}(k)) + E_{3}(k).
\end{equation}
Here
$$
E_{1}^{i}(k)(x,y) := -2 \nabla \phi_{i}(x) \phi_{i}(y)
\brs{\nabla_{x}\br{\Delta_{\R^{n_{i}} \xx \mathcal{M}_{i}} +
    k^{2}}^{-1}(x,y) - \nabla_{x} G_{int}(k)(x,y)},
$$
$$
E_{2}^{i}(k)(x,y) := \phi_{i}(y) v_{i}(x) \br{- \br{\Delta_{\R^{n_{i}} \xx
    \mathcal{M}_{i}} + k^{2}}^{-1}(x,y) + G_{int}(k)(x,y) +
\br{\Delta_{\R^{n_{i}} \xx \mathcal{M}_{i}} + k^{2}}^{-1}(x_{i}^{\circ},y)}
$$
for $i = 1, \cdots, l$, and
$$
E_{3}(k)(x,y) := \br{(\Delta + k^{2}) G_{int}(k)(x,y) - \delta_{y}(x)}
\br{1 - \sum_{i = 1}^{l} \phi_{i}(x) \phi_{i}(y)},
$$
where $\delta_{y}$ is the Dirac-delta function centered at $y$.

\section{Higher-Order Resolvents on $\mathcal{M}$}
  \label{sec:Resolvent}

In this section, we investigate various properties of the higher-order
resolvent operators $(\Delta + k^{2})^{-j}$ for $\mathcal{M}$.
For $a \in \N$ and $c > 0$, define the weight functions $\omega_{a}^{c} :
\mathcal{M} \xx [0,1] \rightarrow (0,\infty)$ through
$$
\omega_{a}^{c}(x,k) := \left\lbrace \begin{array}{c c} 1, & x \in K, \\
                                  \langle d(x_{i}^{\circ},x)
                                  \rangle^{-(n_{i} - a)} e^{-c k d(x_{i}^{\circ},x)}, & x \in
                                                         \R^{n_{i}}
                                                         \xx
                                                         \mathcal{M}_{i}
                                                         \setminus
                                                                                        K_{i},
                                                                                        \
                                                                                        1
                                                                                        \leq
                                                                                        i
                                                                                        \leq
                                                                                        l. \end{array}
                                                     \right.
                                                     $$
The dependence of the functions $\omega_{a}^{c}$ on the constant $c$
will often be kept implicit through the use of the shorthand notation
$\omega_{a}$, and the value of $c$ will then be understood to change
from line to line.
Higher-order analogues of the key lemma from \cite{hs2019} will now be
proved.
   
\begin{prop} 
 \label{prop:Lemma2.7Higher} 
 Let $v \in L^{\infty}(\mathcal{M})$  be compactly supported in $K$. Let
 $u : \mathcal{M} \xx \R_{+} \rightarrow \R$ be a function, whose
 existence is asserted by {\cite[Lem.~2.7]{hs2019}}, that satisfies
 $(\Delta + k^{2})u = v$,
 \begin{align*}\begin{split}  
     &\abs{u(x,k)} \lesssim \norm{v}_{\infty} \omega_{2}(x,k), \ and \\
     &\abs{\nabla u(x,k)} \lesssim \norm{v}_{\infty} \omega_{1}(x,k),
   \end{split}\end{align*}
for all $x \in \mathcal{M}$ and $0 \leq k \leq 1$.
For $j \in \N$, define $u^{(j)} := \partial_{k^{2}}^{(j)} u$. Then
\begin{equation}
  \label{eqtn:HigherLemma2.7}
  \abs{u^{(j)}(x,k)} \lesssim k^{-2 j} \norm{v}_{\infty} \omega_{2}(x,k),
  \end{equation}
for all $x \in \mathcal{M}$ and $0 < k \leq 1$.
\end{prop}

\begin{proof}
 It will first be proved that for any $j \in \N$, the derivative
 \begin{equation}
   \label{eqtn:uDerivative}
 u^{(j)}(x,k) = \partial_{k^{2}}^{(j)} (\Delta + k^{2})^{-1}v(x) =
 (-1)^{j} j!(\Delta + k^{2})^{-(j + 1)}v(x)
 \end{equation}
 satisfies the integral identity
  \begin{equation}
    \label{eqtn:jthDerivative}
    u^{(j)}(x,k) = (-1)^{j} \int^{\infty}_{0} t^{j} e^{-t k^{2}} e^{-t
    \Delta} v(x) \, dt
\end{equation}
for any $k > 0$ and $x \in \mathcal{M}$. This will be proved by induction. The base case,
$$
u^{(0)}(x,k) = u(x,k) = \int^{\infty}_{0} e^{-t k^{2}} e^{-t \Delta}
v(x) \, dt,
$$
was shown to hold for any $k \geq 0$ in {\cite[Lem.~2.7]{hs2019}}. For the
inductive step, assume that \eqref{eqtn:jthDerivative} holds for a
particular $j \in \N$. Then,
\begin{align}\begin{split}
    \label{eqtn:InductiveStep}
    u^{(j + 1)}(x,k) &= \partial_{k^{2}} u^{(j)}(x,k) \\
    &= \partial_{k^{2}} (-1)^{j} \int^{\infty}_{0} t^{j} e^{-t k^{2}}
    e^{-t \Delta}v(x) \, dt \\
    &= (-1)^{j} \lim_{h \rightarrow 0} \int^{\infty}_{0} t^{j}
    \br{\frac{e^{-t(k^{2} + h)} - e^{-t k^{2}}}{h}} e^{-t \Delta}v(x)
    \, dt.
 \end{split}\end{align}
The mean value theorem tells us that there must exist some $c(h) \in
[0,h]$ for which
\begin{align*}\begin{split}  
\abs{\frac{e^{-t (k^{2} + h)} - e^{-t k^{2}}}{h}} 
&= t e^{-t(k^{2} + c(h))} \\
&\leq t e^{-t k^{2}}.
 \end{split}\end{align*}
As the semigroup generated by the Laplace-Beltrami operator is
Markovian, $\norm{e^{-t \Delta}v}_{\infty} \leq \norm{v}_{\infty}$ for
all $t > 0$. In addition, from {\cite[Cor.~4.9]{grigoryan}}, we
know that $\norm{e^{-t \Delta} v}_{\infty} \lesssim
\norm{v}_{1} t^{-\frac{n_{min}}{2}} \lesssim \norm{v}_{1} t^{-\frac{3}{2}}$ for $t \geq 1$. Thus,
\begin{align*}\begin{split}  
 \int^{\infty}_{0} t^{j + 1} e^{-t k^{2}} \abs{e^{-t \Delta}v(x)} \, dt &=
   \int^{1}_{0} t^{j + 1} e^{-t k^{2}} \abs{e^{-t \Delta}v(x)} \,
 dt + \int^{\infty}_{1} t^{j + 1} e^{-t k^{2}} \abs{e^{-t \Delta}v(x)} \,
 dt \\
 &\lesssim   \norm{v}_{\infty} \int^{1}_{0} t^{j + 1} e^{-t
   k^{2}} \, dt + \norm{v}_{1} \int^{\infty}_{1} t^{j - \frac{1}{2}}
 e^{-t k^{2}} \, dt \\
 &< \infty.
\end{split}\end{align*}
The dominated convergence theorem applied to
\eqref{eqtn:InductiveStep} then yields
$$
u^{(j + 1)}(x,k) = (-1)^{j + 1} \int^{\infty}_{0} t^{j + 1} e^{-t k^{2}} e^{-t \Delta} v(x) \, dt,
$$
thereby completing the inductive proof of \eqref{eqtn:jthDerivative}
for general $j \in \N$, $x \in \mathcal{M}$ and $k > 0$.

\vspace*{0.2in}

The hypothesized estimate for $u^{(j)}$ can now be proved using our freshly minted integral identity. Utilising the bound $x^{j} \lesssim e^{\epsilon x}$ for any $0 <
\epsilon < 1$ leads to
\begin{align*}\begin{split}  
 \abs{u^{(j)}(x,k)} &\leq k^{-2 j} \int^{\infty}_{0} (t k^{2})^{j}
 e^{-t k^{2}} e^{-t \Delta } \abs{v(x)} \, dt \\
 &\lesssim k^{-2 j} \int^{\infty}_{0} e^{-t k^{2}(1 - \epsilon)}
 e^{-t \Delta} \abs{v(x)} \, dt \\
 &= k^{-2 j} \br{\Delta + k^{2}(1 - \epsilon)}^{-1}\abs{v}(x).
\end{split}\end{align*}
Apply {\cite[Lem.~2.7]{hs2019}} to obtain,
$$
 \abs{u^{(j)}(x,k)} \lesssim k^{-2 j} \norm{v}_{\infty}
 \omega_{2}^{c'}(x,k\sqrt{1 - \epsilon}),
 $$
 for some $c' > 0$. Thus  \eqref{eqtn:HigherLemma2.7} holds with constant $c = c' \sqrt{1 -
   \epsilon}$.
  \end{proof}

  Notice that in the previous proposition, the proof of the integral identity
  \eqref{eqtn:jthDerivative} relied upon the key condition $k > 0$. For $k
  = 0$, if $j$ is too large then the integral in
  \eqref{eqtn:jthDerivative} will no longer be guaranteed to converge
  absolutely and, consequently, the derivative will not 
  exist. This is analogous to the well-known fact that on standard
  Euclidean space $\R^{d}$, the Riesz potentials
  $\Delta_{\R^{d}}^{-\alpha}$ are not defined for $\alpha \geq
  \frac{d}{2}$. Acting as a converse to this, the below proposition states that if $j$ is small
  enough then the derivative at $k = 0$ will exist and, moreover,
  satisfy some useful estimates. 

  \begin{prop} 
 \label{prop:DerivativeAt0} 
 Let $v \in L^{\infty}(\mathcal{M})$ be
compactly supported in $K$ and $u$ be as given in
{\cite[Lem.~2.7]{hs2019}}. Then for any $0 \leq j < \frac{n_{min}}{2}
- 1$,
 \begin{align}\begin{split}  
\label{eqtn:DerivativeAt0}
 &\abs{u^{(j)}(x,k)} \lesssim  \norm{v}_{\infty} \omega_{2j + 2}(x,k), \ and \\
 &\abs{\nabla u^{(j)}(x,k)} \lesssim 
 \norm{v}_{\infty} \omega_{2j + 1}(x,k),
\end{split}\end{align}
for all $k \in [0,1]$ and $x \in \mathcal{M}$.
Moreover, for $k > 0$,
 \begin{equation}
   \label{eqtn:DerivativeAt01}
\norm{u^{(j)}(\cdot,k) - u^{(j)}(\cdot,0)}_{L^{\infty}(\mathcal{M})} \lesssim k \norm{v}_{L^{\infty}(\mathcal{M})}
 \end{equation}
 and
 \begin{equation}
   \label{eqtn:DerivativeAt02}
\norm{\nabla u^{(j)}(\cdot,k) - \nabla
  u^{(j)}(\cdot,0)}_{L^{\infty}(\mathcal{M})} \lesssim k^{1 - 2 j} \norm{v}_{L^{\infty}(\mathcal{M})}.
 \end{equation}
\end{prop}

\begin{proof}  
From the previous proposition, we know that the derivative
$u^{(j)}(x,k)$ exists for $k > 0$ and is given by the integral identity
\eqref{eqtn:jthDerivative}. With the addition of the assumption $j
< \frac{n_{min}}{2} - 1$, it is clear that $u^{(j)}(x,0)$ will also
exist and that \eqref{eqtn:jthDerivative} will be equally
applicable. Indeed, this follows in an identical inductive manner as
the proof for $k > 0$ from Proposition \ref{prop:Lemma2.7Higher}. However,
in order to apply the dominated convergence theorem at the inductive
step of the proof, it will be required that
$$
\int^{\infty}_{0} t^{i + 1} e^{-t \cdot 0} \abs{e^{-t \Delta}v(x)} \, dt
$$
is absolutely convergent for when $1 \leq i + 1 \leq j$. This follows from
\begin{align*}\begin{split}  
 \int^{\infty}_{0} t^{i + 1} \abs{e^{-t \Delta}v(x)} \, dt &=
 \int^{1}_{0} t^{i + 1} \abs{e^{-t \Delta}v(x)} \, dt + \int^{\infty}_{1} t^{i + 1} \abs{e^{-t \Delta}v(x)} \, dt \\
 &\lesssim \norm{v}_{\infty} \int^{1}_{0} t^{i + 1} \, dt + \norm{v}_{1} \int^{\infty}_{1} t^{i + 1 - \frac{n_{min}}{2}} \,
 dt.
\end{split}\end{align*}
The second integral will converge if and only if $i + 1 <
\frac{n_{min}}{2} - 1$, which is implied by our assumption $j <
\frac{n_{min}}{2} - 1$.

\vspace*{0.2in}

Notice that the restriction $0 \leq j < \frac{n_{min}}{2} - 1$ not only implies the
existence of $u^{(j)}(x,0)$, but it also tells us that
$\frac{(-1)^{j}}{j!}u^{(j)}$ is a solution of the equation
$(\Delta + k^{2})^{j + 1}\tilde{u} = v$ with $u^{(j)} \in L^{\infty}$
uniformly in $k \in [0,1]$. Also observe that from
the identity \eqref{eqtn:uDerivative},
$$
(\Delta + k^{2})u^{(j)} = -j u^{(j - 1)}.
$$
Recall from the proof of {\cite[Lem.~2.7]{hs2019}} that
$\norm{\Delta^{m} u^{(0)}}_{L^{\infty}(V)} \lesssim 1$ for any
compact $V \subset \mathcal{M} \setminus K$ and $m \in \N$. Then, for
any $1 \leq j < \frac{n_{min}}{2} - 1$,
\begin{align*}\begin{split}  
 \norm{\Delta u^{(j)}}_{L^{\infty}(V)} &\leq \norm{(\Delta + k^{2}) u^{(j)}}_{L^{\infty}(V)} + k^{2}\norm{u^{(j)}}_{L^{\infty}(V)} \\
&\simeq \norm{u^{(j - 1)}}_{L^{\infty}(V)} + k^{2} \norm{u^{(j)}}_{L^{\infty}(V)} \\
&\lesssim 1.
\end{split}\end{align*}
 We also have
\begin{align*}\begin{split}  
 \norm{\Delta^{2} u^{(j)}}_{L^{\infty}(V)} &\leq \norm{\Delta (\Delta +
   k^{2})u^{(j)}}_{L^{\infty}(V)} + k^{2} \norm{\Delta u^{(j)}}_{L^{\infty}(V)} \\
&\simeq \norm{\Delta u^{(j - 1)}}_{L^{\infty}(V)} + k^{2} \norm{\Delta u^{(j)}}_{L^{\infty}(V)} \\
&\lesssim 1,
\end{split}\end{align*}
for all $1 \leq j < \frac{n_{min}}{2} - 1$.
This process can be repeated to arbitrarily high order to obtain
$\norm{\Delta^{m} u^{(j)}}_{L^{\infty}(V)} \lesssim 1$ for any $m, \, j \in \N$ and compact
subset $V \subset \mathcal{M} \setminus K$. This proves that $u^{(j)}
\in C^{\infty}$ uniformly in $k \in [0,1]$ on any compact subset
outside of $K$. The argument
for \eqref{eqtn:DerivativeAt0} then follows in an
identical manner to the corresponding estimates in
{\cite[Lem.~2.7]{hs2019}}. In particular, for $\zeta_{i} \in
C^{\infty}(\mathcal{M})$ with support contained in $\R^{n_{i}} \xx
\mathcal{M}_{i} \setminus K_{i}$, therefore satisfying $\mathrm{supp} \, v
\subset (\mathrm{supp} \,
\zeta_{i})^{c}$, and such that $(1 - \zeta_{i})$ is
compactly supported when viewed as a function on $\R^{n_{i}} \xx \mathcal{M}_{i}$. Define
$$
\tilde{u}_{i}^{(j)}(x,k) := \br{\Delta_{\R^{n_{i}} \xx
    \mathcal{M}_{i}} + k^{2}}^{-(j + 1)} \br{(\Delta + k^{2})^{(j +
    1)} (\zeta_{i} u^{(j)}(x,k))}.
$$
It can be reasoned that $\tilde{u}_{i}^{(j)} = u^{(j)} \zeta_{i}$, at which
point Corollaries \ref{cor:EndResolvent} and
\ref{cor:GradEndResolvent} can be applied to produce \eqref{eqtn:DerivativeAt0}.

\vspace*{0.1in}

Let's now prove \eqref{eqtn:DerivativeAt01}. From the
relation \eqref{eqtn:jthDerivative} in combination with the estimates
$\norm{e^{-t \Delta}v}_{\infty} \lesssim \norm{v}_{\infty}$ for $t
\leq 1$ and $\norm{e^{-t \Delta}v}_{\infty} \lesssim \norm{v}_{1}
t^{-\frac{n_{min}}{2}}$ for $t > 1$,
$$
 \abs{u^{(j)}(x,k) - u^{(j)}(x,0)} \leq
 \norm{v}_{\infty}\int^{1}_{0} t^{j} (1 - e^{-t k^{2}}) \, dt +
 \norm{v}_{1} \int^{\infty}_{1} t^{j}(1 - e^{-t k^{2}})
 t^{-\frac{n_{min}}{2}} \, dt.
$$
For the first term, the estimate $1 - e^{-x} \leq x$ for $x \in (0,1)$
implies
\begin{align*}\begin{split}  
\int^{1}_{0} t^{j} (1 - e^{-t k^{2}}) \, dt &\leq k^{2}
\int^{1}_{0} t^{j + 1} \, dt \\ &\leq k^{2}.
 \end{split}\end{align*}
For the second term, notice that since $j < \frac{n_{min}}{2} - 1$ we
must have $j \leq \frac{n_{min}}{2} - \frac{3}{2}$. A change of
variables then leads to
\begin{align*}\begin{split}  
 \int^{\infty}_{1} t^{j}(1 - e^{-t k^{2}}) t^{-\frac{n_{min}}{2}} \,
 dt &\leq \int^{\infty}_{1} t^{-\frac{3}{2}} (1 - e^{-t k^{2}}) \, dt \\
 &= k \int^{\infty}_{k^{2}} t^{-\frac{3}{2}} (1 - e^{-t})
 \, dt \\
 &\leq k \int^{\infty}_{0} t^{- \frac{3}{2}} (1 - e^{-t})
 \, dt \\
 &\lesssim k.
\end{split}\end{align*}

\vspace*{0.1in}

Finally, let us consider the validity of
\eqref{eqtn:DerivativeAt02}. This estimate has already been proved for
the case $j = 0$ in {\cite[Lem.~2.7]{hs2019}}. For $1 \leq j <
\frac{n_{min}}{2} - 1$, observe that
$$
(\Delta + k^{2}) u^{(j)}(x,k) = - j u^{(j - 1)}(x,k) \quad and \quad
\Delta u^{(j)}(x,0) = -j u^{(j - 1)}(x,0).
$$
Therefore,
$$
 \abs{\Delta u^{(j)}(x, k) - \Delta u^{(j)}(x,0) }
 \lesssim k^{2} \abs{u^{(j)}(x,k)} + \abs{u^{(j - 1)}(x,k) -
   u^{(j - 1)}(x,0)}.
$$
Proposition \ref{prop:Lemma2.7Higher} and \eqref{eqtn:DerivativeAt01}
then imply
\begin{align*}\begin{split}  
 \abs{\Delta u^{(j)}(x,k) - \Delta u^{(j)}(x,0)}
 &\lesssim \br{k^{2 - 2 j} + k} \norm{v}_{\infty} \\
 &\lesssim k^{1 - 2j} \norm{v}_{\infty}.
\end{split}\end{align*}
Remark 2.8 of \cite{hs2019} states that
$$
\norm{\nabla g}_{L^{\infty}(\mathcal{M})} \lesssim \norm{\Delta
  g}_{L^{\infty}(\mathcal{M})} + \norm{g}_{L^{\infty}(\mathcal{M})}.
$$
Therefore,
\begin{align*}\begin{split}  
 \norm{\nabla (u^{(j)}(\cdot,k) - u^{(j)}(\cdot,0))}_{\infty}
 &\lesssim \norm{\Delta (u^{(j)}(\cdot,k) - u^{(j)}(\cdot,0))}_{\infty} + \norm{u^{(j)}(\cdot,k) - u^{(j)}(\cdot,0)}_{\infty} \\
 &\lesssim k^{1 - 2 j} \norm{v}_{\infty},
\end{split}\end{align*}
which completes our proof.
\end{proof}

\begin{prop} 
 \label{prop:HigherOrder} 
 Let $v \in L^{\infty}(\mathcal{M})$ be compactly supported in
 $K$. Then, for any $j \in \N^{*}$,
 \begin{equation}
   \label{eqtn:HigherOrder}
   \abs{\nabla (\Delta + k^{2})^{-j}v(x)} \lesssim k^{-2 (j - 1)}
   \norm{v}_{\infty} \omega_{1}(x,k)
 \end{equation}
 for all $x \in \mathcal{M}$ and $0 < k \leq 1$.
\end{prop}

\begin{proof}
  Before beginning the proof, note that in the notation of
  Propositions \ref{prop:Lemma2.7Higher} and \ref{prop:DerivativeAt0}, $\nabla (\Delta + k^{2})^{-j}v
  = \nabla u^{(j - 1)}$. We use resolvent notation in this proposition
  to simplify our computations.

 This estimate has already been proved for $j = 1$ in
 {\cite[Lem.~2.7]{hs2019}}. For $j > 1$,
 \begin{align*}\begin{split}  
 \abs{\nabla (\Delta + k^{2})^{-j}v(x)} &= \abs{\nabla (\Delta +
   k^{2})^{-1} (\Delta + k^{2})^{-(j - 1)}v(x)} \\
 &\leq \sum_{i = 1}^{4} \abs{\nabla G_{i}(k)(\Delta + k^{2})^{-(j -
     1)}v(x)} \\
 &\leq \sum_{i = 1}^{4} \int_{\mathcal{M}} \abs{\nabla_{x} G_{i}(k)(x,y)} \abs{(\Delta +
 k^{2})^{-(j - 1)}v(y)} \, dy \\
 &=: \sum_{i = 1}^{4} Y_{i}(k)(x).
\end{split}\end{align*}
Our desired estimate will be proved for each of the above terms
separately.

\vspace*{0.1in}

\underline{\textit{The Term $Y_{1}(k)$.}}

\vspace*{0.1in}

The product rule implies that
$$
\nabla_{x} G_{1}(k)(x,y) = \sum_{i = 1}^{l} \nabla_{x} (\Delta_{\R^{n_{i}} \xx
\mathcal{M}_{i}} + k^{2})^{-1}(x,y) \phi_{i}(x) \phi_{i}(y) +
(\Delta_{\R^{n_{i}} \xx \mathcal{M}_{i}} + k^{2})^{-1}(x,y) \nabla
\phi_{i}(x) \phi_{i}(y).
$$
This leads to the splitting
$$
 Y_{1}(k)(x) \leq Y_{1}^{1}(k)(x) + Y^{2}_{1}(k)(x),
 $$
 where
 $$
Y^{1}_{1}(k)(x) = \int_{\mathcal{M}} \abs{\sum_{i = 1}^{l} \nabla_{x} (\Delta_{\R^{n_{i}} \xx
\mathcal{M}_{i}} + k^{2})^{-1}(x,y) \phi_{i}(x) \phi_{i}(y)} \abs{(\Delta +
k^{2})^{-(j - 1)}v(y)} \, dy
$$
and
 $$
Y^{2}_{1}(k)(x) = \int_{\mathcal{M}} \abs{\sum_{i = 1}^{l} (\Delta_{\R^{n_{i}} \xx
\mathcal{M}_{i}} + k^{2})^{-1}(x,y) \nabla \phi_{i}(x) \phi_{i}(y)} \abs{(\Delta +
k^{2})^{-(j - 1)}v(y)} \, dy.
$$
If $x \in K$, then both of these terms will clearly vanish, trivially
leading to
$$
Y_{1}(k)(x) \lesssim k^{-2(j - 1)} \norm{v}_{\infty} \omega_{1}(x,k).
$$
Let us consider $x \in \R^{n_{i}} \xx
\mathcal{M}_{i} \setminus K_{i}$ for some $1 \leq i \leq l$.
 For the first term, let $D > 0$ be
such that $d(x_{i}^{\circ},y) \geq 2 D$ for all $y \in
\R^{n_{i}} \xx \mathcal{M}_{i} \setminus K_{i}$.  Corollary
\ref{cor:GradEndResolvent} and Proposition \ref{prop:Lemma2.7Higher}
then produce,
\begin{align*}\begin{split}  
 &Y_{1}^{1}(k)(x) \lesssim \int_{\R^{n_{i}} \xx \mathcal{M}_{i} \setminus K_{i}}
 \abs{\nabla_{x} (\Delta_{\R^{n_{i}} \xx \mathcal{M}_{i}} +
   k^{2})^{-1}(x,y)} \abs{(\Delta + k^{2})^{-(j - 1)}v(y)} \, dy \\
 & \ \ \lesssim k^{-2(j - 2)} \norm{v}_{\infty} \int_{\R^{n_{i}} \xx \mathcal{M}_{i}
   \setminus K_{i}} \brs{d(x,y)^{1 - n_{i}} + d(x,y)^{1 - N}} \langle
 d(x_{i}^{\circ},y) \rangle^{2 - n_{i}} e^{- c k (d(x,y) +
 d(x_{i}^{\circ},y))} \, dy \\
&=: I_{1} + I_{2},
\end{split}\end{align*}
where the above integral has been separated into two distinct integrals $I_{1}$ and
$I_{2}$ over the respective regions
$$
\mathcal{A}_{1} := (\R^{n_{i}} \xx \mathcal{M}_{i} \setminus K_{i})
\cap B(x,d(x_{i}^{\circ},x) / 2)
$$
and
$$
\mathcal{A}_{2} := (\R^{n_{i}} \xx \mathcal{M}_{i} \setminus K_{i})
\cap B(x,d(x_{i}^{\circ},x) / 2)^{c}.
$$
 For $I_{1}$, it is evident that
$d(x_{i}^{\circ},y) \gtrsim d(x_{i}^{\circ},x)$ for any $y \in
\mathcal{A}_{1}$. Therefore,
\begin{align*}\begin{split}  
 I_{1} &= k^{-2(j - 2)} \norm{v}_{\infty} \int_{\mathcal{A}_{1}} \brs{d(x,y)^{1 - n_{i}} + d(x,y)^{1 - N}}
 \langle d(x_{i}^{\circ},y) \rangle^{2 - n_{i}} \exp (-c k (d(x,y) +
 d(x_{i}^{\circ},y))) \, dy \\
 &\lesssim k^{-2(j - 2)} \norm{v}_{\infty} \langle d(x_{i}^{\circ},x) \rangle^{2 - n_{i}} e^{-c k
   d(x_{i}^{\circ},x)}  \int_{\mathcal{A}_{1}} \brs{d(x,y)^{1 - n_{i}} +
 d(x,y)^{1 - N}} \exp(- c k d(x,y)) \, dy \\
 &\lesssim k^{-2(j - 2)} \norm{v}_{\infty} \langle d(x_{i}^{\circ},x) \rangle^{2 - n_{i}} e^{-c k
   d(x_{i}^{\circ},x)} \br{\int_{B(x,D)} d(x,y)^{1 - N} \, dy +
   \int_{B(x,d(x_{i}^{\circ},x)/2) \setminus B(x,D)} d(x,y)^{1 -
     n_{i}} \, dy} \\
 &\lesssim k^{-2(j - 2)} \norm{v}_{\infty} \langle d(x_{i}^{\circ},x) \rangle^{2 - n_{i}} e^{-c k
   d(x_{i}^{\circ},x)} \br{1 + d(x_{i}^{\circ},x)}
 \\
 &\lesssim k^{-2 (j - 1)} \norm{v}_{\infty} \langle d(x_{i}^{\circ},x) \rangle^{1 - n_{i}} e^{-\frac{c}{2} k d(x_{i}^{\circ},x)},
\end{split}\end{align*}
where the last line follows on absorbing the term $(k
d(x_{i}^{\circ},x))^{2}$ into the exponential.
For the integral $I_{2}$,
\begin{align*}\begin{split}  
 I_{2} &= k^{-2(j - 2)} \norm{v}_{\infty}\int_{\mathcal{A}_{2}} \brs{d(x,y)^{1 -
     n_{i}} + d(x,y)^{1 - N}} \langle d(x_{i}^{\circ},y) \rangle^{2 -
   n_{i}} e^{- c k (d(x,y) + d(x_{i}^{\circ},y))} \, dy \\
 &\lesssim k^{-2(j - 2)} \norm{v}_{\infty} d(x_{i}^{\circ},x)^{1 - n_{i}} e^{-c k
   d(x_{i}^{\circ},x)} \int_{B(x,d(x_{i}^{\circ},x)/2)^{c}} \langle
 d(x_{i}^{\circ},y) \rangle^{2 - n_{i}} e^{- c k d(x_{i}^{\circ},y)}
 \, dy \\
 &\lesssim k^{-2(j - 2)} \norm{v}_{\infty} \langle d(x_{i}^{\circ},x) \rangle^{1 -
   n_{i}} e^{-c k d(x_{i}^{\circ},x)} \int_{\R^{n_{i}}}
 \frac{1}{\abs{y_{1}}^{n_{i} - 2}} e^{-c k \abs{y_{1}}} \, dy_{1} \\
 &\simeq k^{-2(j - 2)} \norm{v}_{\infty} \langle d(x_{i}^{\circ},x) \rangle^{1 -
   n_{i}} e^{-c k d(x_{i}^{\circ},x)} \int_{0}^{\infty} r e^{-c k r} \, dr
 \\
 &\lesssim k^{-2(j - 1)} \norm{v}_{\infty} \langle d(x_{i}^{\circ},x) \rangle^{1 -
   n_{i}} e^{-c k d(x_{i}^{\circ},x)}.
\end{split}\end{align*}
For the term $Y_{1}^{2}(k)$, Corollary
\ref{cor:EndResolvent} and Proposition \ref{prop:Lemma2.7Higher} imply
\begin{align*}\begin{split}  
 &Y^{2}_{1}(k)(x) = \int_{\R^{n_{i}} \xx \mathcal{M}_{i} \setminus
   K_{i}} \abs{(\Delta_{\R^{n_{i}} \xx \mathcal{M}_{i}} + k^{2})^{-1}(x,y) \nabla
   \phi_{i}(x) \phi_{i}(y)} \abs{(\Delta + k^{2})^{-(j - 1)}v(y)} \,
 dy \\
 & \ \ \lesssim \abs{\nabla \phi_{i}(x)} k^{-2(j - 2)} \norm{v}_{\infty} \int_{\R^{n_{i}} \xx
 \mathcal{M}_{i} \setminus K_{i}} \brs{d(x,y)^{2 - n_{i}} + d(x,y)^{2
   - N}} \langle d(x_{i}^{\circ},y) \rangle^{2 - n_{i}} e^{-c k
 (d(x,y) + d(x_{i}^{\circ},y))} \, dy.
\end{split}\end{align*}
Using the same argument as for the term $Y^{1}_{1}(k)$, we will obtain
\begin{align*}\begin{split}  
  Y^{2}_{1}(k)(x) &\lesssim  \abs{\nabla \phi_{i}(x)} k^{-2(j - 1)} \norm{v}_{\infty}
  \langle d(x_{i}^{\circ},x) \rangle^{2 - n_{i}} e^{-c k
    d(x_{i}^{\circ},x)} \\
  &\lesssim k^{-2(j - 1)} \norm{v}_{\infty} \omega_{1}(x,k),
\end{split}\end{align*}
where we used the fact that $\nabla \phi_{i}$ is compactly supported
to obtain the last line.

\vspace*{0.1in}

\underline{\textit{The Term $Y_{2}(k)$.}}

\vspace*{0.1in}

The kernel $\nabla_{x}
G_{2}(k)(x,y)$ is compactly supported in some neighbourhood of
$$
K_{\Delta} := \lb (x,x) \in \mathcal{M}^{2} : x \in K \rb.
$$
There must then exist some compact set $V$ containing $K$ such that
$\nabla_{x} G_{2}(k)(x,y)$ is supported in $V \xx V$. This implies that
$Y_{2}(k)$ must also be compactly supported in $V$. For $x \in V^{c}$,
the estimate
\begin{equation}
  \label{eqtn:Y2}
Y_{2}(k)(x) \lesssim k^{-2 (j-1)} \norm{v}_{\infty} \omega_{1}(x,k)
\end{equation}
will then be trivially satisfied. It remains to ascertain the validity
of this estimate for $x \in V$. The operators $\lb \nabla G_{2}(k)
\rb_{k \in (0,1)}$ constitute a family of pseudodifferential operators
of order $-1$. At the scale of an individual chart on $\mathcal{M}$,
it is evident that in local coordinates the symbol of $\nabla
G_{2}(k)$, denoted $a_{2}(k)(x,\xi)$, will satisfy
$$
\abs{\partial^{\alpha}_{\xi} a_{2}(k)(x,\xi)} \lesssim (\abs{\xi}^{2} +
k^{2})^{\frac{-1 - \abs{\alpha}}{2}}
$$
for all multi-indices $\alpha \geq 0$ and $k \in (0,1)$. From this,
standard pseudodifferential operator theory (c.f. {\cite[Sec.~0.2]{taylor}}
for instance) implies
$$
\abs{\nabla_{x} G_{2}(k)(x,y)} \lesssim d(x,y)^{1 - N},
$$
uniformly in $k \in (0,1)$.
This, in combination with Proposition \ref{prop:Lemma2.7Higher},
implies
\begin{align*}\begin{split}  
 Y_{2}(k)(x) &\lesssim k^{-2(j - 2)} \norm{v}_{\infty} \int_{V} d(x,y)^{1 - N} \langle
 d(x_{i}^{\circ},y) \rangle^{2 - n_{i}} \exp(- c k d(x_{i}^{\circ},y))
 \, dy \\
 &\lesssim k^{-2(j - 1)} \norm{v}_{\infty},
\end{split}\end{align*}
which completes the proof of \eqref{eqtn:Y2} for $x \in V$.

\vspace*{0.1in}

\underline{\textit{The Term $Y_{3}(k)$.}}

\vspace*{0.1in}

Recall from \cite{hs2019} that the kernel $\nabla_{x} G_{3}(k)(x,y)$
satisfies the estimate
$$
\abs{\nabla_{x} G_{3}(k)(x,y)} \lesssim \omega_{1}(x,k) \omega_{2}(y,k).
$$
Proposition \ref{prop:Lemma2.7Higher} then implies,
\begin{align*}\begin{split}  
 Y_{3}(k)(x) &= \int_{\mathcal{M}} \abs{\nabla_{x} G_{3}(k)(x,y)}
 \abs{(\Delta + k^{2})^{-(j - 1)}v(y)} \, dy \\
 &\lesssim k^{-2(j - 2)} \norm{v}_{\infty} \int_{\mathcal{M}} \omega_{1}(x,k)
 \omega_{2}(y,k) \omega_{2}(y,k) \, dy \\
 &= k^{-2(j - 2)} \norm{v}_{\infty} \omega_{1}(x,k) \int_{\mathcal{M}}
 \omega_{2}(y,k)^{2} \, dy.
\end{split}\end{align*}
Therefore, in order to prove the desired estimate for $Y_{3}(k)$, it
is sufficient to show that
$$
\int_{\mathcal{M}} \omega_{2}(y,k)^{2} \, dy \lesssim k^{-2}.
$$
Considering the integral over the center $K$ first,
$$
 \int_{K} \omega_{2}(y,k)^{2} \, dy \lesssim \int_{K} 1 \, dy \lesssim
 1 \leq k^{-2}.
$$
On the ends $\R^{n_{i}} \xx \mathcal{M}_{i} \setminus K_{i}$ for $1
\leq i \leq l$ we have, 
\begin{align*}\begin{split}  
 \int_{\R^{n_{i}} \xx \mathcal{M}_{i} \setminus K_{i}}
 \omega_{2}(y,k)^{2} \, dy &\lesssim \int_{\R^{n_{i}} \xx
   \mathcal{M}_{i} \setminus K_{i}} \langle d(x_{i}^{\circ},y)
 \rangle^{4 - 2 n_{i}} \exp (- 2 c k d(x_{i}^{\circ}, y)) \, dy \\
 &\lesssim \int^{\infty}_{0} \langle r \rangle^{4 - 2 n_{i}} \exp(- 2 c
 k r) r^{n_{i} - 1} \, dr \\
&\lesssim k^{-1} \leq k^{-2}.
 \end{split}\end{align*}

\vspace*{0.1in}

\underline{\textit{The Term $Y_{4}(k)$.}}

\vspace*{0.1in}

This can be handled in an
identical manner to the term $Y_{3}(k)$ since $\nabla_{x} G_{4}(k)(x,y)$
satisfies even stronger estimates than $\nabla_{x} G_{3}(k)(x,y)$
(c.f. {\cite[Sec.~3]{hs2019}}).

 \end{proof}

In \cite{hs2019}, the decomposition \eqref{eqtn:ResolventSplitting} allowed for the Riesz transform to be
separated into four corresponding components. The
$L^{p}$-boundedness of each part was then proved independently. In this
article, we will also
make use of this decomposition. Notice that
$$
\br{\Delta + k^{2}}^{-M} = \frac{(-1)^{M - 1}}{(M - 1)!}  \partial_{k^{2}}^{(M-1)} (\Delta + k^{2})^{-1}.
$$
On combining this with the splitting \eqref{eqtn:ResolventSplitting},
we have the relation
\begin{equation}
  \label{eqtn:HighResolventSplit}
  \br{\Delta + k^{2}}^{-M} = \sum_{i = 1}^{4} H_{i}^{(M)}(k),
\end{equation}
where $H_{i}^{(M)}(k) := \frac{(-1)^{M - 1}}{(M - 1)!}  \partial_{k^{2}}^{(M
  - 1)}
G_{i}(k)$. To simplify notation, when the
integer $M > 0$ is understood, the shorthand notation $H_{i}(k)$ will
be employed.  For the remainder of this section we will investigate various properties of these operators
and, in particular, obtain asymptotic estimates for the kernels of
$H_{3}(k)$ and $H_{4}(k)$.

\begin{prop} 
  \label{prop:H3}
  The kernel of the operator $H_{3}(k)$ satisfies
  $$
\abs{H_{3}(k)(x,y)} \lesssim k^{-2(M - 1)} \omega_{2}(x,k) \omega_{2}(y,k)
$$
and
$$
\abs{\nabla_{x} H_{3}(k)(x,y)} \lesssim k^{-2(M - 1)} \omega_{1}(x,k) \omega_{2}(y,k),
$$
for all $x, \, y \in \mathcal{M}$ and $0 < k \leq 1$.
\end{prop}

\begin{proof}  
  We have
  \begin{align*}\begin{split}  
 &H_{3}(k)(x,y) = \frac{(-1)^{M - 1}}{(M - 1)!} \partial^{(M - 1)}_{k^{2}}
 G_{3}(k)(x,y) \\
 & \quad = \frac{(-1)^{M - 1}}{(M - 1)!} \partial^{(M - 1)}_{k^{2}} \sum_{i = 1}^{l}
 \br{\Delta_{\R^{n_{i}} \xx \mathcal{M}_{i}} +
   k^{2}}^{-1}(x_{i}^{\circ}, y) u_{i}(x,k) \phi_{i}(y) \\
 & \quad = \frac{(-1)^{M - 1}}{(M - 1)!} \sum_{i = 1}^{l} \sum_{j = 0}^{M - 1} \br{\begin{array}{c} 
 M - 1 \\ j
 \end{array}} \partial^{(j)}_{k^{2}} (\Delta_{\R^{n_{i}} \xx
   \mathcal{M}_{i}} + k^{2})^{-1}(x_{i}^{\circ}, y) \cdot u_{i}^{(M - 1 -
 j)}(x,k) \phi_{i}(y) \\
& \quad =   \sum_{i = 1}^{l} \sum_{j = 0}^{M - 1} \frac{(-1)^{M + j -1}}{(M - j - 1)!}  (\Delta_{\R^{n_{i}} \xx
 \mathcal{M}_{i}} + k^{2})^{-(j+1)}(x_{i}^{\circ},y) u_{i}^{(M - 1 - j)}(x,k) \phi_{i}(y).
\end{split}\end{align*}
Proposition \ref{prop:Lemma2.7Higher} then implies that
\begin{align*}\begin{split}  
 \abs{H_{3}(k)(x,y)} &\lesssim \sum_{i =1}^{l} \sum_{j = 0}^{M - 1}
 \br{\Delta_{\R^{n_{i}} \xx \mathcal{M}_{i}} +
   k^{2}}^{-(j+1)}(x_{i}^{\circ},y) \abs{u_{i}^{(M - 1 - j)}(x,k)}
 \phi_{i}(y) \\
 &\lesssim \omega_{2}(x,k)  \sum_{j = 0}^{M - 1}
 k^{-2(M - 1 - j)} \sum_{i = 1}^{l} \br{\Delta_{\R^{n_{i}} \xx \mathcal{M}_{i}} +
   k^{2}}^{-(j + 1)}(x_{i}^{\circ},y) \phi_{i}(y).
\end{split}\end{align*}
Proposition \ref{prop:EndResolvent}, when combined with the fact that
$ \mathrm{supp} \ \phi_{i} \subset \R^{n_{i}} \xx \mathcal{M}_{i}
\setminus K_{i}$ and $x_{i}^{\circ}$ is in the interior of $K_{i}$,
leads to
$$
\abs{H_{3}(k)(x,y)} \lesssim \omega_{2}(x,k) \omega_{2}(y,k) \sum_{j =
0}^{M - 1} k^{-2(M - 1 - j)} k^{-2 j} \simeq k^{-2(M - 1)} \omega_{2}(x,k) \omega_{2}(y,k),
$$
which proves our claim. The estimate for $\nabla_{x} H_{3}(k)(x,y)$
follows via an identical argument, except the use of Proposition
\ref{prop:Lemma2.7Higher} must be replaced with Proposition \ref{prop:HigherOrder}.
\end{proof}

\begin{lem} 
 \label{lem:ErrorTerm} 
 For any $j \in \N$, all $0 < k \leq 1$ and $y \in \mathcal{M}$,
 $$
\norm{\partial_{k^{2}}^{(j)} E(k) (\cdot, y)}_{\infty} \lesssim k^{-2
  j} \omega_{1}(y,k).
 $$
\end{lem}

\begin{proof}
  This lemma has already been proved for the case $j = 0$ in
  \cite{hs2019}, and so it can be assumed that $j > 0$.
From the splitting \eqref{eqtn:ErrorSplitting}, it is 
sufficient to prove this bound for the kernels of each of the components $E_{1}^{i}(k)$,
$E_{2}^{i}(k)$ for $i = 1, \cdots, l$ and $E_{3}(k)$.

\vspace*{0.1in}

\underline{\textit{The Error Term $E^{i}_{1}(k)$.}}

\vspace*{0.1in}

On differentiating the expression for $E^{i}_{1}$ with respect to $k^{2}$,
$$
 \partial^{(j)}_{k^{2}} E^{i}_{1}(k)(x,y) = - 2 \nabla \phi_{i}(x)
 \phi_{i}(y) \brs{(-1)^{j} j! \nabla_{x} \br{\Delta_{\R^{n_{i}} \xx
       \mathcal{M}_{i}} + k^{2}}^{-(j + 1)}(x,y) - \nabla_{x}
   \partial^{(j)}_{k^{2}} G_{int}(k)(x,y)}.
 $$
 $G_{int}(k)$ agrees with the resolvent $\br{\Delta_{\R^{n_{i}}
     \xx \mathcal{M}_{i}} + k^{2}}^{-1}$ near the diagonal and on the support of $\nabla \phi_{i}(x) \cdot
 \phi_{i}(y)$. Therefore $E_{1}^{i}$ will vanish on this set.
 
For
 $(x,y)$ away from the diagonal and on the support of $\nabla
 \phi_{i}(x) \phi_{i}(y)$, we have
   \begin{align}\begin{split}
          \label{eqtn:Error}
 \abs{\partial^{(j)}_{k^{2}} E^{i}_{1}(k)(x,y)}
 &\lesssim \abs{\nabla_{x} \br{\Delta_{\R^{n_{i}} \xx
       \mathcal{M}_{i}} + k^{2}}^{-(j + 1)}(x, y)} +
 \abs{\nabla_{x} \partial^{(j)}_{k^{2}} G_{int}(k)(x,y)}
 \\
 &\lesssim k^{-2 j} \omega_{1}(y,k) + \abs{\nabla_{x}
   \partial^{(j)}_{k^{2}} G_{int}(k)(x,y)},
\end{split}\end{align}
where the last line follows from Proposition
\ref{prop:EndResolvent}.
The operators $\lb \nabla
\partial^{(j)}_{k^{2}} G_{int}(k) \rb_{k \in (0,1)}$ constitute a
family of pseudodifferential operators of order $-1-2 j$. At the scale of
an individual chart on $\mathcal{M}$, it is evident that in local
coordinates the symbol of $\nabla \partial^{(j)}_{k^{2}} G_{int}(k)$,
denoted $a_{int}(k)(x,\xi)$, will satisfy
$$
\abs{\partial^{\alpha}_{\xi} a_{int}(k)(x,\xi)} \lesssim
(\abs{\xi}^{2} + k^{2})^{\frac{-1 - 2 j - \abs{\alpha}}{2}}
$$
for all multi-indices $\alpha \geq 0$ and $k \in (0,1)$. From this,
standard pseudodifferential operator theory
(c.f. {\cite[Sec.~0.2]{taylor}}) implies that for any $b > -1 - 2 j + N$,
$$
\abs{\nabla_{x} \partial^{(j)}_{k^{2}}G_{int}(k)(x,y)} \lesssim k^{N
  -1 - 2 j - b} d(x,y)^{-b}
$$
uniformly in $k \in (0,1)$, for all $x, \, y \in \mathcal{M}$. Setting $b = N - 1$ gives
  $$
\abs{\nabla_{x} \partial^{(j)}_{k^{2}}G_{int}(k)(x,y)} \lesssim
k^{-2 j} d(x,y)^{1 - N}
$$
uniformly in $k \in (0,1)$, for all $x, \, y \in \mathcal{M}$. As we
are considering $(x,y)$ away from the diagonal and
$G_{int}(k)(x,y)$ is compactly supported in $\mathcal{M}^{2}$, this estimate implies
\begin{align*}\begin{split}  
 \abs{\nabla_{x} \partial^{(j)}_{k^{2}} G_{int}(k)(x,y)} &\lesssim
 k^{-2 j} \langle d(x,y) \rangle^{1 - N}
 \\
 &\lesssim k^{-2 j} \omega_{1}(y,k),
\end{split}\end{align*}
thereby implying our desired estimate.

\vspace*{0.1in}

\underline{\textit{The Error Term $E^{i}_{2}(k)$.}}

\vspace*{0.1in}

For the second term,
\begin{align*}\begin{split}  
 \partial^{(j)}_{k^{2}} E^{i}_{2}(k)(x,y)&= \phi_{i}(y) v_{i}(x)
 \left[ (-1)^{j} j! \br{\Delta_{\R^{n_{i}} \xx \mathcal{M}_{i}} +
     k^{2}}^{-(j + 1)}(x,y) + \partial^{(j)}_{k^{2}} G_{int}(k)(x,y) \right.
   \\
   & \qquad  \left. + (-1)^{j} j! \br{\Delta_{\R^{n_{i}} \xx
         \mathcal{M}_{i}} + k^{2}}^{-(j + 1)}(x_{i}^{\circ},y) \right].
 \end{split}\end{align*}
The first two terms will vanish on the diagonal and
on the support of $\phi_{i}(y) v_{i}(x)$, since this is where
$G_{int}(k)$ coincides with the resolvent. The desired estimate would
then follow from Proposition \ref{prop:EndResolvent}.

For $(x,y)$ away from the diagonal and on the support of $\phi_{i}(y) v_{i}(x)$, 
\begin{align*}\begin{split}  
 \abs{\partial_{k^{2}}^{(j)} E^{i}_{2}(k)(x,y)} &\lesssim
\abs{\br{\Delta_{\R^{n_{i}} \xx \mathcal{M}_{i}} +
    k^{2}}^{-(j + 1)}(x_{i}^{\circ},y) - \br{\Delta_{\R^{n_{i}} \xx
      \mathcal{M}_{i}} + k^{2}}^{-(j + 1)}(x,y)} \\ & \qquad \qquad + \abs{\partial^{(j)}_{k^{2}}
  G_{int}(k)(x,y)} \\
&\lesssim \abs{x_{i}^{\circ} - x} \abs{\nabla_{x}
  \br{\Delta_{\R^{n_{i}} \xx \mathcal{M}_{i}} + k^{2}}^{-(j + 1)}(\tilde{x},y)} + \abs{\partial^{(j)}_{k^{2}}
  G_{int}(k)(x,y)}  \\
 &\lesssim k^{-2 j} \omega_{1}(y,k) + \abs{\partial^{(j)}_{k^{2}}
  G_{int}(k)(x,y)},
\end{split}\end{align*}
for some $\tilde{x}$ close to the set $K_{i}$, 
where the last line follows from Proposition
\ref{prop:EndResolvent}. Similar pseudodifferential reasoning as for
the term $E_{1}^{i}(k)$ then implies $\abs{\partial^{(j)}_{k^{2}} G_{int}(k)(x,y)}
\lesssim k^{-2 j} \omega_{1}(y,k)$, thereby proving the desired estimate.

\vspace*{0.1in}

\underline{\textit{The Error Term $E_{3}(k)$.}}

\vspace*{0.1in}

Finally, from the form of $\partial_{k^{2}}^{(j)}
E_{3}(k)$, it is clear that in an analogous manner to the previous two
terms, pseudodifferential reasoning can
be applied to obtain the estimate
$\norm{\partial^{(j)}_{k^{2}}E_{3}(k)(\cdot,y)}_{\infty} \lesssim k^{-2 j} \omega_{1}(y,k)$.
 \end{proof}

\begin{prop} 
  \label{prop:H4}
    The kernel of the operator $H_{4}(k)$ satisfies
 $$
\abs{H_{4}(k)(x,y)} \lesssim k^{-2(M - 1)} \omega_{2}(x,k) \omega_{1}(y,k)
$$
and
$$
\abs{\nabla_{x} H_{4}(k)(x,y)} \lesssim k^{-2(M - 1)} \omega_{1}(x,k) \omega_{1}(y,k)
$$
for all $x, \, y \in \mathcal{M}$ and $0 < k \leq 1$.
\end{prop}

\begin{proof}
  Recall that
  \begin{align*}\begin{split}  
      G_{4}(k)(x,y) &= - (\Delta + k^{2})^{-1}v_{y}(x) \\
      &= - \int^{\infty}_{0} e^{-t k^{2}} e^{-t \Delta}v_{y}(x) \, dt,
    \end{split}\end{align*}
  where $v_{y}(x) := E(k)(x,y)$.
  If we define $v^{(j)}_{y} := \partial_{k^{2}}^{(j)} v_{y}$ for $j
  \in \N$, then
  Lemma \ref{lem:ErrorTerm} implies
  $$
\norm{v^{(j)}_{y}}_{\infty} =
\norm{\partial_{k^{2}}^{(j)}E(k)(\cdot,y)}_{\infty} \lesssim k^{-2 j} \omega_{1}(y,k).
$$
Therefore $e^{-t \Delta} v_{y}^{(j)}$ is well-defined and we have by
the dominated convergence theorem
$$
\partial_{k^{2}}^{(j)} e^{-t \Delta} v_{y}(x) = e^{-t \Delta} v_{y}^{(j)}(x).
$$
Thus
\begin{align*}\begin{split}
    \label{eqtn:H4Identity}
 \partial_{k^{2}}^{(M - 1)} G_{4}(k)(x,y) &= \partial_{k^{2}}^{(M -
   1)} \int^{\infty}_{0} e^{-t k^{2}} e^{-t \Delta}v_{y}(x) \, dt \\
 &= \int^{\infty}_{0} \partial^{(M - 1)}_{k^{2}} \br{e^{-t k^{2}}
   e^{-t \Delta}v_{y}(x)} \, dt \\
 &= \sum_{j = 0}^{M - 1} \br{ \begin{array}{c} M - 1 \\ j \end{array}}
(-1)^{j} \int^{\infty}_{0} t^{j} e^{-t k^{2}}  e^{-t \Delta}v_{y}^{(M - 1 -
   j)}(x) \, dt.
\end{split}\end{align*}
The integral identity \eqref{eqtn:jthDerivative} then implies
\begin{equation}
  \label{eqtn:H4Identity}
 \partial_{k^{2}}^{(M - 1)} G_{4}(k)(x,y) =  \sum_{j = 0}^{M - 1} \br{ \begin{array}{c} M - 1 \\ j \end{array}}
 (\Delta + k^{2})^{-(j+ 1)}v_{y}^{(M - 1 - j)}(x).
  \end{equation}
An application of Proposition \ref{prop:Lemma2.7Higher} leads to,
\begin{align*}\begin{split}  
 \abs{\partial_{k^{2}}^{(M - 1)} G_{4}(k)(x,y)} &\lesssim \sum_{j = 0}^{M -
   1} k^{-2 j} \omega_{2}(x,k) \norm{v_{y}^{(M - 1 - j)}}_{\infty} \\
 &\lesssim \omega_{2}(x,k) \omega_{1}(y,k) \sum_{j = 0}^{M - 1} k^{-2
   j} k^{-2 (M - 1 - j)} \\
 &\simeq k^{-2(M - 1)} \omega_{2}(x,k) \omega_{1}(y,k).
\end{split}\end{align*}
The definition $H_{4}(k) := \frac{(-1)^{M - 1}}{(M - 1)!}
\partial_{k^{2}}^{(M - 1)} G_{4}(k)$ then allows us to immediately
obtain
$$
\abs{H_{4}(k)(x,y)} \lesssim k^{-2(M - 1)} \omega_{2}(x,k) \omega_{1}(y,k).
$$
Estimates for $\abs{\nabla_{x} H_{4}(k)(x,y)}$ follow in an identical
manner from \eqref{eqtn:H4Identity}, except the use of Proposition
\ref{prop:Lemma2.7Higher} must be replaced with an application of
Proposition \ref{prop:HigherOrder}.
 \end{proof}

 \section{The Low Energy Square Function}
 \label{sec:Low}

 Recall from Section \ref{sec:Resolvent} that the higher-order
 resolvent has the splitting
 $$
(\Delta + k^{2})^{-M} = \sum_{i = 1}^{4} H_{i}(k),
$$
where $H_{i}(k) := \frac{(-1)^{M - 1}}{(M - 1)!}
\partial_{k^{2}}^{(M - 1)} G_{i}(k)$. The low energy square function can
then be controlled from above by
$$
S_{<}f(x) \lesssim \sum_{i = 1}^{4} S_{<}^{i}f(x),
$$
where $S_{<}^{i}$ is the part of the low energy square function corresponding to $H_{i}(k)$,
$$
S_{<}^{i}f(x) := \br{\int^{1}_{0} \abs{\nabla H_{i}(k) f(x)}^{2} k^{4
    M - 3} \, dk}^{\frac{1}{2}}.
$$
The $L^{p}$-boundedness for $p \in (1,n_{min})$ and weak-type $(1,1)$ property of $S_{<}$ will be proved by demonstrating that each
component $S_{<}^{i}$ is itself bounded on $L^{p}$ and weak-type $(1,1)$ for $i
= 1, \cdots, 4$.

\subsection{The Operator $S_{<}^{1}$}
\label{sec:S1}

From the definition of $H_{1}(k)$, we have
\begin{align*}\begin{split}  
H_{1}(k)(x,y)  &= \frac{(-1)^{M - 1}}{(M - 1)!} \partial^{(M - 1)}_{k^{2}}
\sum_{i = 1}^{l} \br{\Delta_{\R^{n_{i}} \xx \mathcal{M}_{i}} +
  k^{2}}^{-1}(x,y) \phi_{i}(x) \phi_{i}(y) \\
&= \sum_{i = 1}^{l} (\Delta_{\R^{n_{i}} \xx \mathcal{M}_{i}} +
k^{2})^{-M}(x,y) \phi_{i}(x) \phi_{i}(y).
\end{split}\end{align*}
As $H_{1}(k)$ consists of finitely many terms, in order to prove the
boundedness of the operator $S^{1}_{<}$, we need only prove
boundedness of the operators
$$
S^{1,i}_{<}f(x) := \br{\int^{1}_{0} \abs{\int_{\mathcal{M}} \nabla_{x}
  \brs{(\Delta_{\R^{n_{i}} \xx \mathcal{M}_{i}} + k^{2})^{-M}(x,y)
    \phi_{i}(x) \phi_{i}(y)} f(y) \, dy}^{2} k^{4 M - 3} \, dk}^{\frac{1}{2}}
$$
for all $1 \leq i \leq l$. From the product rule, this operator is in turn
controlled by
$$
S^{1,i}_{<}f(x) \leq \Lambda^{i}f(x) + \Pi^{i}f(x),
$$
where
$$
\Lambda^{i}f(x) := \br{\int^{1}_{0} \abs{\int_{\mathcal{M}} \nabla_{x}
  (\Delta_{\R^{n_{i}} \xx \mathcal{M}_{i}} + k^{2})^{-M}(x,y)
  \phi_{i}(x) \phi_{i}(y) f(y) \, dy}^{2} k^{4 M - 3} \, dk}^{\frac{1}{2}}
$$
and
$$
\Pi^{i}f(x) := \br{\int^{1}_{0} \abs{\int_{\mathcal{M}} 
  (\Delta_{\R^{n_{i}} \xx \mathcal{M}_{i}} + k^{2})^{-M}(x,y)
  \nabla \phi_{i}(x) \phi_{i}(y) f(y) \, dy}^{2} k^{4 M - 3} \, dk}^{\frac{1}{2}}.
$$
First consider the operator $\Lambda^{i}$. 
It is well-known from classical theory that the square function
$$
\br{\int^{\infty}_{0} \abs{\nabla (\Delta_{\R^{n_{i}} \xx
      \mathcal{M}_{i}} + k^{2})^{-M}g(x)}^{2} k^{4 M -
    3} \, dk}^{\frac{1}{2}} = \br{\int^{\infty}_{0} \abs{t \nabla (t^{2}\Delta_{\R^{n_{i}} \xx
      \mathcal{M}_{i}} + I)^{-M}g(x)}^{2} \, \frac{dt}{t}}^{\frac{1}{2}}
$$
is bounded on $L^{p}(\R^{n_{i}} \xx \mathcal{M}_{i})$ for all $p \in
(1,\infty)$ and weak-type $(1,1)$. Therefore,
\begin{align*}\begin{split}  
 \norm{\Lambda_{i}f}_{p} &= \norm{\phi_{i} \cdot \br{\int^{1}_{0} \abs{\nabla
  (\Delta_{\R^{n_{i}} \xx \mathcal{M}_{i}} + k^{2})^{-M} (\phi_{i}
  \cdot f)}^{2} k^{4 M - 3} \, dk}^{\frac{1}{2}}}_{p} \\
&\leq  \norm{\br{\int^{\infty}_{0} \abs{\nabla
  (\Delta_{\R^{n_{i}} \xx \mathcal{M}_{i}} + k^{2})^{-M} (\phi_{i}
  \cdot f)}^{2} k^{4 M - 3} \, dk}^{\frac{1}{2}}}_{p} \\
&\lesssim \norm{f}_{p}
\end{split}\end{align*}
for any $p \in (1,\infty)$. Weak-type $(1,1)$ bounds for $\Lambda_{i}$
follow in the same manner. 

Next, consider the operator $\Pi^{i}$. For $r > 0$, define the set $D_{r} := \lb (x,y) \in \mathcal{M}^{2} :
d(x,y) \leq r \rb$.
 Minkowski's integral inequality
implies that
\begin{align*}\begin{split}  
 \Pi^{i}f(x) &\leq \int_{\mathcal{M}} \br{\int^{1}_{0}
   \abs{(\Delta_{\R^{n_{i}} \xx \mathcal{M}_{i}} + k^{2})^{-M}(x,y)
     \nabla \phi_{i}(x) \phi_{i}(y)}^{2} k^{4 M - 3} \,
   dk}^{\frac{1}{2}} \abs{f(y)} \, dy \\
 &= \int_{\mathcal{M}} \abs{\nabla \phi_{i}(x)} \br{\int^{1}_{0}
   \abs{(\Delta_{\R^{n_{i}} \xx \mathcal{M}_{i}} +
     k^{2})^{-M}(x,y)}^{2} k^{4 M - 3} \, dk}^{\frac{1}{2}}
 \phi_{i}(y) \abs{f(y)} \, dy \\
 &\leq \int_{\mathcal{M}} \pi_{1}^{i}(x,y) \abs{f(y)} \, dy +
 \int_{\mathcal{M}} \pi_{2}^{i}(x,y) \abs{f(y)} \, dy \\
 &=: \Pi_{1}^{i}f(x) + \Pi_{2}^{i}f(x),
\end{split}\end{align*}
where $\pi_{1}^{i}(x,y)$ and $\pi_{2}^{i}(x,y)$ are the kernels
defined by
$$
\pi_{1}^{i}(x,y) := \abs{\nabla \phi_{i}(x)} \chi_{D_{r}}(x,y) \br{\int^{1}_{0}
  \abs{(\Delta_{\R^{n_{i}} \xx \mathcal{M}_{i}} +
    k^{2})^{-M}(x,y)}^{2} k^{4 M - 3} \, dk}^{\frac{1}{2}} \phi_{i}(y)
$$
and
$$
\pi_{2}^{i}(x,y) := \abs{\nabla \phi_{i}(x)} \abs{1 - \chi_{D_{r}}(x,y)} \br{\int^{1}_{0}
  \abs{(\Delta_{\R^{n_{i}} \xx \mathcal{M}_{i}} +
    k^{2})^{-M}(x,y)}^{2} k^{4 M - 3} \, dk}^{\frac{1}{2}} \phi_{i}(y).
$$
Let's prove that $\Pi_{1}^{i}$ is $L^{p}$-bounded for all $p \in [1,\infty]$.
From Proposition \ref{prop:EndResolvent},
\begin{align*}\begin{split}  
 &\abs{\nabla \phi_{i}(x)} \chi_{D_{r}}(x,y) \br{\int^{1}_{0} \abs{ (\Delta_{\R^{n_{i}} \xx
       \mathcal{M}_{i}} + k^{2})^{-M}(x,y))^{2}} k^{4 M - 3} \,
   dk}^{\frac{1}{2}} \phi_{i}(y) \\ & \hspace*{2in} \lesssim \chi_{D_{r}}(x,y) \br{\int^{1}_{0}
   \abs{d(x,y)^{2 - N} \exp \br{- c k d(x,y)}}^{2} k \,
   dk}^{\frac{1}{2}} \\
 & \hspace*{2in} = \chi_{D_{r}}(x,y) d(x,y)^{2 - N} \br{\int^{1}_{0} k \exp\br{- 2 c
     k d(x,y)} \, dk}^{\frac{1}{2}} \\
 & \hspace*{2in} \lesssim \chi_{D_{r}}(x,y) d(x,y)^{1 - N}.
\end{split}\end{align*}
It is obvious that an operator with this kernel will be bounded on
$L^{p}$ for all $p \in [1,\infty]$ since the local decay $d(x,y)^{1 -
  N}$ is stronger than $d(x,y)^{-N}$, and globally it is cut off by the
function $\chi_{D_{r}}$. Therefore $\Pi_{1}^{i}$ is bounded on $L^{p}$
for all $p \in [1,\infty]$.

\vspace*{0.1in}

Next, consider the operator $\Pi_{2}^{i}$.
From an application of H\"{o}lder's inequality for $p \geq 1$,
\begin{align*}\begin{split}  
 \norm{\Pi^{i}_{2}f}_{p}^{p} &= \int_{\mathcal{M}}
 \br{\int_{\mathcal{M}} \pi_{2}^{i}(x,y) \abs{f(y)} \, dy}^{p} \, dx
 \\
 &\leq \int_{\mathcal{M}}
 \norm{\pi^{i}_{2}(x,\cdot)}^{p}_{L^{p'}(\mathcal{M})} \, dx \cdot \norm{f}^{p}_{L^{p}(\mathcal{M})}.
\end{split}\end{align*}
On applying Proposition \ref{prop:EndResolvent} for $x, \, y \in
\R^{n_{i}} \xx \mathcal{M}_{i} \setminus K_{i}$ with $d(x,y) > r$,
\begin{align*}\begin{split}  
 \br{\int^{1}_{0} \abs{(\Delta_{\R^{n_{i}} \xx \mathcal{M}_{i}} +
     k^{2})^{-M}(x,y)}^{2} k^{4 M - 3} \, dk}^{\frac{1}{2}} &\lesssim
 d(x,y)^{2 - n_{i}} \br{\int^{1}_{0} k \exp(- 2 c k d(x,y)) \,
   dk}^{\frac{1}{2}} \\
 &\lesssim d(x,y)^{1 - n_{i}}.
\end{split}\end{align*}
It then follows from the fact that $\nabla \phi_{i}$ is compactly supported,
\begin{align*}\begin{split}  
 \int_{\mathcal{M}}
 \norm{\pi^{i}_{2}(x,\cdot)}_{L^{p'}(\mathcal{M})}^{p} \, dx &\lesssim
 \int_{\mathrm{supp} \, \nabla \phi_{i}} \norm{d(x,\cdot)^{1 -
     n_{i}}}_{L^{p'}(B(x,r)^{c})}^{p} \, dx \\
 &\lesssim \sup_{x \in \mathrm{supp} \, \nabla \phi_{i}}
 \norm{d(x,\cdot)^{1 - n_{i}}}_{L^{p'}(B(x,r)^{c})}^{p}.
\end{split}\end{align*}
For $p = 1$, this quantity is obviously finite and thus $\Pi_{2}^{i}$
is bounded on $L^{1}$. For $p > 1$,
$$
\norm{d(x,\cdot)^{1 - n_{i}}}_{L^{p'}(B(x,r)^{c})} = \br{\int_{d(x,y) > r}
  d(x,y)^{(1 - n_{i})p'} \, dy}^{\frac{1}{p'}}.
$$
This quantity will be bounded from above by a constant, uniformly in
$x$, provided that
 $$
p' (1 - n_{i}) < -n_{i} \quad \Leftrightarrow \quad p' >
\frac{n_{i}}{n_{i} - 1} = n_{i}' \quad \Leftrightarrow \quad p < n_{i}.
$$
This proves that $\Pi^{i}_{2}$ is bounded on $L^{p}$  for $1 \leq p <
n_{i}$. This completes the
proof of the $L^{p}$-boundedness of $\Pi^{i}$ for all $p \in [1,n_{i})$,
thereby demonstrating that the operator $S_{<}^{1}$ is 
$L^{p}$-bounded for $p \in (1,n_{min})$ and weak-type $(1,1)$.

\subsection{The Operator $S_{<}^{2}$}
\label{subsec:S2}

The operator $S^{2}_{<}$ is given by
$$
S^{2}_{<}f(x) = \br{\int^{1}_{0} \abs{\int_{\mathcal{M}} \nabla_{x}
    H_{2}(k)(x,y) f(y) \, dy}^{2} k^{4 M - 3} \, dk}^{\frac{1}{2}}.
$$
Minkowski's integral inequality implies that
\begin{align*}\begin{split}  
 S^{2}_{<}f(x) &\leq \int_{\mathcal{M}} \br{\int^{1}_{0} \abs{\nabla_{x}
     H_{2}(k)(x,y)}^{2} k^{4M - 3} \, dk}^{\frac{1}{2}} \abs{f(y)} \, dy \\
&= \int_{\mathcal{M}} h_{2}(x,y) \abs{f(y)} \, dy,
 \end{split}\end{align*}
where $h_{2}(x,y)$ is the kernel
$$
h_{2}(x,y) := \br{\int^{1}_{0} \abs{\nabla_{x} H_{2}(k)(x,y)}^{2}
  k^{4M - 3} \, dk}^{\frac{1}{2}}.
$$
First observe that the kernel $h_{2}$ is compactly supported in
$\mathcal{M}^{2}$. Next, notice that the operators $\lb \nabla H_{2}(k) \rb_{k \in (0,1)}$ constitute a
family of pseudodifferential operators of order $1 - 2 M$. At the
scale of an individual chart on $\mathcal{M}$, it
is evident that in local coordinates the symbol of $\nabla H_{2}(k)$, denoted
$a_{2}(k)(x,\xi)$, will satisfy
$$
\abs{\partial^{\alpha}_{\xi}a_{2}(k)(x,\xi)} \lesssim (\abs{\xi}^{2} +
k^{2})^{\frac{1 - 2 M - \abs{\alpha}}{2}}
$$
for all multi-indices $\alpha \geq 0$ and $k \in (0,1)$. From this,
standard pseudodifferential operator theory (c.f. {\cite[Sec.~0.2]{taylor}} for instance)
implies
$$
\abs{\nabla_{x}H_{2}(k)(x,y)} \lesssim k^{N + 1 - 2 M - b}
d(x,y)^{-b},
$$
for any $b > N + 1 - 2 M$. Setting $b = N - \frac{1}{2}$ then yields
$$
\abs{\nabla_{x}H_{2}(k)(x,y)} \lesssim k^{\frac{3}{2} - 2 M}
d(x,y)^{\frac{1}{2} - N}.
$$
Therefore,
\begin{align*}\begin{split}  
 h_{2}(x,y) &= \br{\int^{1}_{0} \abs{\nabla_{x} H_{2}(k)(x,y)}^{2}
   k^{4M - 3} \, dk}^{\frac{1}{2}} \\
 &\lesssim \br{\int^{1}_{0} dk}^{\frac{1}{2}} d(x,y)^{\frac{1}{2} - N}
 \\
 &= d(x,y)^{\frac{1}{2} - N}.
\end{split}\end{align*}
Thus $S_{<}^{2}$ is pointwise bounded from above by an operator with
kernel that is compactly supported in $\mathcal{M}^{2}$ and controlled by
$d(x,y)^{\frac{1}{2} - N}$. It follows immediately that $S_{<}^{2}$ must be
bounded on $L^{p}$ for all $p \in [1,\infty]$.

\subsection{The Operator $S_{<}^{3}$}
\label{subsec:S3}

In an identical manner to the operator $S_{<}^{2}$, Minkowski's integral
inequality allows us to control $S_{<}^{3}$ from above by
$$
S^{3}_{<}f(x) \leq \int_{\mathcal{M}} h_{3}(x,y) \abs{f(y)} \, dy,
$$
where $h_{3}(x,y)$ is the kernel defined through
$$
h_{3}(x,y) := \br{\int^{1}_{0} \abs{\nabla_{x} H_{3}(k)(x,y)}^{2} k^{4 M -
  3} \, dk}^{\frac{1}{2}}.
$$
H\"{o}lder's inequality implies that
\begin{align*}\begin{split}  
 \norm{S^{3}_{<}f}^{p}_{p} &\leq \int_{\mathcal{M}}
 \br{\int_{\mathcal{M}} h_{3}(x,y) \abs{f(y)} \, dy}^{p} \, dx \\
 &\leq \int_{\mathcal{M}} \br{\int_{\mathcal{M}} h_{3}(x,y)^{p'} \,
   dy}^{\frac{p}{p'}} \, dx \cdot \norm{f}^{p}_{p}.
\end{split}\end{align*}
Thus if it can be proved that
\begin{equation}
  \label{eqtn:KernelCondition}
\int_{\mathcal{M}} \br{\int_{\mathcal{M}} h_{3}(x,y)^{p'} \,
  dy}^{\frac{p}{p'}} \, dx < \infty
\end{equation}
then the operator $S^{3}_{<}$ will be bounded on $L^{p}$.
In order to
prove \eqref{eqtn:KernelCondition} it is  sufficient to prove the below
four separate conditions,
\begin{equation}
  \label{eqtn:KernelCondition1}
  \int_{K} \br{\int_{K} h_{3}(x,y)^{p'} \, dy}^{\frac{p}{p'}} \, dx < \infty,
\end{equation}
\begin{equation}
  \label{eqtn:KernelCondition2}
  \int_{\R^{n_{i}} \xx \mathcal{M}_{i} \setminus K_{i}} \br{\int_{K}
    h_{3}(x,y)^{p'} \, dy}^{\frac{p}{p'}} \, dx < \infty \quad for \
  any \ 1 \leq i \leq l,
\end{equation}
\begin{equation}
  \label{eqtn:KernelCondition3}
  \int_{K} \br{\int_{\R^{n_{j}} \xx \mathcal{M}_{j} \setminus K_{j}}
    h_{3}(x,y)^{p'} \, dy}^{\frac{p}{p'}} \, dx < \infty \quad for \
  any \ 1 \leq j \leq l,
\end{equation}
and
\begin{equation}
  \label{eqtn:KernelCondition4}
  \int_{\R^{n_{i}} \xx \mathcal{M}_{i} \setminus K_{i}} \br{\int_{\R^{n_{j}} \xx \mathcal{M}_{j} \setminus K_{j}}
    h_{3}(x,y)^{p'} \, dy}^{\frac{p}{p'}} \, dx < \infty \quad for \
  any \ 1 \leq i, \, j \leq l.
\end{equation}
For the first estimate \eqref{eqtn:KernelCondition1}, observe that
Proposition \ref{prop:H3} implies $\nabla_{x} H_{3}(k)(x,y) \lesssim k^{2 -
  2M}$ for all $x, \, y \in \mathcal{M}$. This estimate, when applied
to the definition of $h_{3}(x,y)$, gives
\begin{equation}
  \label{eqtn:h31}
h_{3}(x,y) \lesssim 1 \quad \forall \ x, \, y \in \mathcal{M},
\end{equation}
which immediately implies the validity of
\eqref{eqtn:KernelCondition1}.

Let us consider
\eqref{eqtn:KernelCondition2}. According to Proposition \ref{prop:H3},
for $x \in \R^{n_{i}} \xx \mathcal{M}_{i} \setminus K_{i}$ and $y \in K$,
$$
\abs{\nabla_{x} H_{3}(k)(x,y)} \lesssim k^{2 - 2 M} \langle
d(x_{i}^{\circ}, x) \rangle^{1 - n_{i}} \exp \br{- c k d(x_{i}^{\circ},x)}.
$$
This implies that
\begin{align}\begin{split}
    \label{eqtn:h32}
 h_{3}(x,y) &= \br{\int^{1}_{0} \abs{\nabla_{x} H_{3}(k)(x,y)}^{2} k^{4 M -
     3} \, dk}^{\frac{1}{2}} \\
 &\lesssim \langle d(x_{i}^{\circ},x) \rangle^{1 - n_{i}} \br{\int^{1}_{0} k
   \exp(- 2 c k d(x_{i}^{\circ},x)) \, dk}^{\frac{1}{2}} \\
 &\simeq \langle d(x_{i}^{\circ}, x) \rangle^{1 - n_{i}}
 \br{\frac{1}{d(x_{i}^{\circ},x)^{2}}}^{\frac{1}{2}} \\
 &\simeq \langle d(x_{i}^{\circ}, x) \rangle^{-n_{i}}.
\end{split}\end{align}
On applying this estimate to \eqref{eqtn:KernelCondition2},
$$
 \int_{\R^{n_{i}} \xx \mathcal{M}_{i} \setminus K_{i}} \br{\int_{K}
   h_{3}(x,y)^{p'} \, dy}^{\frac{p}{p'}} \, dx \lesssim
 \int_{\R^{n_{i}} \xx \mathcal{M}_{i} \setminus K_{i}} \langle
 d(x_{i}^{\circ}, x) \rangle^{-n_{i} p} \, dx.
$$
This will clearly be finite provided that $p > 1$.

Let us now consider \eqref{eqtn:KernelCondition3}. For $x \in K$ and
$y \in \R^{n_{j}} \xx \mathcal{M}_{j} \setminus K_{j}$, Proposition \ref{prop:H3} tells us that
$$
\abs{\nabla_{x} H_{3}(k)(x,y)} \lesssim k^{2 - 2 M} \langle d(x_{j}^{\circ},y)
\rangle^{2 - n_{j}} \exp(- c k d(x_{j}^{\circ}, y)).
$$
This leads to the pointwise estimate
\begin{align*}\begin{split}  
 h_{3}(x,y) &\lesssim  \langle d(x_{j}^{\circ},y) \rangle^{2 - n_{j}}
 \br{\int^{1}_{0} k \exp(- 2 c k d(x_{j}^{\circ},y)) \,
   dk}^{\frac{1}{2}} \\
 &\lesssim \langle d(x_{j}^{\circ},y) \rangle^{1 - n_{j}}.
\end{split}\end{align*}
Applying this to \eqref{eqtn:KernelCondition3},
$$
\int_{K} \br{\int_{\R^{n_{j}} \xx \mathcal{M}_{j} \setminus K_{j}}
  h_{3}(x,y)^{p'} \, dy}^{\frac{p}{p'}} \, dx \lesssim \int_{K}
\br{\int_{\R^{n_{j}} \xx \mathcal{M}_{j} \setminus K_{j}} \langle
  d(x_{j}^{\circ},y) \rangle^{(1- n_{j}) p'} \, dy}^{\frac{p}{p'}} \, dx.
$$
This will be finite provided that
$$
(n_{j} - 1) p' > n_{j} \quad \Leftrightarrow \quad p' >
\frac{n_{j}}{n_{j - 1}} = n_{j}' \quad \Leftrightarrow \quad p < n_{j}.
$$
Finally, it remains to prove estimate
\eqref{eqtn:KernelCondition4}. For $x \in \R^{n_{i}} \xx
\mathcal{M}_{i} \setminus K_{i}$ and $y \in \R^{n_{j}} \xx \mathcal{M}_{j}
\setminus K_{j}$, Proposition \ref{prop:H3} implies that
$$
\abs{\nabla_{x} H_{3}(k)(x,y)} \lesssim k^{2 - 2M} \langle
d(x_{i}^{\circ},x) \rangle^{1 - n_{i}} \langle d(x_{j}^{\circ},y)
\rangle^{2 - n_{j}} \exp (- c k (d(x_{i}^{\circ},x) + d(x_{j}^{\circ},y))).
$$
This leads to the pointwise bound
\begin{align}\begin{split}
    \label{eqtn:h33}
 \abs{h_{3}(x,y)} &\lesssim \langle d(x_{i}^{\circ},x) \rangle^{1 -
   n_{i}} \langle d(x_{j}^{\circ},y) \rangle^{2 - n_{j}}
 \br{\int^{1}_{0} k \exp(- 2 c k (d(x_{i}^{\circ},x) +
   d(x_{j}^{\circ},y))) \, dk}^{\frac{1}{2}} \\
 &\lesssim \langle d(x_{i}^{\circ},x) \rangle^{1 -
   n_{i}} \langle d(x_{j}^{\circ},y) \rangle^{2 - n_{j}}
 \br{\frac{1}{(d(x_{i}^{\circ},x) +
     d(x_{j}^{\circ},y))^{2}}}^{\frac{1}{2}} \\
 &\leq \min \br{\langle d(x_{i}^{\circ},x) \rangle^{-n_{i}} \langle
   d(x_{j}^{\circ},y) \rangle^{2 - n_{j}}, \langle d(x_{i}^{\circ},x)
   \rangle^{1 - n_{i}} \langle d(x_{j}^{\circ},y) \rangle^{1 - n_{j}}}.
\end{split}\end{align}
Applying this to \eqref{eqtn:KernelCondition4},
\begin{align*}\begin{split}  
 \int_{\R^{n_{i}} \xx \mathcal{M}_{i} \setminus K_{i}} &
 \br{\int_{\R^{n_{j}} \xx \mathcal{M}_{j} \setminus K_{j}} h_{3}(x,y)^{p'}
 \, dy}^{\frac{p}{p'}} \, dx  \\ &  \lesssim \int_{\R^{n_{i}} \xx
 \mathcal{M}_{i} \setminus K_{i}} \br{\int_{D_{j}^{1}(x)} \langle d(x_{i}^{\circ},x) \rangle^{(1 -
   n_{i})p'} \langle d(x_{j}^{\circ},y) \rangle^{(1 - n_{j})p'} \,
 dy}^{\frac{p}{p'}} \, dx \\
&  \qquad  + \int_{\R^{n_{i}} \xx
 \mathcal{M}_{i} \setminus K_{i}} \br{\int_{D_{j}^{2}(x)} \langle
 d(x_{i}^{\circ},x) \rangle^{-n_{i} p'} \langle d(x_{j}^{\circ},y) \rangle^{(2 - n_{j})p'} \,
 dy}^{\frac{p}{p'}} \, dx \\
&=: I_{1} + I_{2},
\end{split}\end{align*}
where
\begin{align*}\begin{split}  
 &D_{j}^{1}(x) := \lb y \in \R^{n_{j}} \xx
     \mathcal{M}_{j} \setminus K_{j} : d(x_{j}^{\circ},y) \geq
     d(x_{i}^{\circ},x) \rb, \\
     &D_{j}^{2}(x) := \lb y \in \R^{n_{j}} \xx
     \mathcal{M}_{j} \setminus K_{j} : d(x_{j}^{\circ},y) <
     d(x_{i}^{\circ},x) \rb.
   \end{split}\end{align*}
 For the first term,
 $$
I_{1} = \int_{\R^{n_{i}} \xx \mathcal{M}_{i} \setminus K_{i}} \frac{1}{\langle
  d(x_{i}^{\circ},x) \rangle^{(n_{i} - 1) p}}  \br{\int_{D_{j}^{1}(x)}
\frac{dy}{\langle d(x_{j}^{\circ},y) \rangle^{(n_{j} -
    1)p'}}}^{\frac{p}{p'}} \, dx.
$$
The interior integral is given by
\begin{align}\begin{split}
    \label{eqtn:D1j}
\int_{D^{1}_{j}(x)} \frac{dy}{\langle d(x_{j}^{\circ},y)
  \rangle^{(n_{j} - 1) p'}} &= \int_{\R^{n_{j}} \xx \mathcal{M}_{j}
  \setminus K_{j}} \mathbbm{1}_{d(x_{j}^{\circ},y) \geq
    d(x_{i}^{\circ},x)} \frac{dy}{\langle d(x_{j}^{\circ},y) \rangle^{(n_{j} -
      1)p'}}  \\
  &\lesssim \int_{\R^{n_{j}} \xx \mathcal{M}_{j}}
  \frac{dy}{(d(x_{j}^{\circ},y) + d(x_{i}^{\circ},x))^{(n_{j} - 1)p'}}
  \\
  &\lesssim \int_{\R^{n_{j}}} \frac{dy_{1}}{(|y_{1} -
      x_{j,1}^{\circ}| + d(x_{i}^{\circ},x))^{(n_{j} - 1)p'}},
\end{split}\end{align}
where the notation $x_{j,1}^{\circ}$ denotes the Euclidean component
of $x_{j}^{\circ}$ in $\R^{n_{j}} \xx \mathcal{M}_{j}$.
 This will be integrable when
$$
(n_{j} - 1) p' > n_{j} \quad \Leftrightarrow \quad p' >
\frac{n_{j}}{n_{j} - 1} = n_{j}' \quad \Leftrightarrow \quad p < n_{j}.
$$
In which case,
\begin{align*}\begin{split}  
 \int_{D^{1}_{j}(x)} \frac{dy}{\langle d(x_{j}^{\circ},y)
   \rangle^{(n_{j} - 1) p'}} &\lesssim
 \int^{\infty}_{d(x_{i}^{\circ},x)} \frac{r^{n_{j} - 1}}{r^{(n_{j} -
     1)p'}} \, dr \\
 &\simeq \brs{- r^{-(n_{j} - 1)p' +
     n_{j}}}^{\infty}_{d(x_{i}^{\circ},x)} \\
 &\lesssim \langle d(x_{i}^{\circ}, x) \rangle^{-(n_{j} - 1)p' + n_{j}},
\end{split}\end{align*}
where we used the fact that $p < n_{j}$ to deduce that $-(n_{j} - 1)p'
+ n_{j} < 0$ when performing the integration.
Applying this estimate to $I_{1}$ gives
\begin{align*}\begin{split}  
 I_{1} &\lesssim \int_{\R^{n_{i}} \xx \mathcal{M}_{i} \setminus K_{i}}
\frac{1}{\langle d(x_{i}^{\circ},x) \rangle^{(n_{i} - 1)p}} \cdot
\langle d(x_{i}^{\circ},x) \rangle^{-(n_{j} - 1)p + n_{j}
  \frac{p}{p'}} \, dx \\
&= \int_{\R^{n_{i}} \xx \mathcal{M}_{i} \setminus K_{i}} \frac{dx}{
  \langle d(x_{i}^{\circ},x) \rangle^{(n_{i} - 1)p + (n_{j} - 1)p -
    n_{j}(p - 1)}}.
\end{split}\end{align*}
This will be finite provided that
$$
(n_{i} - 1)p + (n_{j} - 1)p - n_{j}(p - 1) > n_{i} \quad
\Leftrightarrow \quad p > \frac{n_{i} - n_{j}}{n_{i} - 2}.
$$
Since $\frac{n_{i} - n_{j}}{n_{i} - 2} < \frac{n_{i} - 2}{n_{i} -
  2} = 1$, this will be satisfied when $p > 1$. It has therefore been
proved that $I_{1}$ is finite when $1 < p < n_{j}$.

It remains to consider the term $I_{2}$, 
\begin{equation}
  \label{eqtn:I2}
  I_{2} = \int_{\R^{n_{i}} \xx \mathcal{M}_{i} \setminus K_{i}}
  \frac{1}{\langle d(x_{i}^{\circ},x) \rangle^{n_{i}p}}
  \br{\int_{D^{2}_{j}(x)} \frac{dy}{\langle d(x_{j}^{\circ},y)^{(n_{j}
        - 2) p'}}}^{\frac{p}{p'}} \, dx.
\end{equation}
 It will be proved that this
term is finite for all $1 < p < \infty$. The interior integral is
given by
\begin{align*}\begin{split}  
 \int_{D^{2}_{j}(x)} \frac{dy}{\langle d(x_{j}^{\circ},y)
   \rangle^{(n_{j} - 2) p'}} &= \int_{\R^{n_{j}} \xx \mathcal{M}_{j}
   \setminus K_{j}}  \mathbbm{1}_{d(x_{j}^{\circ},y) <
   d(x_{i}^{\circ},x)} \frac{dy}{\langle d(x_{j}^{\circ},y)
   \rangle^{(n_{j} - 2)p'}} \\
 &\lesssim \int_{\R^{n_{j}} \xx \mathcal{M}_{j}}
 \mathbbm{1}_{d(x_{j}^{\circ},y) < d(x_{i}^{\circ},x)}
 \frac{dy}{(d(x_{j}^{\circ},y) + D_{K})^{(n_{j} - 2)p'}} \\
 &\lesssim \int_{\R^{n_{j}}} \mathbbm{1}_{|y_{1} - x_{j,1}^{\circ}| <
   d(x_{i}^{\circ},x)} \frac{dy_{1}}{(|y_{1} - x_{j,1}^{\circ}| +
   D_{K})^{(n_{j} - 2)p'}} \\
 &\lesssim \int^{2 d(x_{i}^{\circ},x)}_{D_{K}} r^{n_{j} - 1 - (n_{j} -
   2)p'} \, dr,
\end{split}\end{align*}
where $D_{K} > 0$ is some constant that satisfies $d(x_{k}^{\circ},z) >
D_{K}$ for all $z \in \mathcal{M} \setminus K$ and $k = 1, \cdots, l$.
 Let's estimate this integral based on three distinct cases.

\vspace*{0.1in}

\underline{\textit{Case 1.}} Suppose $n_{j} - 1 - (n_{j} - 2)p' > -1$,
which itself is equivalent to the condition $p > \frac{n_{j}}{2}$. Then
\begin{align*}\begin{split}  
 \int^{2 d(x_{i}^{\circ},x)}_{D_{K}} r^{n_{j} - 1 - (n_{j} - 2)p'} \, dr
 &\simeq \brs{r^{n_{j} - (n_{j} - 2)p'}}^{2 d(x_{i}^{\circ},x)}_{D_{K}}
 \\
 &\leq \langle d(x_{i}^{\circ},x) \rangle^{n_{j} - (n_{j} - 2)p'}.
\end{split}\end{align*}
Substituting this back into $I_{2}$,
\begin{align*}\begin{split}  
 I_{2} &\lesssim \int_{\R^{n_{i}} \xx \mathcal{M}_{i} \setminus K_{i}}
 \frac{1}{\langle d(x_{i}^{\circ},x) \rangle^{n_{i}p}} \cdot \langle
 d(x_{i}^{\circ},x) \rangle^{n_{j} \frac{p}{p'} - (n_{j} - 2)p} \, dx
 \\
 &= \int_{\R^{n_{i}} \xx \mathcal{M}_{i} \setminus K_{i}} \frac{dx}{\langle
 d(x_{i}^{\circ},x) \rangle^{n_{i}p + (n_{j} - 2)p - n_{j}
   \frac{p}{p'}}}.
 \end{split}\end{align*}
This will be integrable when
$$
n_{i} p + (n_{j} - 2) p - n_{j} (p - 1) > n_{i} \quad \Leftrightarrow
\quad p > \frac{n_{i} - n_{j}}{n_{i} - 2},
$$
which will be satisfied when $p > 1$ since $\frac{n_{i} - n_{j}}{n_{i}
- 2} < \frac{n_{i} - 2}{n_{i} - 2} = 1$.

\vspace*{0.1in}

\underline{\textit{Case 2.}} Suppose  that $n_{j} - 1 - (n_{j} - 2)p'
< -1$, which itself is equivalent to the condition $p <
\frac{n_{j}}{2}$. We would then have
\begin{align*}\begin{split}  
 \int^{2 d(x_{i}^{\circ},x)}_{D_{K}} r^{n_{j} - 1 - (n_{j} - 2)p'} \, dr &\simeq
 \brs{-r^{n_{j} - (n_{j} - 2)p'}}^{2 d(x_{i}^{\circ},x)}_{D_{K}} \\
 &\lesssim 1.
\end{split}\end{align*}
Thus
$$
I_{2} \lesssim \int_{\R^{n_{i}} \xx \mathcal{M}_{i} \setminus K_{i}}
\frac{1}{\langle d(x_{i}^{\circ},x) \rangle^{n_{i} p}} \, dx,
$$
which will be finite when $p > 1$.

\vspace*{0.1in}

\underline{\textit{Case 3.}} Suppose $n_{j} - 1 - (n_{j} - 2)p' = -1$,
which itself is equivalent to the condition $p = \frac{n_{i}}{2}$. We
would then have
\begin{align*}\begin{split}  
 \int^{2 d(x_{i}^{\circ},x)}_{D_{K}} r^{n_{j} - 1 - (n_{j} - 2)p'} \, dr
 &= \int^{2 d(x_{i}^{\circ},x)}_{D_{K}} r^{-1} \, dr \\
 &= \frac{1}{D_{K}^{\epsilon}} \int^{2 d(x_{i}^{\circ},x)}_{D_{K}}
 \frac{D_{K}^{\epsilon}}{r} \, dr \\
 &\lesssim \int^{2 d(x_{i}^{\circ},x)}_{D_{K}} r^{\epsilon - 1} \, dr \\
 &\lesssim \langle d(x_{i}^{\circ},x) \rangle^{\epsilon}.
\end{split}\end{align*}
Therefore,
$$
I_{2} \lesssim \int_{\R^{n_{i}} \xx \mathcal{M}_{i} \setminus K_{i}}
\frac{1}{\langle d(x_{i}^{\circ},x) \rangle^{n_{i} p - \epsilon \frac{p}{p'}}} \, dx.
$$
This will be finite when
$$
n_{i}p - \epsilon (p - 1) > n_{i} \quad \Leftrightarrow \quad p > 1.
$$
This completes our proof of estimate
\eqref{eqtn:KernelCondition4} and thus shows that the operator
$S^{3}_{<}$ is bounded on $L^{p}$ for all $p \in (1, \min_{i} n_{i})$.

\vspace*{0.1in}

It remains to prove that $S^{3}_{<}$ is weak-type $(1,1)$. On
combining \eqref{eqtn:h31}, \eqref{eqtn:h32} and \eqref{eqtn:h33}, it
is evident that
$$
\sup_{y \in \mathcal{M}} \abs{h_{3}(x,y)} \lesssim
\left\lbrace \begin{array}{c c} 
 \langle d(x_{i}^{\circ},x) \rangle^{-n_{i}} & for \ x \in \R^{n_{i}}
                                               \xx \mathcal{M}_{i}
                                               \setminus K_{i}, \ 1 \leq i
                                               \leq l, \\
               1 & for \ x \in K.
 \end{array} \right.
$$
This implies that for $f \in C^{\infty}_{c}$,
\begin{align*}\begin{split}  
 \abs{S_{<}^{3}f(x)} &\leq \int_{\mathcal{M}} h_{3}(x,y) \abs{f(y)} \, dy
 \\
 &\leq \sup_{y \in \mathcal{M}} \abs{h_{3}(x,y)} \int_{\mathcal{M}}
 \abs{f(y)} \, dy \\
 &\lesssim \left\lbrace \begin{array}{c c} 
 \langle d(x_{i}^{\circ},x) \rangle^{-n_{i}} \norm{f}_{L^{1}} & for \
                                                                x \in
                                                                \R^{n_{i}}
                                                                \xx
                                                                \mathcal{M}_{i}
                                                                \setminus
                                                                K_{i}, \ 1
                                                                \leq i
                                                                \leq l,
                          \\
                          \norm{f}_{L^{1}} & for \ x \in K.
 \end{array} \right.
 \end{split}\end{align*}
This function is clearly in $L^{1,\infty}$, and thus
$S_{<}^{3}$ is weak-type $(1,1)$.

\subsection{The Operator $S_{<}^{4}$}

The previous proof of the $L^{p}$-boundedness and weak-type $(1,1)$ property of the operator
$S^{3}_{<}$ was derived entirely from the pointwise estimates for the
kernel of $\nabla H_{3}(k)$ provided by Proposition
\ref{prop:H3}. Since the kernels of $\nabla H_{4}(k)$ satisfy even
stronger pointwise estimates by Proposition \ref{prop:H4}, an
identical argument can be applied to
obtain the $L^{p}$-boundedness and weak-type $(1,1)$ property of the operator $S^{4}_{<}$. Moreover,
due to the increased strength of the pointwise estimates for $\nabla
H_{4}(k)$, one will obtain $L^{p}$-boundedness for all $p$
in the full range $(1,\infty)$ instead of the restricted range
$(1,n_{min})$. This improved range of boundedness will prove to be an
essential part of our proof of the unboundedness of $S$ for $p \geq
n_{min}$ in Section \ref{sec:Unboundedness}. As such, the reader is advised
keep it in mind.

 \section{The High Energy Square Function}
 \label{sec:High}
 
 Recall that the high energy part of the square function is defined by
 $$
 S_{>}f(x) := \br{\int^{\infty}_{1} \abs{\nabla (k^{2} + \Delta)^{-M} f(x)}^{2}
 	k^{4 M - 3} \, dk}^{\frac{1}{2}}. 
 $$
 The main aim of this section is to verify that the operator $S_{>}$ is bounded on all $L^p$ spaces for $1<p<\infty$ and is weak type $(1,1)$. The argument we describe below has the same structure as in  \cite[Section 5]{hs2019} which we adapt to the square function setting.

\begin{prop} 
	\label{high} 
	The high energy square function operator $S_{>}$ is bounded on $L^{p}(\mathcal{M})$
	for any $p \in (1, \infty)$. That is, for every $1< p < \infty$ there
	exists $c > 0$ such that 
	$$
 \norm{S_{>} f}_{p}	\leq c \norm{f}_{p} 
	$$
	for all $f \in L^{p}(\mathcal{M})$. 
	In addition the operator $S_{>} $ 
	satisfies weak type $(1,1)$ estimates. 
\end{prop} 
\begin{proof}
	Denote by $F_M$ the Fourier transform of the function $\lambda \to (\lambda+1)^{-M}$. Note that 
	$$
	\int_{-\infty}^\infty e^{-i\xi \lambda}(\lambda^2+k^2)^{-M}d \lambda= k^{1-2M}F_M(k|\xi|).
		$$
	In fact $F_1(|\xi|)	=c e^{-|\xi|}$ and $ k^{1-2(M+1)}F_{M+1}(k|\xi|)=
	\partial_{k^2} k^{1-2M}F_M(k|\xi|)$. Note also that 
	there exist  positive constants $C, c>0$ (depending on $M$) such 
	that 
	$$
	F_M(|\xi|) \le C e^{-c|\xi|}.
	$$

	Now, let $\eta \in C_c^{\infty}(\R)$ be an even  compactly supported function such that 
	$0 \le \eta (\xi)\le 1$,  $\eta(\xi) =1$ for all $ |\xi| \le
        1/2$ and $\eta (\xi)=0$ for $|\xi| \ge 1$. For any $r>0$ we
        define the function ${G}^M_{r,k}$
	as the inverse Fourier transform of 
	$$
	k^{1-2M}F_M(k|\xi|)\eta (\xi/r).
	$$	
	Then we set 
\begin{equation}
H^M_{r,k}(\lambda)=(\lambda^2+k^2)^{-M}-{G}^M_{r,k}(\lambda)
\end{equation}	
so that the Fourier transform of $H^M_{r,k}$ is equal to 
$$
k^{1-2M}F_M(k|\xi|)\Big(1-\eta (\xi/r)\Big).
$$	
Hence for some  positive constants $C, c>0$ 
\begin{equation}\label{H}
\sup_{\lambda}|H^M_{r,k}(\lambda)| \le C e^{-ckr}
\end{equation}
	

 Proceeding as in \cite[(42) p. 1089]{hs2019}, we have (using the evenness of $\widehat{G}^M_{r,k}(\xi)$)
\begin{align}
G^M_{r,k}(\sqrt{\Delta}) = \frac1{2\pi} \int_{-\infty}^\infty e^{it \sqrt{\Delta} } \widehat{G^M_{r,k}}(t)  \, dt 
= \frac1{\pi} \int_{0}^r \cos(t\sqrt{\Delta} ) \widehat{G^M_{r,k}}(t) \, dt .
\label{cosine}
\end{align}
By \eqref{cosine} and by the finite speed of propagation of $\cos(t\sqrt{\Delta} )$ it follows that  $ G^M_{r,k}(\sqrt{\Delta})$ has Schwartz kernel supported in the set $ \{ (z, z') \mid d(z,z') \leq r \}$. 
Next, set $r = \tilde{r} > 0$ to be half the injectivity radius of $\mathcal{M}$.
Note that the injectivity radius of $\mathcal{M}$ is strictly positive
and, without loss of generality, it can be assumed that $\tilde{r} = 1$.

The operator $\nabla G^{M}_{\tilde{r},k}(\sqrt{\Delta}) = \nabla G^M_{1,k}(\sqrt{\Delta})$ is a pseudodifferential operator of order $1-2M$  since the function $ G^M_{1,k}(\lambda)$ is a symbol of order $-2M$
see \cite[Chapter XII, Section 1]{taylor}. This implies that for any $b \ge N+1-2M$, see \cite[Poposition 2.2 p. 6]{taylor2} 
		\begin{align*}
 |\nabla G^M_{1,k}(\sqrt{\Delta})(x,y)| \le C k^{N-2M-b+1}d(x,y)^{-b}.
	\end{align*}
Hence, taking $b_1= N- \frac{1}{2}$ and $b_2=N + \frac{1}{2}$,
\begin{align*}
\int^\infty_1|\nabla G^M_{1,k}(\sqrt{\Delta})(x,y)|^2k^{4M-3}
\le \int_1^{d(x,y)^{-1}}  C k^{2N-4M-2b_1+2}d(x,y)^{-2b_1} k^{4M-3}dk \\+
\int_{d(x,y)^{-1}}^{\infty}  C k^{2N-4M-2b_2+2} d(x,y)^{-2b_2} k^{4M-3} \le Cd(x,y)^{-2N}.
\end{align*}
Note also that if $X$ is $C^\infty$ vector field on $\mathcal{M}$, uniformly bounded in every $C^m$ norm then 
$$
\int^\infty_1| X \nabla G^M_{1,k}(\sqrt{\Delta})(x,y)|^2k^{4M-3} \, dk
 \le Cd(x,y)^{-2(N+1)}.  
$$
Hence we can use vector-valued standard Calderon-Zygmund approach (with respect to  $L^2([1,\infty),k^{4M-3}dk)$
	to conclude the argument for $\nabla G^M_{1,k}(\sqrt{\Delta})$
	and $k\ge 1$, see for example \cite[\S 6.4, p. 28 ]{can23}.

Next we consider the family of operators $\nabla H^M_{r,k}(\sqrt{\Delta})$.
By Minkowski's integral inequality
\begin{align*}
 &\bigg(\int_{1}^\infty \bigg|\int_{\mathcal{M}} K_{\nabla H^M_{1,k}(\sqrt{\Delta})}(x,y)f(y)dy \bigg|^2  k^{4M-3}dk \bigg)^{\frac12}\\
 & \hspace*{2in}\le \int_{\mathcal{M}}\int_{1}^\infty \bigg(\bigg|\ K_{\nabla H^M_{1,k}(\sqrt{\Delta})}(x,y) \bigg|^2  k^{4M-3}dk \bigg)^{\frac12}
|f(y)|dy \\ & \hspace*{2in} =\int_{\mathcal{M}}h_*(x,y)
|f(y)|dy,
\end{align*}
where
$$
h_*(x,y)=\bigg(\int_{1}^\infty\bigg|\ K_{\nabla H^M_{1,k}(\sqrt{\Delta})}(x,y) \bigg|^2  k^{4M-3}dk \bigg)^{\frac12}
$$
By Schur's integral test, to conclude the proof it is now enough to show that there exists a constant $C$ such that 
\begin{eqnarray}\label{xx}
\sup_y \int_{\mathcal{M}}h_*(x,y) dx \le C
\end{eqnarray}
and 
\begin{eqnarray}\label{yy}
\sup_x \int_{\mathcal{M}}h_*(x,y) dy \le C.
\end{eqnarray}
Here we again use the approach described in 
\cite{hs2019}. 
First we note that finite propagation speed implies that 
$$
K_{\nabla H^M_{{r_1},k}(\sqrt{\Delta})}(x,y)=K_{\nabla H^M_{{r_2},k}(\sqrt{\Delta})}(x,y) \text{ if } x\notin B(y,r_2), \quad r_2 \geq {r_1} > 0 .
$$ 
It follows that  for every  $r>1$
\begin{equation}\begin{aligned}
\int_{x\notin B(y,r)}|K_{ \nabla H^M_{{1},k}(\sqrt{\Delta})}(x,y)|^2dx &=
\int_{x\notin B(y,r)}|K_{\nabla H^M_{{r},k}(\sqrt{\Delta})}(x,y)|^2dx \\ &\le
\int_{x \in M}|K_{\nabla H^M_{{r},k}(\sqrt{\Delta})}(x,y)|^2dx  \\
&= \Big\langle \nabla_x K_{H^M_{{r},k}(\sqrt{\Delta})}(\cdot,y), \nabla _x K_{H^M_{{r},k}(\sqrt{\Delta})}(\cdot,y) \Big\rangle \\
&= \Big\langle \Delta K_{H^M_{{r},k}(\sqrt{\Delta})}(\cdot,y), K_{H^M_{{r},k}(\sqrt{\Delta})}(\cdot,y) \Big\rangle \\
&= \int_{x \in M}|K_{ \sqrt{\Delta} H^M_{r,k}(\sqrt{\Delta})}(x,y)|^2dx.
\end{aligned}
\end{equation}
A straightforward argument, see \cite[Proposition~2.4]{hs2019}, 
shows that $\|(I+\Delta)^{-n}\|^2_{2 \to \infty} < \infty$ for any $n \ge [N/4]+1$. Hence by \eqref{H} for some $c>0$
\begin{eqnarray*}
	\sup_y \int_{x \in M}|K_{\sqrt{\Delta} H^M_{r,k}(\sqrt{\Delta})}(x,y)|^2dx &=&
	\big\|\sqrt{\Delta} H^M_{{r},k}(\sqrt{\Delta})\big\|^2_{1\to 2} =  \big\|\sqrt{\Delta}H^M_{{r},k}(\sqrt{\Delta})\big\|^2_{2 \to \infty}\\
	&\le&  \big\|  (I+\Delta)^{ n}  \sqrt{\Delta} H^M_{{r},k}(\sqrt{\Delta}) \big\|^2_{2  \to 2}   \big\|(I+\Delta)^{-n}\big\|^2_{2 \to \infty} \\[6pt]
	&\leq& C \sup_{\lambda \geq 0} 
	\big|(1+\lambda^2)^n {\lambda}H^M_{{r},k}(\lambda)\big|^2 \\
	&\le& C e^{-ck r}. 
\end{eqnarray*}
Now we can conclude that 
\begin{equation}
\sup_y\int_{x\notin B(y,r)} \int_{1}^\infty\bigg|\ K_{\nabla H^M_{1,k}(\sqrt{\Delta})}(x,y) \bigg|^2  k^{4M-3}dkdx \lesssim   e^{-cr}
\label{squared}\end{equation}
This proves \eqref{xx} since it implies in particular that 
$$
\sup_y\int_{r \leq d(x,y) \leq 2r} h_*(x,y)^2 dx \lesssim  e^{-cr}.
$$
The measure of the set $\{x \in M \mid r \leq d(x,y) \leq 2r \}$ is
bounded by $C r^N$, $r \geq  \tilde{r} = 1$, where $N$ is the dimension of $\mathcal{M}$, uniformly in $y \in \mathcal{M}$.
So we can apply H\"older's inequality to find that 
$$
\sup_y\int_{r \leq d(x,y) \leq 2r} h_*(x,y) dx \lesssim  e^{-cr}.
$$
Then the above  estimates can  be summed over a sequence of dyadic annuli to obtain \eqref{xx}. 

To prove \eqref{yy} we use the notion of the Hodge Laplacian. Recall that for any bounded Borel function $F$ if we write $\Delta_q$ for the
Hodge Laplacian acting on $q$-forms then we can write $d F(\sqrt{ \Delta_0}) =  F(\sqrt{ \Delta_1}) d$. This means that to verify \eqref{yy} it is enough 
to prove \eqref{xx} for $H^M_{{r},k}(\sqrt{\Delta_1})$. This can be
achieved by essentially repeating the above argument for $\Delta_0=\Delta$. See 
\cite{hs2019} for more detailed calculations. Now the continuity of the off-diagonal part of high energy square 
function follows from \eqref{xx}, \eqref{yy} and  Minkowski integral inequality. 
\end{proof} 

 \section{The Reverse Inequality}
 \label{sec:Reverse}

 In this section, the reverse inequality portion of Theorem
 \ref{thm:Main} will be proved. That is, it will be shown that $\norm{f}_{p} \lesssim \norm{S
   f}_{p}$ for any $p \in (n_{min}',n_{min})$. This will be achieved
 using, by now, standard arguments that can be found in
{\cite[Thm.~7.1]{auscher2007necessary}}. By the
functional calculus of $\Delta$ on $L^{2}$, for $f \in L^{p} \cap L^{2}$ we have the resolution of the
identity,
$$
f = c_{M} \int^{\infty}_{0} t^{2} \Delta (1 + t^{2} \Delta)^{-2 M} f \, \frac{dt}{t}
$$
for some constant $c_{M} > 0$ that is dependent on $M$. Then for any
$g \in L^{p'}$, Fubini-Tonelli's theorem, whose application will be
justified retrospectively, implies 
\begin{align*}\begin{split}  
 \int_{\mathcal{M}} f \cdot \overline{g}  &\simeq \int_{\mathcal{M}} \br{
 \int^{\infty}_{0} t^{2} \Delta (1 + t^{2} \Delta)^{-2 M} f(x) \,
 \frac{dt}{t}} \overline{g(x)} \, dx \\
 &= \int^{\infty}_{0} \int_{\mathcal{M}} t^{2} \Delta (1 + t^{2}
 \Delta)^{-2 M} f(x) \cdot \overline{g(x)} \, dx \, \frac{dt}{t} \\
 &= \int^{\infty}_{0} \int_{\mathcal{M}} t \nabla (1 + t^{2}
 \Delta)^{-M}f(x) \overline{t \nabla (1 + t^{2} \Delta)^{-M}g(x)} \, dx \, \frac{dt}{t}.
\end{split}\end{align*}
On applying Fubini-Tonelli once more, followed by Cauchy-Schwarz and
H\"{o}lder's inequality,
\begin{align*}\begin{split}  
  \int_{\mathcal{M}} f \cdot \overline{g} &= \int_{\mathcal{M}}
  \int^{\infty}_{0} t \nabla (1 + t^{2} \Delta)^{-M}f(x) \overline{t
    \nabla (1 + t^{2} \Delta)^{-M}g(x)} \, \frac{dt}{t} \, dx \\
  &\leq \int_{\mathcal{M}} S f(x) \cdot S g(x) \, dx \\
  &\leq \norm{S f}_{p} \norm{S g}_{p'}.
\end{split}\end{align*}
Since $S$ is bounded on both $L^{p}$ and $L^{p'}$ for $p \in
(n_{min}',n_{min})$ it follows that this quantity is finite, thereby
retrospectively justifying our two previous applications of the
Fubini-Tonelli theorem. Moreover, the $L^{p'}$-boundedness of $S$ in
particular implies
$$
 \int_{\mathcal{M}} f \cdot \overline{g}  \lesssim \norm{S f}_{p} \norm{g}_{p'},
 $$
 and thus
 $$
\norm{f}_{p} \lesssim \norm{S f}_{p}
$$
for all $f \in L^{p} \cap L^{2}$.
As $S$ is bounded on $L^{p}$, this estimate must also hold true for
all $f \in L^{p}$ by density.

 \section{Unboundedness for $p \geq n_{min}$}
 \label{sec:Unboundedness}

Throughout this section we impose the restriction $2 M <
n_{min}$ and provide a proof of the unboundedness portion of Theorem
\ref{thm:Main}. The upper restriction on the order of the resolvent will allow us to utilize Proposition \ref{prop:DerivativeAt0},
which will form a key component of the proof.
Our proof will exploit some of the ideas utilized in the unboundedness
argument of {\cite[Sec.~6]{hs2019}}. However, a number of difficulties
will arise due to the quadratic nature of the square function and the
higher-order degree of the resolvent.

\vspace*{0.1in}

 In order to prove unboundedness of the square
 function operator on $L^{p}$ for $p \geq n_{min}$
 it is clearly sufficient to prove unboundedness of the low energy
 square function. Recall from our proof of the boundedness portion of Theorem
 \ref{thm:Main} that the components $S^{2}_{<}$ and $S^{4}_{<}$ are
 bounded on $L^{p}(\mathcal{M})$ for all $p \in (1,\infty)$. The
 unboundedness of the operator $S_{<}$ will thus follow directly from
 the unboundedness of the operator
 \begin{align*}\begin{split}
S_{<}^{1 + 3}f(x) &:= \br{\int^{1}_{0} \abs{\nabla (H_{1}(k) +
    H_{3}(k))f(x)}^{2} k^{4M - 3} \, dk}^{\frac{1}{2}} \\
&= \br{\int^{1}_{0} \brs{\int_{\mathcal{M}} (\nabla H_{1}(k)(x,y) +
    \nabla H_{3}(k)(x,y)) f(y) \, dy}^{2} k^{4 M - 3} \, dk}^{\frac{1}{2}}.
 \end{split}\end{align*}
Recall from the proof of the boundedness portion of Theorem \ref{thm:Main} that the operator $S^{1}_{<}$
divides into two parts, $\Lambda^{i}$ and $\Pi^{i}$, depending on whether the gradient hits the
resolvent factor or the function $\phi_{i}$. The term where the gradient hits the
resolvent $\Lambda_{i}$ was proved to be bounded on $L^{p}$ for all $p \in (1,\infty)$. It is
therefore sufficient to prove unboundedness of the operator
\begin{align}\begin{split}
    \label{eqtn:Xi}
 &\Bigg(\int^{1}_{0} \Bigg[\sum_{i = 1}^{l} \int_{\mathcal{M}} \nabla
    \phi_{i}(x) (\Delta_{\R^{n_{i}} \xx \mathcal{M}_{i}} +
    k^{2})^{-M}(x,y) \phi_{i}(y) f(y) + \xi_{i}(k)(x,y) f(y)  \\ & \qquad
\qquad \qquad   + \nabla u_{i}(x,k)
    (\Delta_{\R^{n_{i}} \xx \mathcal{M}_{i}} +
    k^{2})^{-M}(x_{i}^{\circ},y) \phi_{i}(y) f(y) \, dy \Bigg]^{2}
    k^{4 M - 3} \, dk \Bigg)^{\frac{1}{2}},
  \end{split}\end{align}
where $\xi_{i}(k)(x,y)$ is defined through
$$
\xi_{i}(k)(x,y) :=  \sum_{j = 0}^{M - 2} \frac{(-1)^{M + j - 1}}{(M -
  j - 1)!} (\Delta_{\R^{n_{i}} \xx \mathcal{M}_{i}} + k^{2})^{-(j +
  1)}(x_{i}^{\circ},y) \nabla u_{i}^{(M - 1 - j)}(x,k) \phi_{i}(y).
$$
Let $\tau$ be a nonnegative function, not identically zero, that is
compactly supported on one of the ends. It is clear that the unboundedness of the
operator \eqref{eqtn:Xi} will be implied by the unboundedness of the operator
\begin{align}\begin{split}
    \label{eqtn:Xi1}
 &\tau(x) \Bigg(\int^{1}_{0} \Bigg[\sum_{i = 1}^{l} \int_{\mathcal{M}} \nabla
    \phi_{i}(x) (\Delta_{\R^{n_{i}} \xx \mathcal{M}_{i}} +
    k^{2})^{-M}(x,y) \phi_{i}(y) f(y) + \xi_{i}(k)(x,y) f(y)  \\ & \qquad
\qquad \qquad   + \nabla u_{i}(x,k)
    (\Delta_{\R^{n_{i}} \xx \mathcal{M}_{i}} +
    k^{2})^{-M}(x_{i}^{\circ},y) \phi_{i}(y) f(y) \, dy\Bigg]^{2} k^{4 M - 3} \, dk\Bigg)^{\frac{1}{2}}.
  \end{split}\end{align}
Let $\xi_{i}(k)$ be the operator corresponding to the kernel
$\xi_{i}(k)(x,y)$ and define
$$
\Xi f(x) := \tau(x) \br{\int^{1}_{0} \brs{\sum_{i = 1}^{l}
    \xi_{i}(k)f(x)}^{2} k^{4 M - 3} \, dk}^{\frac{1}{2}}.
$$
If it can be proved that $\Xi$ is bounded on $L^{p}$ for any $p \in
(2,\infty)$ then the task of proving the unboundedness of the
operator \eqref{eqtn:Xi1} on $L^{p}$ for $p \geq n_{min}$
will be reduced to proving the unboundedness of the operator
  \begin{align}\begin{split}
\label{eqtn:XiRemoved}
 &\tau(x) \Bigg(\int^{1}_{0} \Bigg[\sum_{i = 1}^{l} \int_{\mathcal{M}} \nabla
    \phi_{i}(x) (\Delta_{\R^{n_{i}} \xx \mathcal{M}_{i}} +
    k^{2})^{-M}(x,y) \phi_{i}(y) f(y)  \\ & \qquad
\qquad \qquad   + \nabla u_{i}(x,k)
    (\Delta_{\R^{n_{i}} \xx \mathcal{M}_{i}} +
    k^{2})^{-M}(x_{i}^{\circ},y) \phi_{i}(y) f(y) \, dy\Bigg]^{2} k^{4 M - 3} \, dk\Bigg)^{\frac{1}{2}}.
  \end{split}\end{align}
To this end, the triangle inequality tells us that the $L^{p}$-boundedness of $\Xi$
will follow from the $L^{p}$-boundedness of the operators
$$
\Xi^{j}_{i}f(x) := \br{\int^{1}_{0} \abs{\xi^{j}_{i}(k)f(x)}^{2} k^{4M -
    3} \, dk}^{\frac{1}{2}},
$$
for each $1 \leq i \leq l$ and $0 \leq j \leq M - 2$, where
$$
\xi^{j}_{i}(k)(x,y) := \tau(x)
\br{\Delta_{\R^{n_{i}} \xx \mathcal{M}_{i}} + k^{2}}^{-(j +
  1)}(x_{i}^{\circ},y) \nabla u_{i}^{(M - 1 - j)}(x,k) \phi_{i}(y).
$$
This operator can then itself be controlled from above by
\begin{align*}\begin{split}  
    \Xi^{j}_{i}f(x) &\leq
    \br{\int^{1}_{0} \abs{\xi^{j,1}_{i}(k)f(x)}^{2}
  k^{4M - 3} \, dk}^{\frac{1}{2}} + \br{\int^{1}_{0}
  \abs{\xi^{j,2}_{i}(k)f(x)}^{2} k^{4M - 3} \, dk}^{\frac{1}{2}} \\
&=: \Xi^{j,1}_{i}f(x) + \Xi^{j,2}_{i}f(x),
 \end{split}\end{align*}
where
$$
\xi^{j,1}_{i}(k)(x,y) := \tau(x)
(\Delta_{\R^{n_{i}} \xx \mathcal{M}_{i}} + k^{2})^{-(j + 1)}(x_{i}^{\circ},y)
\abs{\nabla u_{i}^{(M - 1 - j)}(x,k) - \nabla u_{i}^{(M - 1 -
    j)}(x,0)} \phi_{i}(y)
$$
and
$$
\xi^{j,2}_{i}(k)(x,y) := \tau(x) \br{\Delta_{\R^{n_{i}} \xx \mathcal{M}_{i}} +
k^{2}}^{-(j + 1)}(x_{i}^{\circ},y) \abs{\nabla u_{i}^{(M - 1 -
  j)}(x,0)} \phi_{i}(y).
$$

\begin{lem} 
 \label{lem:Xij1} 
For any $0 \leq j \leq M - 1$, $1 \leq i \leq l$ and $p \in (2,\infty)$, the operator $\Xi_{i}^{j,1}$ is bounded on $L^{p}$.
 \end{lem}

\begin{proof}
Minkowski's integral inequality, that can be applied since $\frac{p}{2} > 1$,
implies that
\begin{align}\begin{split}  
\label{eqtn:Xiij1}
 \norm{\Xi^{j,1}_{i}f}_{p} &= \norm{\int^{1}_{0}
   \abs{\xi^{j,1}_{i}(k)f}^{2} k^{4 M - 3} \,
   dk}^{\frac{1}{2}}_{\frac{p}{2}} \\
 &\leq \br{\int^{1}_{0}
   \norm{\abs{\xi_{i}^{j,1}(k)f}^{2}}_{\frac{p}{2}} k^{4 M - 3} \,
   dk}^{\frac{1}{2}} \\
 &= \br{\int^{1}_{0} \norm{\xi_{i}^{j,1}(k)}^{2}_{p} k^{4 M - 3} \,
   dk}^{\frac{1}{2}} \norm{f}_{p}.
 \end{split}\end{align}
For fixed $k$,
$\xi^{j,1}_{i}(k)$ is a rank one operator given by $\xi^{j,1}_{i}(k) =
a_{j} \cdot \langle b_{j}, \cdot \rangle$ with
$$
a_{j}(x) = \tau(x) \abs{\nabla u_{i}^{(M - 1 - j)}(x,k) - \nabla u_{i}^{(M - 1 -
    j)}(x,0)}
$$
and
$$
b_{j}(y) := \br{\Delta_{\R^{n_{i}} \xx \mathcal{M}_{i}} + k^{2}}^{-(j +
  1)}(x_{i}^{\circ},y) \phi_{i}(y).
$$
From Proposition \ref{prop:DerivativeAt0}, it is clear that
$$
\norm{a_{j}}_{p} \lesssim k^{3 + 2 j - 2 M}.
$$
Let $D > 0$ be such that $d(x_{i}^{\circ},y) \geq D$ for all $y \in
\mathrm{supp} \ \phi_{i}$. Then, in a similar manner to (52) of \cite{hs2019},
\begin{align}\begin{split}  
\label{eqtn:b}
 \norm{b_{j}}_{p'} &= \br{\int_{\R^{n_{i}} \xx \mathcal{M}_{i}}
   \abs{(\Delta_{\R^{n_{i}} \xx \mathcal{M}_{i}} + k^{2})^{-(j +
       1)}(x_{i}^{\circ},y) \phi_{i}(y)}^{p'} \, dy}^{\frac{1}{p'}} \\
 &\leq k^{-2 j} \br{\int_{d(x_{i}^{\circ},y) \geq D}
   \abs{d(x_{i}^{\circ},y)^{2 - n_{i}} \exp(-c k
     d(x_{i}^{\circ},y))}^{p'} \, dy}^{\frac{1}{p'}} \\
 &\lesssim \left\lbrace \begin{array}{c c} k^{\frac{n_{i}}{p} - 2 - 2
                          j} & p > \frac{n_{i}}{2} \\ k^{-2 j} \log
                          \, k & p = \frac{n_{i}}{2} \\ k^{-2 j} & p
                                                                   < \frac{n_{i}}{2}. \end{array} \right.
                                                             \end{split}\end{align}
 We therefore have
\begin{align*}\begin{split}  
\norm{\xi^{j,1}_{i}(k)}_{p} &= \norm{a_{j}}_{p} \cdot \norm{b_{j}}_{p'}
\\
&\lesssim \left\lbrace \begin{array}{c c} k^{\frac{n_{i}}{p} + 1 - 2
                         M} & p > \frac{n_{i}}{2} \\ k^{3 - 2 M }
                         \log \, k & p = \frac{n_{i}}{2} \\ k^{3 - 2
                         M} & p < \frac{n_{i}}{2}. \end{array} \right.
 \end{split}\end{align*}
On applying this estimate to \eqref{eqtn:Xiij1} we find that
$\norm{\Xi_{i}^{j,1}}_{p} < \infty$ for any $p \in (2,\infty)$.
\end{proof}

\begin{lem} 
 \label{lem:Xij2} 
For any $0 \leq j \leq M - 2$, $1 \leq i \leq l$ and $p \in (2,\infty)$, the operator $\Xi_{i}^{j,2}$ is bounded on $L^{p}$.
 \end{lem}

\begin{proof}
As was the case for $\xi^{j,1}_{i}(k)$, for fixed $k$ the operator
$\xi^{j,2}_{i}(k)$ is of rank one. In particular, we have
$\xi^{j,2}_{i}(k) = a'_{j} \cdot \langle b_{j}, \cdot \rangle$ where 
$$
a_{j}'(x) := \tau(x) \abs{\nabla u_{i}^{(M - 1 - j)}(x,0)}
$$
and $b_{j}$ is as defined previously for the operator $\xi^{j,1}_{i}(k)$. It
is obvious from Proposition \ref{prop:DerivativeAt0} that
$\norm{a_{j}'}_{p} \lesssim 1$ uniformly in
$k$. Estimate \eqref{eqtn:b} then implies
\begin{align*}\begin{split}  
 \norm{\xi^{j,2}_{i}(k)}_{p} &\simeq \norm{a_{j}'}_{p} \cdot \norm{b_{j}}_{p'}
 \\
&\lesssim \norm{b_{j}}_{p'} \\
&\lesssim \left\lbrace \begin{array}{c c} k^{\frac{n_{i}}{p} - 2 - 2
                         j} & p > \frac{n_{i}}{2} \\ k^{-2 j} \log \,
                         k & p = \frac{n_{i}}{2} \\ k^{-2 j} & p < \frac{n_{i}}{2}. \end{array} \right.
 \end{split}\end{align*}
For $p > \frac{n_{i}}{2}$, once again by Minkowski's inequality,
\begin{align*}\begin{split}  
 \norm{\Xi^{j,2}_{i}}_{p} &\leq \br{\int^{1}_{0}
   \norm{\xi^{j,2}_{i}(k)}_{p}^{2} k^{4 M - 3} \, dk}^{\frac{1}{2}} \\
&\leq \br{\int^{1}_{0} k^{\frac{2 n_{i}}{p} + 4 (M - j) - 7} \, dk}^{\frac{1}{2}}.
 \end{split}\end{align*}
This will be finite provided that
$$
\frac{2 n_{i}}{p} + 4 (M - j) - 7 > -1,
$$
which is implied by the restriction $j \leq M - 2$. The cases $p =
\frac{n_{i}}{2}$ and $p < \frac{n_{i}}{2}$ proceed similarly thereby
producing $\norm{\Xi^{j,2}_{i}}_{p} < \infty$.
  \end{proof}

Now that the $L^{p}$-boundedness of $\Xi$ has been proved for $p \in (2,\infty)$, it remains
to consider the unboundedness of the operator given in
\eqref{eqtn:XiRemoved}. Notice that
\begin{align*}\begin{split}  
 & \br{\int^{1}_{0} \brs{\int_{\mathcal{M}} \nabla \phi_{i}(x)
     ((\Delta_{\R^{n_{i}} \xx \mathcal{M}_{i}} + k^{2})^{-M}(x,y) -
     (\Delta_{\R^{n_{i}} \xx \mathcal{M}_{i}} +
     k^{2})^{-M}(x_{i}^{\circ},y)) \phi_{i}(y) f(y) \, dy}^{2} k^{4 M
     - 3} \,
   dk}^{\frac{1}{2}} \\
 & \leq \int_{\mathcal{M}} \nabla \phi_{i}(x)
 \brs{\int^{1}_{0} \abs{(\Delta_{\R^{n_{i}} \xx \mathcal{M}_{i}} +
     k^{2})^{-M}(x,y) - (\Delta_{\R^{n_{i}} \xx \mathcal{M}_{i}} +
     k^{2})^{-M}(x_{i}^{\circ},y)}^{2} k^{4 M - 3} \, dk}^{\frac{1}{2}} \phi_{i}(y)
 f(y) \, dy \\
 &=: \int_{\mathcal{M}} p(x,y) f(y) \, dy =: Pf(x).
\end{split}\end{align*}

\begin{lem} 
 \label{lem:OperatorP} 
 The operator $P$ is bounded on $L^{p}$ for all $p \in (1,\infty)$.
\end{lem}

\begin{proof}  
Let $V \subset \mathrm{supp} \, \phi_{i}$ be an open subset, compactly
supported, that contains $\mathrm{supp} \, \nabla
\phi_{i}$. For $y \in V$, Proposition \ref{prop:EndResolvent} implies 
\begin{align*}\begin{split}  
 p(x,y) &\lesssim \nabla \phi_{i}(x)
 \br{\int^{1}_{0} \abs{(\Delta_{\R^{n_{i}} \xx \mathcal{M}_{i}} +
     k^{2})^{-M}(x,y)}^{2} k^{4M - 3} \, dk}^{\frac{1}{2}} \phi_{i}(y) 
 \\
 & \hspace*{1in} + \nabla \phi_{i}(x)  \br{\int^{1}_{0} \abs{(\Delta_{\R^{n_{i}} \xx \mathcal{M}_{i}} +
     k^{2})^{-M}(x_{i}^{\circ},y)}^{2} k^{4M - 3} \, dk}^{\frac{1}{2}}
 \phi_{i}(y)  \\
 & \lesssim \nabla \phi_{i}(x) d(x,y)^{2 - N} \br{\int^{1}_{0}
   k \exp(-2 c k d(x,y)) \, dk}^{\frac{1}{2}} \phi_{i}(y)  \\
 & \hspace*{1in} + \nabla \phi_{i}(x) d(x_{i}^{\circ},y)^{2 -
   n_{i}} \br{\int^{1}_{0} k \exp(-2 c k d(x_{i}^{\circ},y)) \,
   dk}^{\frac{1}{2}} \phi_{i}(y)  \\
 & \lesssim \nabla \phi_{i}(x) \br{d(x,y)^{1 - N} +
   d(x_{i}^{\circ},y)^{1 - n_{i}}} \phi_{i}(y).
\end{split}\end{align*}
Schur's test then tells us that the operator with kernel $p(x,y)
\mathbbm{1}_{V}(y)$ is bounded on $L^{p}$ for all $p \in (1,\infty)$.

On the other hand, for $y \in \mathrm{supp} \, \phi_{i} \setminus V$,
Proposition \ref{prop:EndResolvent} and the mean value theorem give
\begin{align*}\begin{split}  
& \nabla \phi_{i}(x) \br{\int^{1}_{0} \abs{(\Delta_{\R^{n_{i}} \xx \mathcal{M}_{i}} + k^{2})^{-M}(x,y) - (\Delta_{\R^{n_{i}} \xx
       \mathcal{M}_{i}} + k^{2})^{-M}(x_{i}^{\circ},y)}^{2} k^{4 M - 3} \,
   dk}^{\frac{1}{2}} \phi_{i}(y) \\ & \qquad  \qquad \lesssim \nabla
 \phi_{i}(x) \br{\int^{1}_{0} d(x_{i}^{\circ},x)^{2} \brs{k^{4 - 4 M}\langle
 d(x_{i}^{\circ},y) \rangle^{-2(n_{i} - 1)} \exp(-2 c k
 d(x_{i}^{\circ},y))} k^{4M - 3} \, dk}^{\frac{1}{2}} \phi_{i}(y) \\
&  \qquad \qquad \lesssim \nabla \phi_{i}(x) \langle
d(x_{i}^{\circ},y) \rangle^{-(n_{i} - 1)} \br{\int^{1}_{0} k \exp(- 2 c k
  d(x_{i}^{\circ},y)) \, dk}^{\frac{1}{2}} \phi_{i}(y) \\
&  \qquad \qquad \lesssim \nabla \phi_{i}(x) \langle d(x_{i}^{\circ},y)
\rangle^{-n_{i}} \phi_{i}(y).
\end{split}\end{align*}
The kernel $\mathbbm{1}_{\mathrm{supp} \, \phi_{i} \setminus V}(y) p(x,y)$ is therefore compactly supported in $x$ and it decays to
order $n_{i}$ in $y$. From the argument used to prove the boundedness
of $S_{<}^{3}$, it is clear that the corresponding operator
 will be bounded on
$L^{p}$ for all $p \in (1,\infty)$. This demonstrates that $P$ is
bounded on $L^{p}$ for all $p \in (1,\infty)$. 
 \end{proof}

From the above lemma, it is clear that the unboundedness of the
operator in \eqref{eqtn:XiRemoved} will follow from the unboundedness of
the operator
$$
 \tau(x) \br{\int^{1}_{0} \brs{\sum_{i = 1}^{l} \int_{\mathcal{M}} (\nabla
    \phi_{i}(x) + \nabla u_{i}(x,k)) (\Delta_{\R^{n_{i}} \xx \mathcal{M}_{i}} +
    k^{2})^{-M}(x_{i}^{\circ},y) \phi_{i}(y) f(y) \, dy}^{2} k^{4M - 3} \, dk}^{\frac{1}{2}}.
$$
Choose $i$ such that $n_{i}$ is minimal. Then, in an identical manner
to \cite{hs2019}, the unboundedness of the above operator can be further
reduced to proving the unboundedness of the operator
$$
\tau(x) \cdot \br{\int^{1}_{0} \brs{\int_{\mathcal{M}}
     (\nabla \phi_{i}(x) + \nabla u_{i}(x,k)) (\Delta_{\R^{n_{i}} \xx
       \mathcal{M}_{i}} + k^{2})^{-M}(x_{i}^{\circ},y) \phi_{i}(y)
     f(y) \, dy}^{2} k^{4 M - 3} \, dk}^{\frac{1}{2}}.
$$
From an application of Lemma \ref{lem:Xij1} for the case $j = M - 1$,
it is evident that in order to
deduce that the above operator is unbounded on $L^{p}$ for $p \geq
n_{min}$ it is sufficient to show that
 $$
Rf(x) := \br{\int^{1}_{0} \abs{r(k)f(x)}^{2} k^{4 M - 3} \, dk}^{\frac{1}{2}}
 $$
 is unbounded on $L^{p}$, where $r(k)$ is the operator with kernel
 defined through
$$
r(k)(x,y) := \tau(x) (\nabla \phi_{i}(x) + \nabla u_{i}(x,0))
(\Delta_{\R^{n_{i}} \xx \mathcal{M}_{i}} +
k^{2})^{-M}(x_{i}^{\circ},y) \phi_{i}(y).
$$
 This final claim will be proved in the below lemma.

 \begin{lem} 
 \label{lem:R2} 
 The operator $R$ is unbounded on $L^{p}$ for any $p \geq \min_{i}
 n_{i}$.
\end{lem}

\begin{proof}
  From the definition of the kernel $r(k)$, for $f \in
  L^{p}$ non-negative and $x \in \mathcal{M}$ we have
  \begin{align*}\begin{split}  
      Rf(x) &= \br{\int^{1}_{0}\abs{r(k)f(x)}^{2} k^{4 M - 3} \,
        dk}^{\frac{1}{2}} \\
      &= a(x) \br{\int^{1}_{0}
        \abs{\int_{\mathcal{M}}(\Delta_{\R^{n_{i}} \xx
            \mathcal{M}_{i}} + k^{2})^{-M}(x_{i}^{\circ},y)
          \phi_{i}(y) f(y) \, dy}^{2} k^{4 M - 3} \, dk}^{\frac{1}{2}},
    \end{split}\end{align*}
  where $a(x) := \tau(x) \abs{\nabla \phi_{i}(x) + \nabla
    u_{i}(x,0)}$. From Corollary \ref{cor:LowerEndResolvent},
  \begin{align*}\begin{split}  
 Rf(x) &\gtrsim a(x) \br{\int^{1}_{0} \br{\int_{\mathcal{M}}
     d(x_{i}^{\circ},y)^{2 M - n_{i}} \exp(-c k d(x_{i}^{\circ},y))
     \phi_{i}(y) f(y) \, dy}^{2} k^{4 M - 3} \, dk}^{\frac{1}{2}} \\
 &= a(x) \br{\int^{1}_{0} \br{\int_{\mathcal{M}}
     d(x_{i}^{\circ},y)^{2 M - n_{i}} \exp(-c k d(x_{i}^{\circ},y))
     \phi_{i}(y) f(y) \, dy} \right. \\ & \qquad \qquad \qquad
 \left. \cdot \br{\int_{\mathcal{M}}
     d(x_{i}^{\circ},z)^{2 M - n_{i}} \exp(-c k d(x_{i}^{\circ},z))
     \phi_{i}(z) f(z) \, dz} k^{4M - 3} \, dk}^{\frac{1}{2}}.
\end{split}\end{align*}
On applying Tonelli's theorem,
\begin{equation}
  \label{eqtn:Tonelli}
 Rf(x) \gtrsim a(x) \br{\int_{\mathcal{M}} \int_{\mathcal{M}}
   d(x_{i}^{\circ},y)^{2M - n_{i}} d(x_{i}^{\circ}, z)^{2M - n_{i}} \eta(y,z)
 \phi_{i}(y) \phi_{i}(z) f(y) f(z) \, dy
   \, dz}^{\frac{1}{2}},
 \end{equation}
 with
 $$
 \eta(y,z) :=    \int^{1}_{0} k^{4 M - 3} \exp (- c k (d(x_{i}^{\circ},y) +
 d(x_{i}^{\circ},z))) \, dk.
 $$
For $y, \, z \in \mathrm{supp} \, \phi_{i}$ the quantity $d := 
(d(x_{i}^{\circ},y) + d(x_{i}^{\circ},z))$ is bounded from
below. Therefore, after a change of variables it can be said that
\begin{align*}\begin{split}  
 \eta(y,z) &= \frac{1}{d^{4 M - 2}} \int^{d}_{0} k^{4 M - 3} \exp
 \br{- c k} \, dk \\
&\geq \frac{1}{d^{4M - 2}} \int^{D}_{0} k^{4 M - 3} \exp \br{- c k} \,
dk \\
&\gtrsim \frac{1}{d^{4 M - 2}},
 \end{split}\end{align*}
where $D > 0$ is some constant independent of $y$ and $z$.
Applying this to \eqref{eqtn:Tonelli},
$$
Rf(x) \gtrsim a(x) \br{\int_{\mathcal{M}} \int_{\mathcal{M}}
  \frac{d(x_{i}^{\circ},y)^{2M - n_{i}} d(x_{i}^{\circ},z)^{2M -
      n_{i}}}{(d(x_{i}^{\circ},y) + d(x_{i}^{\circ},z))^{4 M - 2}}
 \phi_{i}(y) \phi_{i}(z) f(y) f(z) \, dy
   \, dz}^{\frac{1}{2}}.
 $$
 Due to symmetry, we can then say that
 \begin{align*}\begin{split}  
 R f(x) &\gtrsim a(x) \br{\int_{\mathcal{M}} \int_{\mathcal{M}} \frac{d(x_{i}^{\circ},y)^{2M - n_{i}} d(x_{i}^{\circ},z)^{2M -
      n_{i}}}{(d(x_{i}^{\circ},y) + d(x_{i}^{\circ},z))^{4 M - 2}}
  \mathbbm{1}_{d(x_{i}^{\circ},y) \leq d(x_{i}^{\circ},z)} \phi_{i}(y)
  \phi_{i}(z) f(y) f(z) \, dy \, dz}^{\frac{1}{2}} \\
&\geq a(x) \br{\int_{\mathcal{M}} d(x_{i}^{\circ},y)^{2 M - n_{i}} \brs{
  \int_{\mathcal{M}} d(x_{i}^{\circ},z)^{2 - 2 M -n_{i}}
  \mathbbm{1}_{d(x_{i}^{\circ},y) \leq d(x_{i}^{\circ},z)} \phi_{i}(z)
  f(z) \, dz}
  \phi_{i}(y)  f(y) \, dy}^{\frac{1}{2}}.
\end{split}\end{align*}
Set $f_{\varepsilon}(y) := d(x_{i}^{\circ},y)^{-\frac{n_{i}}{p}(1 +
  \varepsilon)} \phi_{i}(y)$ for $\varepsilon > 0$. It is obvious
that $f_{\varepsilon} \in L^{p}$. In the argument to follow, as is
frequently the case with asymptotic arguments, we will
need to keep careful track of the dependence of the constants on $\varepsilon$.  On applying the previous estimate
for $R f$ to the function $f_{\varepsilon}$,
\begin{equation}
  \label{eqtn:feps}
 R f_{\varepsilon}(x) \gtrsim a(x) \br{\int_{\mathcal{M}}
   d(x_{i}^{\circ},y)^{2M - n_{i}(1 + \frac{(1 + \varepsilon)}{p})}
   \phi_{i}(y) \int_{\mathcal{M}} d(x_{i}^{\circ},z)^{2 - 2 M - n_{i}(1 +
     \frac{(1 + \varepsilon)}{p})} \mathbbm{1}_{d(x_{i}^{\circ},y)
     \leq d(x_{i}^{\circ},z)} \phi_{i}(z) \, dz \, dy}^{\frac{1}{2}}.
 \end{equation}
 Evidently
 \begin{align*}\begin{split}  
 \int_{\mathcal{M}} d(x_{i}^{\circ},z)^{2 - 2 M - n_{i}(1 +
     \frac{(1 + \varepsilon)}{p})} \mathbbm{1}_{d(x_{i}^{\circ},y)
     \leq d(x_{i}^{\circ},z)} \phi_{i}(z) \, dz &\gtrsim
   \int^{\infty}_{d(x_{i}^{\circ},y)} r^{1 - 2 M - \frac{n_{i}(1 +
       \varepsilon)}{p}} \, dr \\
   &= \frac{1}{\frac{n_{i}(1 + \varepsilon)}{p} + 2 M - 2}
   d(x_{i}^{\circ},y)^{2 - 2 M - \frac{n_{i}(1 + \varepsilon)}{p}} \\
   &\gtrsim d(x_{i}^{\circ},y)^{2 - 2 M - \frac{n_{i}(1 + \varepsilon)}{p}},
 \end{split}\end{align*}
where the implicit constant in the final line is independent of
$\varepsilon$ since for $\varepsilon$ small $(1 + \varepsilon)^{-1}
\gtrsim 1$. On applying this to \eqref{eqtn:feps}, 
\begin{align*}\begin{split}  
  R f_{\varepsilon}(x) &\gtrsim a(x) \br{\int_{\mathcal{M}}
   d(x_{i}^{\circ},y)^{2 - n_{i}(1 + \frac{2(1 + \varepsilon)}{p})}
   \phi_{i}(y)  \, dy}^{\frac{1}{2}} \\
 &\simeq a(x) \br{\int_{D}^{\infty} r^{1 - 2 n_{i} \frac{(1 +
       \varepsilon)}{p}} \, dr}^{\frac{1}{2}},
\end{split}\end{align*}
for some constant $D > 0$ determined by the distance of $\mathrm{supp}
\, \phi_{i}$ to $x_{i}^{\circ}$. This will blow up everywhere on the non-zero
support of $a$ when $p > (1 + \varepsilon)n_{i}$, thereby forcing
the norm $\norm{R f_{\varepsilon}}_{p}$ to become infinite. Since $\varepsilon$ was arbitrary, this implies
that $R $ is an unbounded operator for any $p > n_{i}$.

\vspace*{0.1in}

For the case $p = n_{i}$, we must carefully consider the asymptotics
of $R f_{\varepsilon}$ as $\varepsilon \rightarrow 0$. From the
previous estimate,
\begin{align}\begin{split}
    \label{eqtn:feps2}
 R f_{\varepsilon}(x) &\gtrsim a(x) \br{\int^{\infty}_{D} r^{1 - 2(1
     + \varepsilon)} \, dr} \\
 &\gtrsim a(x) \frac{1}{\varepsilon} D^{-2 \varepsilon}.
\end{split}\end{align}
Also,
\begin{align}\begin{split}
    \label{eqtn:feps3}
\norm{f_{\varepsilon}}_{p} &= \br{\int_{\mathcal{M}}
  d(x_{i}^{\circ},y)^{-n_{i}(1 + \varepsilon)} \phi_{i}(y) \,
  dy}^{\frac{1}{p}} \\
&\lesssim \br{\int^{\infty}_{D} r^{-1 - n_{i} \varepsilon} \,
  dr}^{\frac{1}{p}} \\
&\simeq \frac{1}{\varepsilon^{\frac{1}{p}}} D^{-\varepsilon}.
\end{split}\end{align}
Suppose that $\norm{R f}_{n_{i}} \lesssim \norm{f}_{n_{i}}$. Then,
from \eqref{eqtn:feps2} and \eqref{eqtn:feps3}, we would have
$$
\frac{1}{\varepsilon} D^{-2 \varepsilon} \lesssim
\frac{1}{\varepsilon^{\frac{1}{p}}} D^{- \varepsilon}
\quad \Leftrightarrow \quad \varepsilon \lesssim \varepsilon^{p} D^{\varepsilon
p},
$$
which is clearly not true asymptotically as $\varepsilon \rightarrow 0$. A contradiction has been reached and we can safely conclude that
$R$ is also unbounded for $p = n_{i}$.
 \end{proof}

 \section{The Horizontal Square Function}
 \label{sec:Laplacian}

 In this section, Theorem \ref{thm:Main2} will be proved. In an
 analogous manner to the argument from Section \ref{sec:Low} for the
 operator $S$, $s$ can be expressed as
 \begin{align*}\begin{split}  
 sf(x) &= \br{\int^{\infty}_{0} \abs{t^{2} \Delta (I +
     t^{2}\Delta)^{-M}f(x)}^{2} \, \frac{dt}{t}}^{\frac{1}{2}} \\
 &\br{\int^{\infty}_{0} \abs{\Delta \br{\frac{1}{t^{2}} +
       \Delta}^{-M}f(x)}^{2} t^{3 - 4M} \, dt}^{\frac{1}{2}} \\
 &= \br{\int^{\infty}_{0} \abs{\Delta (k^{2} + \Delta)^{-M}f(x)}^{2}
   k^{4M - 5} \, dk}^{\frac{1}{2}}.
 \end{split}\end{align*}
This can be controlled from above by the sum of the high and low
energy components,
$$
s_{<}f(x) := \br{\int^{1}_{0} \abs{\Delta (k^{2} + \Delta)^{-M}f(x)}^{2}
   k^{4M - 5} \, dk}^{\frac{1}{2}}
 $$
 and
 $$
s_{>}f(x) := \br{\int^{\infty}_{1} \abs{\Delta (k^{2} + \Delta)^{-M}f(x)}^{2}
   k^{4M - 5} \, dk}^{\frac{1}{2}}.
 $$
The proof of the high energy part $s_{>}$ is quite similar to our discussion 
of the vertical square function. In fact it is simpler argument because the 
kernel of the operator $\Delta (k^{2} + \Delta)^{-M}$ is symmetric.
	Hence it suffices to prove only \eqref{xx} and \eqref{yy} 
	automatically follows. Thus one does not have to use the notion of Hodge 
	Laplacian this time. Otherwise the proof is unchanged and we skip it. Here we discuss only details of the low energy component. 
 

 \subsection{Low Energy}

 As a first step to proving the boundedness of the low energy square
 function, observe that
 $$
 s_{<}f(x) = \br{\int^{1}_{0} \abs{\brs{(k^{2} + \Delta)^{-(M - 1)} -
     k^{2}(k^{2} + \Delta)^{-M}}f(x)}^{2} k^{4M - 5} \, dk}^{\frac{1}{2}}.
 $$
 From the formula
 $$
 (k^{2} + \Delta)^{-j} = \sum_{i = 1}^{4} H_{i}^{(j)}(k)
 $$
 we then obtain
 $$
s_{<}f(x) \lesssim \sum_{i = 1}^{4} s_{<}^{i}f(x),
$$
where
$$
s^{i}_{<}f(x) := \br{\int^{\infty}_{0} \abs{\brs{H^{(M - 1)}_{i}(k) -
      k^{2}H^{(M)}_{i}(k)}f(x)}^{2} k^{4M - 5} \, dk}^{\frac{1}{2}}.
$$
The $L^{p}$-boundedness and weak-type $(1,1)$ property of each of these operators will be proved separately.

\subsubsection{The Operator $s_{<}^{1}$}

Notice that
\begin{align*}\begin{split}  
 H^{(M - 1)}_{1}(k)(x,y) - k^{2} H_{1}^{(M)}(k)(x,y) &= \sum_{i =
   1}^{l} (\Delta_{\R^{n_{i}} \xx \mathcal{M}_{i}} + k^{2})^{-(M - 1)}(x,y)
 \phi_{i}(x) \phi_{i}(y) \\ & \qquad \qquad - k^{2}(\Delta_{\R^{n_{i}} \xx
   \mathcal{M}_{i}} + k^{2})^{-M}(x,y) \phi_{i}(x) \phi_{i}(y) \\
 &= \sum_{i = 1}^{l} \Delta_{\R^{n_{i}} \xx \mathcal{M}_{i}}
 \br{\Delta_{\R^{n_{i}} \xx \mathcal{M}_{i}} + k^{2}}^{-M}(x,y)
 \phi_{i}(x) \phi_{i}(y).
\end{split}\end{align*}
As this kernel only consists of finitely many terms, in order to prove
that $s_{<}^{1}$ is bounded on $L^{p}$ and weak-type $(1,1)$ it is sufficient to prove for
any $1 \leq i \leq l$ that the operator
$$
s_{<}^{1,i} := \br{\int^{1}_{0} \abs{\int_{\mathcal{M}}
    \Delta_{\R^{n_{i}} \xx \mathcal{M}_{i}} (\Delta_{\R^{n_{i}} \xx
      \mathcal{M}_{i}} + k^{2})^{-M}(x,y) \phi_{i}(x) \phi_{i}(y) f(y)
  \, dy}^{2} k^{4 M - 5} \, dk}^{\frac{1}{2}}
$$
is bounded on $L^{p}$ and weak-type $(1,1)$. Recall from classical theory that the operator
\begin{align*}\begin{split}
&\br{\int^{\infty}_{0} \abs{\Delta_{\R^{n_{i}} \xx \mathcal{M}_{i}}
    (\Delta_{\R^{n_{i}} \xx \mathcal{M}_{i}} + k^{2})^{-M}g(x)}^{2}
  k^{4M - 5} \, dk}^{\frac{1}{2}} \\ & \hspace*{2in} = \br{\int^{\infty}_{0} \abs{t^{2}
    \Delta_{\R^{n_{i}} \xx \mathcal{M}_{i}}(t^{2} \Delta_{\R^{n_{i}}
      \xx \mathcal{M}_{i}} + 1)^{-M} g(x)}^{2} \, \frac{dt}{t}}^{\frac{1}{2}} 
 \end{split}\end{align*}
is bounded on $L^{p}(\R^{n_{i}} \xx \mathcal{M}_{i})$ for all $p \in
(1,\infty)$ and weak-type $(1,1)$. Therefore,
\begin{align*}\begin{split}  
 \norm{s_{<}^{1,i}f}_{p} &= \norm{\phi_{i} \cdot \br{\int^{1}_{0}
     \abs{\Delta_{\R^{n_{i}} \xx \mathcal{M}_{i}} (\Delta_{\R^{n_{i}}
         \xx \mathcal{M}_{i}} + k^{2})^{-M} (\phi_{i} \cdot f)}^{2}
     k^{4 M - 5} \, dk}^{\frac{1}{2}}}_{p} \\
 &\lesssim \norm{  \br{\int^{\infty}_{0}
     \abs{\Delta_{\R^{n_{i}} \xx \mathcal{M}_{i}} (\Delta_{\R^{n_{i}}
         \xx \mathcal{M}_{i}} + k^{2})^{-M} (\phi_{i} \cdot f)}^{2}
     k^{4 M - 5} \, dk}^{\frac{1}{2}}}_{p}  \\
 &\lesssim \norm{f}_{p}
\end{split}\end{align*}
for any $p \in (1,\infty)$. Weak-type $(1,1)$ bounds follow in the same
manner.

\subsubsection{The Operator $s_{<}^{2}$}

The operator family $\lb H_{2}^{(M - 1)}(k) - k^{2} H_{2}^{(M)}(k)
\rb_{k \in (0,1)}$ constitutes a family of pseudodifferential
operators of order $2 - 2 M$. In a similar manner to Section
\ref{subsec:S2}, standard pseudodifferential operator theory can be
applied to yield the $L^{p}$-boundedness of $s_{<}^{2}$ for any $p \in
[1,\infty]$.

\subsubsection{The Operator $s_{<}^{3}$}

From an application of the triangle inequality followed by Minkowski's integral
inequality, the operator $s^{3}_{<}$ can be estimated from above by
$$
s^{3}_{<}f(x) \lesssim W_{1}f(x) + W_{2}f(x) = \sum_{i = 1, 2}
\int_{\mathcal{M}} w_{i}(x,y) \abs{f(y)} \, dy,
$$
where $W_{1}$ and $W_{2}$ are operators with the respective kernels
$$
w_{1}(x,y) := \br{\int^{1}_{0} \abs{H_{3}^{(M -
      1)}(k)(x,y)}^{2} k^{4M - 5} \, dk}^{\frac{1}{2}} 
$$
and
$$
w_{2}(x,y) :=  \br{\int^{1}_{0}
  \abs{H^{(M)}_{3}(k)(x,y)}^{2} k^{4 M - 1} \, dk}^{\frac{1}{2}}.
$$
The $L^{p}$-boundedness and weak-type $(1,1)$ property of the operators $W_{1}$ and $W_{2}$ will be
proved separately.

\begin{lem} 
 \label{lem:W1} 
 The operator $W_{1}$ is bounded on $L^{p}$ for all $p
 \in (1,\infty)$ and weak-type $(1,1)$.
 \end{lem}

\begin{proof}  
 As in Section
\ref{subsec:S3}, in order to obtain the boundedness of $W_{1}$ on
$L^{p}$ for $1 < p < \infty$ it suffices to demonstrate that the kernel
satisfies the estimates \eqref{eqtn:KernelCondition1},
\eqref{eqtn:KernelCondition2}, \eqref{eqtn:KernelCondition3} and
\eqref{eqtn:KernelCondition4}.

For the condition \eqref{eqtn:KernelCondition1}, observe that
Proposition \ref{prop:H3} implies $\abs{H_{3}^{(M - 1)}(k)(x,y)} \lesssim k^{4 -
2 M}$ for all $x, \, y \in \mathcal{M}$. This estimate, when applied
to the definition of $w_{1}(x,y)$, gives
\begin{equation}
  \label{eqtn:w11}
w_{1}(x,y) \lesssim 1 \quad \forall \ x, \, y \in \mathcal{M},
\end{equation}
which immediately implies the validity of
\eqref{eqtn:KernelCondition1}

 For
\eqref{eqtn:KernelCondition2}, Proposition \ref{prop:H3} implies that
for $x \in \R^{n_{i}} \xx \mathcal{M}_{i} \setminus K_{i}$ and $y \in K$,
$$
\abs{H_{3}^{(M - 1)}(k)(x,y)} \lesssim  k^{4 - 2 M} \langle
d(x_{i}^{\circ}, x) \rangle^{2
  - n_{i}} \exp (- c k d(x_{i}^{\circ},x)).
$$
Therefore,
\begin{align}\begin{split}
    \label{eqtn:w12}
 w_{1}(x,y) &= \br{\int^{1}_{0} \abs{H_{3}^{(M - 1)}(k)(x,y)}^{2} k^{4
     M - 5} \, dk}^{\frac{1}{2}} \\
 &\lesssim \langle d(x_{i}^{\circ},x) \rangle^{2 - n_{i}}
 \br{\int^{1}_{0} k^{3} \exp(- 2 c k d(x_{i}^{\circ},x)) \,
   dk}^{\frac{1}{2}} \\
 &\lesssim \langle d(x_{i}^{\circ},x) \rangle^{- n_{i}}.
\end{split}\end{align}
This will clearly imply \eqref{eqtn:KernelCondition2} when $p > 1$.

\vspace*{0.1in}

Let us next consider \eqref{eqtn:KernelCondition3}. For $x \in K$ and
$y \in \R^{n_{j}} \xx \mathcal{M}_{j} \setminus K_{j}$, Proposition \ref{prop:H3}
implies that
$$
\abs{H_{3}(k)(x,y)} \lesssim k^{4 - 2M} \langle d(x_{j}^{\circ}, y)
\rangle^{2 - n_{j}} \exp(- c k d(x_{j}^{\circ},y)).
$$
Therefore,
\begin{align*}\begin{split}  
 w_{1}(x,y) &= \br{\int^{1}_{0} \abs{H_{3}^{(M - 1)}(k)(x,y)}^{2} k^{4
     M - 5} \, dk}^{\frac{1}{2}} \\
 &\lesssim \langle d(x_{j}^{\circ},y) \rangle^{2 - n_{j}} \br{\int^{1}_{0}
   k^{3} \exp(- 2 c k d(x_{j}^{\circ},y)) \, dk}^{\frac{1}{2}} \\
 &\lesssim \langle  d(x_{j}^{\circ},y) \rangle^{-n_{j}}.
\end{split}\end{align*}
It is easy to see that this implies \eqref{eqtn:KernelCondition3}
provided $p < \infty$

\vspace*{0.1in}

Finally, it remains to validate the estimate
\eqref{eqtn:KernelCondition4}. Proposition \ref{prop:H3} once again
implies that for $x \in \R^{n_{i}} \xx \mathcal{M}_{i} \setminus K_{i}$ and $y \in
\R^{n_{j}} \xx \mathcal{M}_{j} \setminus K_{j}$,
\begin{align}\begin{split}
    \label{eqtn:w13}
 w_{1}(x,y) &\lesssim \langle d(x_{i}^{\circ},x) \rangle^{2 - n_{i}}
\langle d(x_{j}^{\circ},y) \rangle^{2 - n_{j}} \br{\int^{1}_{0} k^{3} \exp(- 2 c k
   (d(x_{i}^{\circ},x) + d(x_{j}^{\circ},y))) \, dk}^{\frac{1}{2}} \\
 &\lesssim \langle d(x_{i}^{\circ},x) \rangle^{2 - n_{i}}
\langle d(x_{j}^{\circ},y) \rangle^{2 - n_{j}}
\br{\frac{1}{(d(x_{i}^{\circ},x) +
    d(x_{j}^{\circ},y))^{4}}}^{\frac{1}{2}} \\
&\leq \min \br{\langle d(x_{i}^{\circ},x) \rangle^{- n_{i}} \langle
  d(x_{j}^{\circ},y) \rangle^{2 - n_{j}}, \langle d(x_{i}^{\circ},x)
    \rangle^{2 - n_{i}} \langle d(x_{j}^{\circ},y) \rangle^{-n_{j}}}.
\end{split}\end{align}
Applying this to \eqref{eqtn:KernelCondition4},
\begin{align*}\begin{split}  
 \int_{\R^{n_{i}} \xx \mathcal{M}_{i} \setminus K_{i}} &\br{\int_{\R^{n_{j}} \xx
     \mathcal{M}_{j} \setminus K_{j}} w_{1}(x,y)^{p'} \, dy}^{\frac{p}{p'}} \, dx \\
 &\lesssim \int_{\R^{n_{i}} \xx \mathcal{M}_{i} \setminus K_{i}}
 \br{\int_{D^{1}_{j}(x)} \langle d(x_{i}^{\circ},x) \rangle^{(2 -
     n_{i})p'} \langle d(x_{j}^{\circ},y) \rangle^{-n_{j} p'} \,
   dy}^{\frac{p}{p'}} \, dx \\
 &\qquad + \int_{\R^{n_{i}} \xx \mathcal{M}_{i} \setminus K_{i}}
 \br{\int_{D^{2}_{j}(x)} \langle d(x_{i}^{\circ},x) \rangle^{-n_{i} p'} \langle d(x_{j}^{\circ},y) \rangle^{(2-n_{j}) p'} \,
   dy}^{\frac{p}{p'}} \, dx \\
 &=: J_{1} + J_{2},
\end{split}\end{align*}
where $D_{j}^{1}(x)$ and $D_{j}^{2}(x)$ are as defined in Section \ref{subsec:S3}.
For the first term,
$$
J_{1} = \int_{\R^{n_{i}} \xx \mathcal{M}_{i} \setminus K_{i}}
\frac{1}{\langle d(x_{i}^{\circ},x) \rangle^{(n_{i} - 2)p}}
\br{\int_{D^{1}_{j}(x)} \frac{dy}{\langle d(x_{j}^{\circ},y)
    \rangle^{n_{j}p'}}}^{\frac{p}{p'}} \, dx.
$$
For the interior integral, in an analogous manner to \eqref{eqtn:D1j},
$$
\int_{D^{1}_{j}(x)} \frac{dy}{\langle d(x_{j}^{\circ},y)
  \rangle^{n_{j}p'}} \lesssim \int_{\R^{n_{j}}} \frac{dy_{1}}{(|y_{1}
  - x_{j,1}^{\circ}| + d(x_{i}^{\circ},x))^{n_{j} p'}}.
$$
This will be integrable since $p' > 1$, in which case
\begin{align*}\begin{split}  
 \int_{D^{1}_{j}(x)} \frac{dy}{\langle d(x_{j}^{\circ},y)
  \rangle^{n_{j}p'}} &\lesssim \int^{\infty}_{d(x_{i}^{\circ},x)}
\frac{r^{n_{j} - 1}}{r^{n_{j}p'}} \, dr \\
&\simeq \brs{- r^{n_{j}(1 - p')}}^{\infty}_{d(x_{i}^{\circ},x)} \\
&= \langle d(x_{i}^{\circ},x) \rangle^{n_{j}(1 - p')}.
\end{split}\end{align*}
Applying this estimate to $J_{1}$ gives
\begin{align*}\begin{split}  
 J_{1} &\lesssim \int_{\R^{n_{i}} \xx \mathcal{M}_{i} \setminus K_{i}}
 \frac{1}{\langle d(x_{i}^{\circ},x) \rangle^{(n_{i} - 2)p}} \cdot
 \langle d(x_{i}^{\circ},x) \rangle^{-n_{j}} \, dx \\
 &= \int_{\R^{n_{i}} \xx \mathcal{M}_{i} \setminus K_{i}}
 \frac{dx}{\langle d(x_{i}^{\circ},x) \rangle^{(n_{i} - 2)p + n_{j}}}.
\end{split}\end{align*}
This will be finite provided that
$$
(n_{i} - 2)p + n_{j} > n_{i} \quad \Leftrightarrow \quad p >
\frac{n_{i} - n_{j}}{n_{i} - 2}.
$$
Since $\frac{n_{i} - n_{j}}{n_{i} - 2} < \frac{n_{i} - 2}{n_{i} - 2} =
1$, this will be satisfied when $p > 1$. It has therefore been proved
that $J_{1}$ is finite for $1 < p < \infty$.

It remains to consider the term $J_{2}$. Looking back to our proof of
the boundedness of the square function $S$, the term $J_{2}$ is
identical to the term $I_{2}$ from
Section \ref{subsec:S3}. In Section \ref{subsec:S3} it was proved that $I_{2}$
is finite for all $p \in (1,\infty)$. Thus it can be safely concluded
that $W_{1}$ is bounded on $L^{p}$ for all $p \in (1,\infty)$.

Let us now prove that $W_{1}$ is weak-type $(1,1)$. On
combining \eqref{eqtn:w11}, \eqref{eqtn:w12} and \eqref{eqtn:w13}, it
is evident that
$$
\sup_{y \in \mathcal{M}} \abs{w_{1}(x,y)} \lesssim
\left\lbrace \begin{array}{c c} 
 \langle d(x_{i}^{\circ},x) \rangle^{-n_{i}} & for \ x \in \R^{n_{i}}
                                               \xx \mathcal{M}_{i}
                                               \setminus K_{i}, \ 1 \leq i
                                               \leq l, \\
               1 & for \ x \in K.
 \end{array} \right.
$$
This implies that for $f \in C^{\infty}_{c}$,
\begin{align*}\begin{split}  
 \abs{W_{1}f(x)} &\leq \int_{\mathcal{M}} w_{1}(x,y) \abs{f(y)} \, dy
 \\
 &\leq \sup_{y \in \mathcal{M}} \abs{w_{1}(x,y)} \int_{\mathcal{M}}
 \abs{f(y)} \, dy \\
 &\lesssim \left\lbrace \begin{array}{c c} 
 \langle d(x_{i}^{\circ},x) \rangle^{-n_{i}} \norm{f}_{L^{1}} & for \
                                                                x \in
                                                                \R^{n_{i}}
                                                                \xx
                                                                \mathcal{M}_{i}
                                                                \setminus
                                                                K_{i}, \ 1
                                                                \leq i
                                                                \leq l,
                          \\
                          \norm{f}_{L^{1}} & for \ x \in K.
 \end{array} \right.
 \end{split}\end{align*}
This function is clearly in $L^{1,\infty}$ and thus
$W_{1}$ is weak-type $(1,1)$.
\end{proof}

It remains to bound the operator $W_{2}$. This follows trivially once
one notices that $k^{2} H^{(M)}_{3}(k)(x,y)$ will satisfy the same
asymptotic estimates as $H^{(M - 1)}_{3}(k)(x,y)$. Indeed, from
Proposition \ref{prop:H3},
$$
k^{2} H^{(M)}_{3}(k)(x,y) \lesssim k^{4 - 2M} \omega_{2}(x,k) \omega_{2}(y,k).
$$
The argument used for the operator $W_{1}$ is therefore equally
applicable to $W_{2}$, and thus $W_{2}$ is $L^{p}$-bounded for $1 < p
< \infty$ and weak-type $(1,1)$. It can then be concluded that 
$s_{<}^{3}$ is bounded on $L^{p}$ for all $p \in
(1,\infty)$ and weak-type $(1,1)$.

\subsubsection{The Operator $s_{<}^{4}$}

The previous proof of the $L^{p}$-boundedness and weak-type $(1,1)$
property of $s^{3}_{<}$ was
derived entirely from the pointwise estimates for $H^{(M -
  1)}_{3}(k)(x,y)$ and $H^{(M)}_{3}(k)(x,y)$ provided by Proposition \ref{prop:H3}. Since the kernels $H^{(M -
  1)}_{4}(k)(x,y)$ and $H^{(M)}_{4}(k)(x,y)$ satisfy even stronger size estimates, provided by
Proposition \ref{prop:H4}, it follows that $s^{4}_{<}$ must also be bounded on
$L^{p}$ for all $p \in (1,\infty)$ and weak-type $(1,1)$.

\subsection{The Reverse Inequality}
\label{subsec:Reverses}

The proof of the reverse inequality for $s$, $\norm{f}_{p} \lesssim
\norm{s f}_{p}$ for $p \in (1,\infty)$, follows in an essentially
identical manner to our proof of the reverse inequality for $S$ from
Section \ref{sec:Reverse}. The only two notable differences that one
will encounter is that the proof must begin with the resolution of the
identity
$$
f = c_{M} \int^{\infty}_{0} (t \Delta)^{2} (1 + t \Delta)^{-2 M}f \, \frac{dt}{t},
$$
and subsequently the identity
$$
\int_{\mathcal{M}} (t \Delta)^{2} (1 + t \Delta)^{-2 M}f(x) \cdot
\overline{g(x)} \, dx = \int_{\mathcal{M}} t \Delta (1 + t
\Delta)^{-M}f(x) \cdot \overline{t \Delta (1 + t \Delta)g(x)} \, dx
$$
must be utilised.

\noindent
{\bf Acknowledgements:} JB and  AS  were
supported by Australian Research Council  Discovery Grant DP DP160100941.
In addition AS was supported by Australian Research Council  Discovery Grant DP200101065. We would like to thank Pierre Portal for interesting comments and referring us to \cite{hytonenNeervenWeis}.

\end{document}